\documentclass[11pt]{article}

\usepackage[paperwidth=230mm, paperheight=305mm]{geometry}

\usepackage{graphicx}
\usepackage{amssymb}
\usepackage{amsmath}
\usepackage{enumerate}
\usepackage{amsthm}
\usepackage{amsmath}
\usepackage{float}
\usepackage{lineno}
\makeatletter
\renewenvironment{proof}[1][\proofname]{%
   \par\pushQED{\qed}\normalfont%
   \topsep6\p@\@plus6\p@\relax
   \trivlist\item[\hskip\labelsep\bfseries#1\@addpunct{.}]%
   \ignorespaces
}{%
   \popQED\endtrivlist\@endpefalse
}
\makeatother

\newtheorem{theorem}{Theorem}
\newtheorem{proposition}[theorem]{Proposition}
 \numberwithin{theorem}{section}
 \newtheorem{corollary}[theorem]{Corollary}
\newtheorem{lemma}[theorem]{Lemma}

\newtheorem*{remark}{Remark}
\newtheorem*{claim}{Claim}

\numberwithin{equation}{section}

\newtheorem{thmx}{Theorem}

\newcommand{\R}{\mathbb{R}}

\newcommand{\N}{\mathbb{N}}

\newcommand{\cF}{\mathcal F}

\begin{document}

\author{Jieliang Hong\\  \textit{University of British Columbia}}
\title{Renormalization of local times of super-Brownian motion}
%\affil{Department of Mathematics, The University of Brithish Columbia,1984 Mathematics Road, Vancouver, BC V6T1Z3, Canada E-mail address: jlhong@math.ubc.ca}
%\affil[$\dagger$]{Department of Mathematics, Pennsylvania State University,Pittsburgh, Pennsylvania 13593}
%\date{}
\maketitle
\begin{abstract}
    For the local time $L_t^x$ of super-Brownian motion $X$ starting from $\delta_0$, we study its asymptotic behavior as $x\to 0$. In $d=3$, we find a normalization $\psi(x)=((2\pi^2)^{-1} \log (1/|x|))^{1/2}$ such that $(L_t^x-(2\pi|x|)^{-1})/\psi(x)$ converges in distribution to standard normal as $x\to 0$. In $d=2$, we show that $L_t^x-\pi^{-1} \log (1/|x|)$ converges a.s. as $x\to 0$. We also consider general initial conditions and get some renormalization results. The behavior of the local time allows us to derive a second order term in the asymptotic behavior of a related semilinear elliptic equation.
\end{abstract}

\section{Introduction and main results}

\subsection{Introduction}
Super-Brownian motion arises as a scaling limit of critical branching random walks. Let $M_F=M_F(\R^d)$ be the space of finite measures on $(\R^d,\mathfrak{B}(\R^d))$ equipped with the topology of weak convergence of measures, and $(\Omega,\cF,\cF_t,P)$ be a filtered probability space. Let $C([0,\infty),M_F(\R^d))$ denote the space of continuous functions from $[0,\infty)$ to $M_F(\R^d)$ with the compact open topology. A $\mathit{Super}$-$\mathit{Brownian}$ $\mathit{ Motion}$ $X$ starting at $\mu\in M_F(\R^d)$ is a continuous $M_F(\R^d)$-valued strong Markov process defined on $(\Omega,\cF,\cF_t,P)$ with $X_0=\mu$ a.s. We write $X_t(\phi)$ for $\int_{\R^d} \phi(y) X_t(dy)$. It is well known that super-Brownian motion is the solution to the following $\mathit{martingale\ problem}$ (see Perkins (2002), II.5): For any $ \phi \in C_b^2(\R^d)$, 
\begin{equation} \label{e1.0}
X_t(\phi)=X_0(\phi)+M_t(\phi)+\int_0^t X_s(\frac{\Delta}{2}\phi) ds,
\end{equation} where $M_t(\phi)$ is a continuous $\cF_t$-martingale such that $M_0(\phi)=0$ and the quadratic variation of $M(\phi)$ is \[[M(\phi)]_t=\int_0^t X_s(\phi^2) ds.\]
The martingale problem uniquely determines the law $P_{X_0}$ of super-Brownian motion $X$ on $C([0,\infty),M_F(\R^d))$.\\

Local times of superprocesses have been studied by many authors (cf. Sugitani (1989), Barlow, Evans and Perkins (1991), Adler and Lewin (1992), Krone (1993), Merle (2006)). In a recent work, Mytnik and Perkins (2017) obtain the exact Hausdorff dimension of the boundary of super-Brownian motion, defined as the boundary of the set of points where the local time is positive. Now we recall that Sugitani (1989) has proved that for $d\leq 3$, there exists a random function $L_t^x$ such that for any $\phi \in C_b(\R^d)$, 
\[\int_0^t X_s(\phi) ds =\int_{\R^d}  \phi(x) L_t^x dx.\] $L_t^x$ is called the local time of $X$ at point $x\in \R^d$ and time $t\geq 0$, which is jointly lower semi-continuous and is monotone increasing in $t\geq 0$. It also can be defined as \[L_t^x:=\lim_{\varepsilon \to 0} \int_0^t X_s(p_\varepsilon^x) ds,\] where $p_\varepsilon^x(y)=p_\varepsilon(y-x)$ is the transition density of $d$-dimensional Brownian motion. The joint continuity of $L_t^x$ is given in Theorem 3 of Sugitani (1989),  which we now recall.
\begin{thmx}(Sugitani (1989)) \label{t3}
Let $d\leq 3$ and $X_0=\mu \in M_F(\R^d)$. Then there is a version of the local time $L_t^x$ which is jointly
continuous on the set of continuity points of $\mu q_t(x)$, where $q_t(x)=\int_0^t p_s(x)ds $ and $\mu q_t(x)=\int \mu(dy) \int_0^t p_s(y-x) ds$.
\end{thmx}

\begin{remark}
 When $d=1$, $\mu q_t(x)$ is always jointly continuous (see Proposition 3.1 in Sugitani (1989)), so the above theorem implies that there is a version of the local time $L_t^x$ that is always jointly continuous, which is also a result of Konno and Shiga (1988). When $d\geq 4$, we have $\int_0^t X_s(\cdot) ds$ is a.s. a singular measure, $\forall t>0$ and so local times do not exist. See Exercise III.5.1 of Perkins (2002) or Dawson, Iscoe and Perkins (1989) for more discussions.
\end{remark}

It is also natural to consider the case under the canonical measure $\N_{x_0}$. Theorem II.7.3(a) in Perkins (2002) gives the existence of a $\sigma$-finite measure $\N_{x_0}$ on $C([0,\infty),M_F(\R^d))$, and it is defined to be the weak limit of $N P_{\delta_{x_0}}^N (X_\cdot^N \in \cdot)$ as $N \to \infty$, where $X_\cdot^N$ under $P_{\delta_{x_0}}^N$ is the approximating branching particle system starting from a single particle at $x_0$ (see Ch.p. II.3 of Perkins (2002)). In this way, $\N_{x_0}$ describes the contribution of a cluster from a single ancestor at $x_0$ and the super-Brownian motion is then obtained by a Poisson superposition of such clusters. In fact,  we have
\[X_t=\int\nu_t\ \Xi(d\nu), t>0, \text{ has law } P_{X_0},\] where $\Xi$ is a Poisson point process with intensity $\N_{X_0}=\int \N_{x_0}(\cdot) X_0(d x_0)$ (see, e.g., Theorem II.7.3(c) in Perkins (2002)). The existence of the local time $L_t^x$ under $\N_{x_0}$ then follows from this decomposition and the existence under $P_{\delta_{x_0}}$. Therefore the local time $L_t^x$ may be decomposed as \[L_t^x=\int L_t^x(\nu)\ \Xi(d\nu)\overset{d}{=}\sum_{i} L_t^x(\nu_i).\] Perhaps surprisingly $L^x_t$ will be jointly continuous on $\{(t,x): t\geq 0, x\in \R^d\}$, $\N_{x_0}$-a.e..

\subsection{Main results}

\begin{theorem}  \label{p0}
Let $d\leq 3$. Then for all $x_0 \in \R^d$, we have $\N_{x_0}$-a.e. that $L_t^x$ is jointly continuous on $\{(t,x): t\geq 0, x\in \R^d\}$. Moreover,
we have \[ \lim_{t\downarrow 0} \sup_{x} L_t^x=0,  \N_{x_0}\text{-a.e..}\]
\end{theorem}

As is indicated in the remark after Theorem \ref{t3}, $L_t^x$ is jointly continuous for all $t\geq 0$ and $x$ in $d=1$. Now we focus on the case $X_0=\delta_0$ in $d=2$ and $d=3$. The continuity of $\mu q_t(x)$ with $\mu=\delta_0$ fails for $x=0$ and $t\geq0$, while Theorem \ref{p0} tells us that the local time is jointly continuous everywhere under the canonical measure $\N_0$. We are then interested in the asymptotic behavior of local time $L_t^x$, under the law $P_{\delta_0}$, as $x\to 0$ in $d=2$ and $3$. By Lemma 1 in Sugitani (1989), we have for any $X_0\in M_F(\R^d)$ and for any fixed $\varepsilon>0$,
\begin{equation} \label{e3.1}
L_t^x-L_\varepsilon^x \text{\ is jointly continuous on \ } \{(t,x): t\geq \varepsilon, x\in \R^d\}, \ \ P_{X_0}\text{-a.s..} 
\end{equation} 
This also follows by the Markov Property at time $\varepsilon$ and the fact that $\mu=X_\varepsilon$ will satisfy the condition on Theorem \ref{t3} that $\mu q_t(x)$ is jointly continuous for all $t$ and $x$. On the other hand we expect that when $d=2$ or $3$, the singularity in the initial condition leads to the singularity of the local time, leading to our main results below. \\

\noindent $\mathbf{Convention\ on\ Constants.}$ Constants whose value is unimportant and may change from line to line are denoted $C$, while constants whose values will be referred to later and appear initially in say, Theorem $i.j$ are denoted $c_{i.j}$.\\

\noindent $\mathbf{Notations.}$ 
If $M$ is a metric space equipped with a metric $d$, let $(\xi_t)_{t\in T}$ be a collection of $M$-valued random vectors. We denote convergence in probability $P$ by $\xi_t \xrightarrow[]{P} \xi_{t_0}$ as $t\to t_0$ if for any $\varepsilon>0$, we have \[P(d(\xi_t,\xi_{t_0})>\varepsilon) \to 0 \text{\ as  \ } t\to t_0.\] We denote weak convergence, or convergence in distribution, by $\xi_t \xrightarrow[]{d} \xi_{t_0}$ as $t\to t_0$ if for any $\phi \in C_b(M)$, \[E\phi(\xi_t) \to E\phi(\xi_{t_0}) \text{\ as  \ } t\to t_0.\]

\begin{theorem} \label{t1}
Let $c_{\ref{t1}}=1/(2\pi)$ and $\psi(x)=(2c_{\ref{t1}}^2 \log (1/|x|))^{1/2}$, and $X$ be a super-Brownian motion in $d=3$ with initial condition $X_0=\delta_0$. Then for each $0<t\leq \infty$

\begin{equation*}
\Big(X,\frac{L_t^x-c_{\ref{t1}}/|x|}{\psi(x)}\Big) \xrightarrow[]{d} \Big(X,Z\Big) \text{ as } x\to 0,
\end{equation*}
where $Z$ denotes a random variable with standard normal law which is independent of $X$. The weak convergence occurs on the space $(C([0,\infty),M_F(\R^3))\times \R )$.
\end{theorem}

\begin{theorem}\label{t2}
Let $c_{\ref{t2}}=1/\pi$ and $X$ be a super-Brownian motion in $d=2$ with initial condition $X_0=\delta_0$. Then with $P_{\delta_0}$-probability one, 
\[ \lim_{x\to 0} L_t^x-c_{\ref{t2}} \log\frac{1}{|x|}=c_{\ref{t2}}(X_t(g_0)-M_t(g_0)), \ \forall\ 0<t\leq \infty,\] where $g_0(y)=\log |y|$ and both terms on the right-hand side are a.s. finite.
\end{theorem}

\begin{remark}
(a) The independence of $Z$ and $X$ is suggested by \eqref{e3.1} and that $(L_\varepsilon^x-c_{\ref{t1}}/|x|)/\psi(x)$ converges in distribution for all $\varepsilon>0$. We also use the same idea to prove the independence of $X$ and $Z$ in Section \ref{s3.2}.\\
(b)
The re-centering constants taking the forms of $c_{\ref{t1}}/|x|$ in $d=3$ and $c_{\ref{t2}}\log (1/|x|)$ in $d=2$ are both suggested by setting $\phi$ to be these two potential functions in the martingale problem \eqref{e1.0}. The scaling by $\psi(x)$ in Theorem \ref{t1} is necessary since the variance of the local time blows up in $d=3$, but not in $d=2$. It will become clearer in Theorem \ref{t4.1} below for the general initial condition case, where a scaling may or may not be needed for the local time in $d=3$.
\end{remark}

Compared to the a.s. convergence case in Theorem \ref{t2} when $d=2$, we also establish the following a.s. convergence result in $d=3$.
\begin{theorem}\label{p2.3}
Let $X$ be a super-Brownian motion in $d=3$ with initial condition $X_0=\delta_0$. Then for any $\alpha>0$, with $P_{\delta_0}$-probability one,
\[ \lim_{x\to 0} \frac{L_t^x-c_{\ref{t1}}/|x|}{1/|x|^\alpha}=0,\ \forall\ 0<t\leq \infty.\] 
\end{theorem}
\begin{remark}
While Theorem \ref{t1} tells us that the re-centered local time has an Gaussian type oscillation of order $(\log(1/|x|))^{1/2}$ for $x$ near $0$ in $d=3$, the above theorem furthermore implies that with $P_{\delta_0}$-probability one, this oscillation will be killed by any polynomial decay.
\end{remark}

Let the extinction time $\zeta$ of $X$ be defined as $\zeta=\zeta_X=\inf \{t\geq 0: X_t(1)=0\}$. Then we have $L_{\infty}^x=L_\zeta^x$. We know that $\zeta<\infty$ a.s. (see Chp II.5 in Perkins (2002)), and so can use \eqref{e3.1} to see that $\lim_{x\to 0} (L_\infty^x-L_t^x)$ is finite a.s. for all $t>0$. On the other hand, Theorem \ref{t2} and Theorem \ref{p2.3} above imply that $\lim_{x\to 0} L_t^x=\infty$, $\forall 0<t\leq \infty$ with $P_{\delta_0}$-probability one in $d=2$ and $d=3$, and so we can see that the singularity in the initial condition indeed leads to the singularity of the local time after a positive time. Mytnik and Perkins (2017) use the $t=\infty$ case in their work on the dimension of the boundary of super-Brownian motion. In the meantime, it would be interesting to find functions $\bar{\psi_1}$ or $\bar{\psi_2}$ so that with $P_{\delta_0}$-probability one, for all $0<t\leq \infty$, \[\limsup_{x\to 0} \frac{L_t^x-c_{\ref{t1}}/|x|}{\bar{\psi_1}(x)}=1,  \ \text{ or }  \ \liminf_{x\to 0} \frac{L_t^x-c_{\ref{t1}}/|x|}{\bar{\psi_2}(x)}=-1,\] 
and we state it as an open problem.

%The canonical decomposition allows us to write \[L_t^x=\sum_{i=1}^{N} L_t^x (X^i)\] where $N$ is Poisson with intensity $\N_0(\tau^x<t)$ with $\tau^x=\inf \{t\geq 0: X_t \text{ hits } x\}$, and given $N$, $(X^i, i\leq N)$ is i.i.d. with law $\N_0(\cdot |\tau^x<t)$. We now understand the above asymptotic behaviors somehow arises from superimposing the local time Nevertheless we choose to give a more analytic argument as some of the ideas will help us resolve the general initial condition case. 

\subsection{General Initial Conditions}
Now that we have the above results for the case $X_0=\delta_0$, we will then consider the general initial condition case $X_0=\mu\in M_F(\R^d)$.

\subsubsection{The case $d=3$}
The following Tanaka formula is from Theorem 6.1 in Barlow, Evans and Perkins (1991): If $\mu(\phi_x)<\infty$ with $\phi_x(y)=c_{\ref{t1}}/|y-x|$, then
\begin{equation}\label{e1.5}
L_t^x= \mu(\phi_x)+M_t(\phi_x)-X_t(\phi_x), 
\end{equation}
where $M_t(\phi_x)$ is an $\cF_t$-martingale which is defined in terms of the martingale measure associated with super-Brownian motion. In particular, we have $M_0(\phi_x)=0$ and $M_t(\phi_x)$ has quadratic variation
\begin{align}\label{e2.6.1}
[M(\phi_x)]_t=\int_0^t X_s(\phi_x^2) ds.
\end{align}
The condition $\mu(\phi_x)<\infty$ on \eqref{e1.5} suggests that we define the set of ``bad'' points by
\begin{equation}
D=\{x_0\in \R^3: \int \frac{1}{|y-x_0|} \mu(dy)=\infty\}. \label{e1.4}
\end{equation}
We show that $D$ is a Lebesgue null set and in particular $D^c$ is dense in $\R^3$ (see Lemma \ref{l5.5}). Then we can consider the behavior of the local time as $x\to x_0$ for $x\in D^c$ and $x_0\in D$. One can show that $(t_0,x_0)$ is a continuity point of $\mu q_t(x)$ for all $t_0\geq 0$ if and only if $x_0$ is a continuity point of $\int 1/|y-x| \mu(dy)$ (see Appendix B(ii)). So Theorem \ref{t3} asserts joint continuity of $L_t^x$ on $\{(t,x): t\geq 0,\ x\text{ is a continuity point of } \int 1/|y-x|\mu(dy)\}$.  Therefore the following is a partial converse to Sugitani's Theorem \ref{t3} in $d=3$:
\begin{theorem} \label{t4.1}
Let $X$ be a super-Brownian motion in $d=3$ with initial condition $X_0=\mu \in M_F(\R^3)$ and $D$ be defined as \eqref{e1.4}.  Then for any point $x_0\in D$, with $P_\mu$-probability one we have for any $t>0$, $x\mapsto L_t^x$ is discontinuous at $x_0$. Moreover, we have \[\lim_{x\in D^c, x\to x_0} L_t^x=\infty \text{ in probability.}\]
\end{theorem}
Now that the discontinuity of $L_t^x$ is established for points in $D$, we extend Theorem \ref{t1} to such points in part (a) of the
following theorem and show that different asymptotic behavior is possible in part (b).
\begin{theorem} \label{t4}
Let $X$ be a super-Brownian motion in $d=3$ with initial condition $X_0=\mu \in M_F(\R^3)$. Let $x_0 \in D$,
\begin{enumerate}[(a)]
\item If $x_n \in D^c$ satisfies \[\int \log^+ (1/|y-x_n|) \mu(dy) \to \infty \text{\ as\ } x_n\to x_0,\] then for all $0<t\leq \infty$,
\begin{equation}
\frac{L_t^{x_n}-\int c_{\ref{t1}} /|y-x_n|\mu(dy)}{(2c_{\ref{t1}}^2 \int \mu(dy)\log^+ (1/|y-x_n|))^{1/2}} \xrightarrow[]{d} Z \text{ as } x_n \to x_0,
\end{equation}
 where $Z$ is a r.v. with standard normal law in $\R$. 
\item If $x_n \in D^c$ satisfies \[\int \log^+ (1/|y-x_n|) \mu(dy) \to \int \log^+ (1/|y-x_0|) \mu(dy)<\infty \text{\ as\ } x_n\to x_0,\] then for all $0<t\leq \infty$,
\begin{equation}
L_t^{x_n}-\int \frac{c_{\ref{t1}}}{|y-x_n|}\mu(dy)\xrightarrow[]{P_\mu} M_t(\phi_{x_0})-X_t(\phi_{x_0}) \text{ as } x_n \to x_0,
\end{equation}
where the right hand side is $P_\mu$-a.s. finite.
\end{enumerate}

\end{theorem}
\begin{remark}
By using the same arguments in Section \ref{s3.2}, we can get the joint convergence in distribution of $X$ and the renormalized local time in Theorem \ref{t4}(a) towards $(X,Z)$, with $Z$ independent of $X$, exactly as in Theorem \ref{t1}.
\end{remark}

\subsubsection{The case $d=2$}

\begin{theorem} \label{t5}
Let $X$ be a super-Brownian motion in $d=2$ with initial condition $X_0=\mu\in M_F(\R^2)$. Then there is a jointly continuous version of
\begin{equation*}
L_t^x-\int \frac{1}{\pi} \log^+ \frac{1}{|y-x|} \mu(dy) 
\end{equation*}
on $\{(t,x): t>0, x\in \R^2\} $.
\end{theorem}
For any $t>0$, $(q_t(y)-(1/\pi) \log^+ (1/|y|))$ can be extended to be a bounded continuous function on $\R^2$ by \eqref{e9.5} in Appendix C. This shows that the above theorem includes Sugitani's Theorem \ref{t3} for $t>0$, and it also gives a partial converse to Sugitani's Theorem \ref{t3} in $d=2$. The more interesting case is where the potential kernel blows up and Theorem \ref{t5} gives a true renormalization of the local time. It is easy to combine the continuity implicit in Theorem \ref{t5} with that of Theorem \ref{t3} to conclude the following Corollary:

%This shows that the above Theorem \ref{t5} implies Sugitani's Theorem \ref{t3} for the case $t>0$ in $d=2$. From this one can also show that the continuity of $x\mapsto \int \log^+ (1/|y-x|) \mu(dy)$ at $x_0$ is equivalent to the continuity of $\mu q_t(x)$ at $(t_0,x_0)$ for some $t_0\geq 0$ (see Appendix B(ii)). So by using  Sugitani's result Theorem \ref{t3} for the case $t=0$ in $d=2$, we have the the following Corollary, whose proof will appear in the proof of Theorem \ref{t5}, that implies $L_t^x$ is jointly continuous on $\{(t_0,x_0):t_0\geq 0\}$ if $x_0$ is a continuity point of  $\int \log^+ (1/|y-x|) \mu(dy)$. 

\begin{corollary}\label{c1.0}
There is a jointly continuous version of the above renormalized local time on $\{(t,x): t>0, x\in \R^2\} \bigcup \{(0,x): x \text{ is a continuity point of } \int \log^+ (1/|y-x|) \mu(dy)\}$.
\end{corollary}

\subsection{Application to semilinear PDE}\label{s1.4}

Consider the super-Brownian motion with initial condition $\mu \in M_F$ in $d=3$. It has been shown (see Theorem 3.3 of Iscoe (1986) and Lemma 2.1 of Mytnik and Perkins (2017)) that for each $\lambda>0$,
 \begin{equation}\label{e1.6.1}
E_{\mu} \Big(\exp \Big({-\lambda L_\infty^x}\Big) \Big)=\exp \Big(-\int V^\lambda (y-x) \mu(dy) \Big),
\end{equation}
where $V^\lambda (x)$ is the unique solution to
\begin{equation}\label{e1.6}
\frac{\Delta}{2}  V^\lambda (x) =\frac{1}{2}(V^\lambda (x))^2-\lambda \delta_0, \  \  V^\lambda(x) > 0.
\end{equation}
Note that the above equation is interpreted in a distributional sense.\\

Such semilinear singular PDEs have been studied by a number of authors in the 80's: V\'eron (1981), Brezis, Peletier and Terman (1986), Brezis and Oswald (1987). It is known (see p. 187 in [4]) that the unique solution $V^\lambda$ is smooth in $\mathbb{R}^3\backslash {\{0\}}$, and near the origin 
\begin{equation}\label{e1.3}
\frac{V^\lambda(x)}{ \lambda/(2\pi|x|)}  \to 1 \text{ as } x \to 0.
\end{equation}
It's also shown in Remark 1(b) of Brezis and Oswald (1987) that \[|V^\lambda(x)-\lambda  \frac{1}{2\pi} \frac{1}{|x|}| \leq C(|\log|x||+1) \text{ for } x\neq 0. \]
Our aim is to find the exact second order term as $x\to 0$. Let $\mu=\delta_0$ in \eqref{e1.6.1} to see that $E_{\delta_0} (\exp ({-\lambda L_\infty^x}) )=\exp (- V^\lambda (x)  )$. A good intuition from Theorem \ref{t1} that 
\begin{equation}
\frac{L_{\infty}^x-c_{\ref{t1}}/|x|}{(2c_{\ref{t1}}^2 \log (1/|x|))^{1/2}} \xrightarrow[]{d} Z
\end{equation}
implies \[L_{\infty}^x-c_{\ref{t1}}/|x| \approx_{law} (2c_{\ref{t1}}^2 \log (1/|x|))^{1/2}  Z \text{ as } x\approx 0,\] and hence we expect that when $x$ goes to 0,
\[e^{-(V^\lambda (x)-\lambda c_{\ref{t1}}/|x|)}=E_{\delta_0} e^{-\lambda (L_\infty^x-c_{\ref{t1}}/|x|)} \approx E e^{-\lambda(2c_{\ref{t1}}^2 \log (1/|x|))^{1/2}  Z} = e^{\frac{1}{2} \lambda^2 2c_{\ref{t1}}^2 \log (1/|x|)}.\] In Section \ref{s6.2} we will show that this intuition is correct and prove the following:

\begin{theorem} \label{t6}
Let $V^\lambda(x)$ be the solution of the semilinear elliptic equation \eqref{e1.6}. Then \[\frac{V^\lambda(x)-\lambda/(2\pi|x|)}{\lambda^2  \log (1/|x|) / (4\pi^2)} \to -1, \text{ as } x \to 0 \text{ in } \R^3.\] 
\end{theorem}

\noindent $\mathbf{Organization\ of\ the\ paper.}$ Section 2 gives the main ideas of the proofs of the main results, Theorem 1.1, 1.2 and 1.3. In fact we give a complete proof of Theorem 1.1 and present conditional proofs of Theorem 1.2 and 1.3 assuming some intermediate results. The proof of Theorem 1.2 and 1.3 will then be finished in Sections 3 and 4. The cumulants of super-Brownian motion discussed in Section 4 may be of independent interests. Here we establish moment estimates following the strategy of Sugitani. Section 5 contains the proof of Theorem 1.4. Sections 6 and 7 are devoted to the cases under general initial conditions and finally Section 8 is the application to PDE.

\section*{Acknowledgements}
This work was done as part of the author's graduate studies at the University of British Columbia. I would like to thank my supervisor, Professor Edwin Perkins, for telling me about this problem and for many helpful discussions and suggestions throughout this work. I also thank two anonymous referees for their careful reading of the manuscript and for their valuable comments which made the paper more readable.

\section{Proof of the Main Results}

\subsection{Continuity under canonical measure (Theorem \ref{p0})}

\begin{proof}[Proof of Theorem \ref{p0}] 
Fix $\varepsilon, \delta > 0$. Conditioning on $\cF_\varepsilon^X$ and on $X_\varepsilon \neq 0$, by Markov property and Theorem II 7.3(c) of Perkins (2002), our canonical cluster decomposition according to ancestors at time $\varepsilon$ implies 
\begin{equation}\label{e0.1}
X_{t+\varepsilon}=\int \nu_t \ \Xi^\varepsilon (d\nu), \ \forall\ t\geq 0,
\end{equation}
where $\Xi^\varepsilon$ is Poisson point process with intensity $\N_{X_\varepsilon}$. Let $\zeta=\zeta_X=\inf \{t\geq 0: X_t(1)=0\}$ denote the extinction time of $X$. Then for any $t\geq 0$,
\begin{align}\label{e0.3}
L_{t+\varepsilon}^x-L_\varepsilon^x=\int L_t^x(\nu) \ \Xi^\varepsilon (d\nu)\geq \int L_t^x(\nu) \ 1_{\{\zeta>\delta\}} \ \Xi^\varepsilon (d\nu)\overset{d}{=}  \sum_{i=1}^{N_\delta} L_t^x(X^i).
\end{align} 
where $N_\delta$ is a Poisson random variable of parameter $\N_{X_\varepsilon}(\zeta>\delta)=2X_\varepsilon(1)/\delta$, given $N_\delta$, $(x_i: i \leq N_\delta)$ are i.i.d. with law $X_\varepsilon/X_\varepsilon(1)$, and given $N_\delta$ and $(x_i)$ the $X^i$ are i.i.d. with law 
$\N_{x_i}(X \in \cdot | \zeta>\delta)$. Recall \eqref{e3.1} that for any $X_0\in M_F(\R^d)$ and any fixed $\varepsilon>0$, we have
\begin{equation*}
L_t^x-L_\varepsilon^x \text{\ is jointly continuous on \ } \{(t,x): t\geq \varepsilon, x\in \R^d\}, \ \ P_{X_0}\text{-a.s.} .
\end{equation*} 
Together with the compactness of the support of local times (see Corollary III.1.7 of Perkins (2002)), we have
\begin{equation}\label{e0.4}
\lim_{t\downarrow 0} \sup_x L_{t+\varepsilon}^x-L_\varepsilon^x=0, \ \ P_{\delta_0}\text{-a.s.} .
\end{equation}
Since we have $N_\delta=1$ with positive probability, we conclude from \eqref{e0.3} and \eqref{e0.4} that for $X_\varepsilon$-almost all $x_0$,
\begin{equation*}
\lim_{t\downarrow 0} \sup_x L_{t}^x=0, \ \N_{x_0}(\cdot |\zeta>\delta)\text{-a.e.}.
\end{equation*}
Since $\N_{x_0}(\zeta>\delta)<\infty$, the above holds $\N_{x_0}(\cdot, \zeta>\delta)$-a.e.. Let $\delta \downarrow 0$ to conclude that 
\begin{equation}\label{e0.5}
\lim_{t\downarrow 0} \sup_x L_{t}^x=0, \ \N_{x_0}\text{-a.e.}.
\end{equation}

Next use \eqref{e0.1} to see that for $t\geq \delta$,
\begin{align} \label{e0.6}
 L_{t+\varepsilon}^x-L_{\delta+\varepsilon}^x&=\int (L_t^x(\nu)-L_\delta^x(\nu)) \ \Xi^\varepsilon (d\nu) \nonumber \\
 &= \int  (L_t^x(\nu)-L_\delta^x(\nu)) \ 1_{\{\zeta>\delta\}}  \ \Xi^\varepsilon (d\nu)\overset{d}{=}  \sum_{i=1}^{N_\delta} (L_t^x(X^i)-L_\delta^x(X^i)).
\end{align}
The last equality is the same with the one in \eqref{e0.3}. By \eqref{e3.1}  we have \[L_{t+\varepsilon}^x-L_{\delta+\varepsilon}^x  \text{\ is jointly continuous on  } \{(t,x): t\geq \delta, x\in \R^d\}, \ \ P_{X_0}\text{-a.s.} . \] 
Since we have $N_\delta=1$ with positive probability, we conclude from \eqref{e0.6} that 
\begin{equation*}
 L_{t}^x-L_{\delta}^x  \text{\ is jointly continuous on } \{(t,x): t\geq \delta, x\in \R^d\}, \ \N_{x_0}(\cdot |\zeta>\delta)\text{-a.e.}.
 \end{equation*} 
 Since $\N_{x_0}(\zeta>\delta)<\infty$, the above holds $\N_{x_0}(\cdot,\zeta>\delta)$-a.e. and furthermore use $L_{t}^x-L_{\delta}^x =0$ for the case $\zeta\leq \delta$ to conclude
 \begin{equation} \label{e0.7}
 L_{t}^x-L_{\delta}^x  \text{\ is jointly continuous on } \{(t,x): t\geq \delta, x\in \R^d\}, \ \N_{x_0}\text{-a.e.}.
 \end{equation} 
 
 Now we are ready to finish the proof. By \eqref{e0.5}, we can choose $\omega$ outside a null set $N_1$ such that $\sup_x L_{t}^x \to 0$ as $t\downarrow 0$. By \eqref{e0.7}, we can choose $\omega$ outside a null set $N_2(\delta)$ such that $L_{t}^x-L_{\delta}^x  \text{\ is jointly continuous on } \{(t,x): t\geq \delta, x\in \R^d\}$. Now take $\delta=1/n$ and $N= \cup_{n=1}^\infty N_2(1/n) \cup N_1$ to see that for $\omega \in N^c$, if we fix any $\epsilon>0$, then for any $t>0$, we can find some $n\geq 1$ such that $\sup_x L_{1/n}^x<\epsilon$ and $1/n<t$. Note we have \[L_t^x-L_{t'}^{x'}=[(L_{t}^x-L_{1/n}^x)-(L_{t'}^{x'}-L_{1/n}^{x'})]+L_{1/n}^x-L_{1/n}^{x' }.\] Then use the joint continuity of $L_{t}^x-L_{1/n}^x$ on $\{(t,x): t\geq 1/n, x\in \R^d\}$ to see that there is some $\gamma=\gamma(\epsilon)>0$ such that $|(L_{t}^x-L_{1/n}^x)-(L_{t'}^{x'}-L_{1/n}^{x'})|<\epsilon$ if $|(t',x')- (t,x)|<\gamma$, and so conclude \[|L_t^x-L_{t'}^{x'}|\leq 3\epsilon, \text{ if } |(t',x')- (t,x)|<\gamma.\] Hence $(t,x) \mapsto L_t^x$ is jointly continuous on $\{(t,x): t>0, x\in \R^d\}$, $\N_{x_0}$-a.e.. The $t=0$ case follows immediately from $\eqref{e0.5}$.
\end{proof}

\subsection{Weak renormalization of the local times in $d=3$ (Theorem \ref{t1})}\label{s2.2}
Recall the Tanaka formula \eqref{e1.5} for the case $\mu=\delta_0$. Then for $x\neq 0$ we have 
\begin{equation} \label{e2.1}
\frac{L_t^x-c_{\ref{t1}}/|x|}{(2c_{\ref{t1}}^2 \log (1/|x|))^{1/2} }=\frac{M_t(\phi_x)}{(2c_{\ref{t1}}^2 \log (1/|x|))^{1/2} }
-\frac{X_t(\phi_x)}{(2c_{\ref{t1}}^2 \log (1/|x|))^{1/2} }. 
\end{equation}
For the second term on the right-hand side of \eqref{e2.1}, recall a result of concentration of mass from Theorem III.3.4 in Perkins (2002).
\begin{lemma}\label{l2.1}
Let $d=2$ or $3$. Then there is some constant $c_{\ref{l2.1}}(d)>0$ such that for all $X_0 \in M_F(\R^d)$, $P_{X_0}$-a.s. we have
 \[\forall \delta>0,\ \exists \  r_0(\delta,\omega)>0 \text{\ so that\ } 
\sup_{y \in \R^d, t\geq \delta} X_t(B(y,r)) \leq c_{\ref{l2.1}}(d) \bar{\psi}(r) \  \forall r\in(0,r_0), \] where $\bar{\psi}(r)=r^2(\log^+(1/r))^{4-d}$.
\end{lemma}
For the case $d=3$ with $X_0=\delta_0$, we use Lemma \ref{l2.1} to see that with $P_{\delta_0}$-probability one, there exist some $r_0(t,\omega)\in (0,1]$ and some constant $C>0$ such that 
\begin{align*} 
&\int_{|y-x|<r_0} \frac{1}{|y-x|}  X_t(dy) \leq \sum_{n=0}^\infty \frac{2^{n+1}}{r_0}  \int 1{(\frac{r_0}{2^{n+1}} \leq |y-x|<\frac{r_0}{2^n})} X_t(dy) \\
\leq& \sum_{n=0}^\infty  \frac{2^{n+1}}{r_0}  \sup_{x \in \R^d} X_t(B(x,\frac{r_0}{2^n}))\leq \sum_{n=0}^\infty  \frac{2^{n+1}}{r_0} \cdot c_{\ref{l2.1}}(3) (\frac{r_0}{2^n})^2 \log^+(\frac{2^n}{r_0})  \leq C.
\end{align*}
and hence 
\begin{align} \label{e2.2}
\int \frac{1}{|y-x|} X_t(dy) \leq \frac{1}{r_0} X_t(1)+\int_{|y-x|<r_0} \frac{1}{|y-x|}  X_t(dy)\leq \frac{1}{r_0} X_t(1)+C.
\end{align}
Therefore
\begin{equation}
\frac{X_t(\phi_x)}{(2c_{\ref{t1}}^2 \log (1/|x|))^{1/2} }\xrightarrow[]{\text{ a.s. }} 0, \text{ as } x\to 0.\label{e2.3}
\end{equation} 

\begin{lemma} \label{l2.2}
For any $0<t<\infty$,
\begin{equation*}
\frac{M_t(\phi_x)}{(2c_{\ref{t1}}^2 \log (1/|x|))^{1/2}} \xrightarrow[]{d} Z \text{ as } x\to 0,
\end{equation*}
where $Z$ is standard normal on the line.
\end{lemma} 

\noindent With the above lemma, we are ready to turn to the
\begin{proof}[Proof of Theorem \ref{t1}]

 For $0<t<\infty$, by \eqref{e2.1}, \eqref{e2.3} and the above lemma, we may apply Theorem 25.4 in Billingsley (1995) to get 
 \begin{align}\label{efinal}
     \frac{L_t^{x}-c_{\ref{t1}}/|x|}{(2c_{\ref{t1}}^2 \log (1/|x|))^{1/2}} \xrightarrow[]{d} Z \text{ as } x\to 0.
 \end{align}

For $t=\infty$, recall the extinction time $\zeta=\zeta_X=\inf \{t\geq 0: X_t(1)=0\}$ of $X$, we have $L_{\infty}^x=L_\zeta^x$ and $0<\zeta<\infty$ a.s.. Fix $\varepsilon>0$. If $\zeta\leq \varepsilon$, then $L_{\infty}^x-L_\varepsilon^x=0$ for all $x$. If $\zeta>\varepsilon$, then it follows that
\[
\lim_{x\to 0} L_\zeta^x-L_\varepsilon^x=L_\zeta^0-L_\varepsilon^0<\infty,\quad P_{\delta_0}\text{-a.s.}
\]
by \eqref{e3.1} with $t=\zeta$. So we conclude that  
 \begin{align}\label{efinal1}
\lim_{x\to 0}\frac{L_\zeta^x-L_\varepsilon^x}{(2c_{\ref{t1}}^2 \log (1/|x|))^{1/2}}=0,\quad P_{\delta_0}\text{-a.s.}.
 \end{align}
Now using \eqref{efinal} with $t=\varepsilon$ and \eqref{efinal1}, we apply Theorem 25.4 in Billingsley (1995) to get
\[\frac{L_\infty^{x}-c_{\ref{t1}}/|x|}{(2c_{\ref{t1}}^2 \log (1/|x|))^{1/2}}=\frac{L_\zeta^{x}-L_\varepsilon^{x}}{(2c_{\ref{t1}}^2 \log (1/|x|))^{1/2}}+\frac{L_\varepsilon^{x}-c_{\ref{t1}}/|x|}{(2c_{\ref{t1}}^2 \log (1/|x|))^{1/2}} \xrightarrow[]{d} Z \text{ as } x\to 0.\]
The proof of the joint convergence in distribution of $X$ and the renormalized local time above towards $(X,Z)$, with $Z$ independent of $X$, will be given in Section \ref{s3.2}.
\end{proof}

In order to prove Lemma \ref{l2.2}, we observe that $[M(\phi_x)]_t/(2c_{\ref{t1}}^2 \log (1/|x|))$ is the quadratic variation of martingale $M_t(\phi_x)/(2c_{\ref{t1}}^2 \log (1/|x|) )^{1/2}$. By using the Dubins-Schwarz theorem (see Revuz and Yor (1994), Theorem V1.6 and V1.7), with an enlargement of the underlying probability space, we can construct some Brownian motion $B^{x} (t)$ in $\R$ depending on $x$ such that
\begin{equation} \label{e2.5.1}
\frac{M_t(\phi_x)}{(2c_{\ref{t1}}^2 \log (1/|x|))^{1/2}}=B^{x} \Big(\frac{[M(\phi_x)]_t}{2c_{\ref{t1}}^2 \log (1/|x|) }\Big).
\end{equation}
In Section \ref{s3.1} we will prove that 
\begin{equation}
\frac{[M(\phi_x)]_t}{2c_{\ref{t1}}^2 \log (1/|x|)} \xrightarrow[]{\text{ a.s.}} 1,  \label{e2.5}
\end{equation}
 and show that 
\[\frac{M_t(\phi_x)}{(2c_{\ref{t1}}^2 \log (1/|x|))^{1/2}}=B^{x} \Big(\frac{[M(\phi_x)]_t}{2c_{\ref{t1}}^2 \log (1/|x|)}\Big) \xrightarrow[]{\text{ d}} Z.  \]
In order to prove \eqref{e2.5}, we recall from \eqref{e2.6.1} that 
\begin{align*}
[M(\phi_x)]_t=\int_0^t ds \int \frac{c_{\ref{t1}}^2}{|y-x|^2}X_s(dy),
\end{align*}
and the key observation is that in $d=3$,
\begin{equation}
\Delta_y \log|y-x| =\frac{1}{|y-x|^2} \text{ \ for \ } y\neq x. \label{e2.6}
\end{equation}
\noindent $\mathbf{Notation.}$ Throughout the paper, we define \[g_x(y):=\log|y-x| \text{ \ for \ } y\in \R^d\backslash\{x\}. \]
Then the martingale problem \eqref{e1.0} suggests the following:
\begin{proposition} \label{p1}
Let $d=3$ and $x\neq 0$ in $\R^3$. Then we have $P_{\delta_0}$-a.s. that 
\begin{equation*}
\frac{1}{2} \int_0^t \int \frac{1}{|y-x|^2} X_s(dy) ds= X_t(g_x)-\delta_0(g_x)- M_t(g_x), \ \forall t\geq 0,
\end{equation*}
where $X_t(g_x)$ is continuous in $t$ and $M_t(g_x)$ is a continuous $L^2$ martingale.
\end{proposition}

\noindent The proof of Proposition \ref{p1} is involved and hence is deferred to Section \ref{s3.3}.

\subsection{Strong renormalization of the local times in $d=2$ (Theorem \ref{t2})}\label{s2.3}

We notice that in $d=2$, $\Delta g_x= 2\pi \delta_x$ holds in a distributional sense. Then again the martingale problem \eqref{e1.0} will suggest that
\begin{equation}\label{e2.8}
X_t(g_x)=\delta_0(g_x)+M_t(g_x)+\pi L_t^x. 
\end{equation}
We will show that this intuition is correct and the proof indeed is very similar to that of Proposition \ref{p1}. In Section \ref{s3.3},  the proofs of Proposition \ref{p1} and the following one will be given.
\begin{proposition}(Tanaka formula for d=2) \label{p2}
Let $d=2$ and $x\neq 0$ in $\R^2$. Then we have $P_{\delta_0}$-a.s. that
\begin{equation}\label{e2.7}
L_t^x-\frac{1}{\pi} \log \frac{1}{|x|}=\frac{1}{\pi} \Big[X_t(g_x)-M_t(g_x)\Big], \ \forall t\geq 0,
\end{equation}
where $X_t(g_x)$ is continuous in $t$ and $M_t(g_x)$ is a continuous $L^2$ martingale.
\end{proposition}

\begin{remark}
Barlow, Evans and Perkins (1991) gives the following Tanaka formula for general initial condition $X_0=\mu$ in $d=2$: If $\int \mu(dy) \log^{+} (1/|y-x|)<\infty$, then
\begin{equation} \label{e2.9}
X_t(g_{\alpha,x})=\mu(g_{\alpha,x})+M_t(g_{\alpha,x})+\alpha \int_0^t X_s(g_{\alpha,x}) ds-L_t^x,\ \forall t\geq 0, P_{\mu}\text{-a.s.,}
\end{equation}
 where 
\begin{equation}
g_{\alpha,x}(y):=\int_0^\infty e^{-\alpha t} p_t(x-y) dt, \text{\  for } \alpha>0, x,y\in \R^2.
\end{equation}
 We can see that $g_{\alpha,x}$ is not well defined for $\alpha=0$ in the case $d=2$ and our result effectively extends this Tanaka formula to the $\alpha=0$ case. This extended Tanaka formula \eqref{e2.7} can be generalized to any compactly supported $\mu \in M_F(\R^2)$ such that $\int \mu(dy) \log^+ (1/|y-x|)<\infty$:
\[L_t^x=\frac{1}{\pi} \Big[X_t(g_x)-M_t(g_x)-\mu(g_x)\Big], \forall t\geq 0, P_{\mu}\text{-a.s.,}\]
The proof is similar to those of Proposition \ref{p1}, Proposition \ref{p2} and Proposition \ref{p5.1}. The idea is to find the appropriate dominating function by using the compact support of $\mu$ to control the $\log^+ (|y-x|)$ part and by using Lemma \ref{l5.3} to control the $\log^+ (1/|y-x|)$ part. We will not give the proof in this paper.
\end{remark}
With Proposition \ref{p2} in hand, Theorem \ref{t2} would follow if we could establish the continuity of $X_t(g_x)$ and $M_t(g_x)$ in $x$ for any fixed $t>0$. Now we consider the two cases $d=2$ and $d=3$ with $X_0=\delta_0$.
\begin{lemma}\label{l2.6}
For any  $u,v \in \R^d\backslash {\{0\}}$,
\begin{equation*}
\Big|\log |u|-\log |v|\Big|  \leq |u-v|^{1/2} (|u|^{-1/2}+|v|^{-1/2}).
\end{equation*}
\end{lemma}
\begin{proof}
Let $0<r_1<r_2$. Then by Cauchy-Schwartz, 
\[ \log r_2-\log r_1=\int_{r_1}^{r_2} x^{-1} dx \leq \Big[\int_{r_1}^{r_2} x^{-2} dx\Big]^{1/2} (r_2-r_1)^{1/2}\leq r_1^{-1/2} (r_2-r_1)^{1/2}.\]
The proof follows by replacing $r_1,r_2$ with $|u|,|v|$ and a triangle inequality.
\end{proof}

\begin{lemma}\label{l2.7}
Let $d=2$ or $3$. Then for any $t>0$, with $P_{\delta_0}$ probability one, $x\mapsto X_t(g_x)$ is continuous for all $x\in \R^d$.
\end{lemma}
\begin{proof}
Fix any $x,x'\in \R^d$. Similar to the derivation of \eqref{e2.2}, we use Lemma \ref{l2.1} to see that with $P_{\delta_0}$-probability one, there is some $r_0(\delta,\omega)\in (0,1]$ and some constant $C(d)>0$ such that  for all $x\in \R^d$, \[\int X_t(dy) \frac{1}{|y-x|^{1/2}} \leq \frac{1}{(r_0)^{1/2}} X_t(1)+C.\]  Then use Lemma \ref{l2.6} to get 
\begin{align*}
|X_t(g_x)-X_t(g_{x'})|&\leq |x-x'|^{1/2} \int  \big(\frac{1}{|y-x|^{1/2}}+\frac{1}{|y-x'|^{1/2}}\big) X_t(dy)\\
&\leq 2|x-x'|^{1/2} \Big(\frac{1}{(r_0)^{1/2}} X_t(1)+C\Big).
\end{align*}
Note that $X_t(1)<\infty$ a.s.. Let $|x-x'|\to 0$ to conclude $|X_t(g_x)-X_t(g_{x'})| \to 0$ a.s..
\end{proof}

\begin{lemma} \label{l2.8}
Let $d=2$ or $3$. Then for any $t>0$, under $P_{\delta_0}$ there exists a version of $M_t(g_x)$ that is continuous in $x\in \R^d$.
\end{lemma}
\noindent For each $n\geq 1$, by Burkholder-Davis-Gundy Inequality, there exists some $C_n>0$ such that 
\begin{equation}
E_{\delta_0}\Big[\big| M_t(g_x)-M_t(g_{x'})\big|^{4n}\Big] \leq C_n 
E_{\delta_0}\Big[\Big( \int_0^t ds \int X_s(dy) (g_x(y)-g_{x'}(y))^2\Big)^{2n}\Big] \label{e2.11}
\end{equation}
for any $x,x'\in \R^d$. By Lemma \ref{l2.6}, we have 
\begin{equation}
(g_x(y)-g_{x'}(y))^2 \leq 2|x-x'| (\frac{1}{|y-x|}+\frac{1}{|y-x'|}).\label{e2.12}
\end{equation}
 Then 
\begin{align}
 E_{\delta_0} \Big[\big| & M_t(g_{x})-M_t(g_{x'})\big|^{4n}\Big]   \leq C_n (2|x-x'|)^{2n} 2^{2n} \nonumber\\
&\times E_{\delta_0}\Big[\big( \int_0^t ds \int X_s(dy) \frac{1}{|y-x|}\big)^{2n}+\big( \int_0^t ds \int X_s(dy) \frac{1}{|y-x'|}\big)^{2n}\Big]. \label{e2.13}
\end{align}
By using moment estimates from Sugitani (1989), we have the following lemma:

\begin{lemma}\label{l2.9}
Let $d=2$ or $3$ and $X_0=\mu \in M_F(\R^d)$. Fix any $t>0$. Then for each $n\geq 1$, there exists some $C=C(t,n,d,\mu(1))>0$ such that for all $x\in \R^d$
\[E_{\mu} \Big[\big( \int_0^t ds \int X_s(dy) \frac{1}{|y-x|}\big)^{2n}\Big] \leq C<\infty.\]
\end{lemma}
\noindent The proof of Lemma \ref{l2.9} will be given in Section \ref{s4.1}. With Lemma \ref{l2.9} in hand, we can proceed to the
\begin{proof}[Proof of Lemma \ref{l2.8}] 
By using Lemma \ref{l2.9} and \eqref{e2.13}, we have \[E_{\delta_0} \Big[\big| M_t(g_x)-M_t(g_{x'})\big|^{4n}\Big] \leq C_n (2|x-x'|)^{2n} 2^{2n}  C=C(t,n,d) |x-x'|^{2n}.\] By taking $n$ large enough we may apply Kolmogorov's continuity criterion to obtain a continuous version of $M_t(g_x)$ in $x$. 
\end{proof}
 \noindent Before proceeding to the proof of Theorem \ref{t2}, we state the following lemma:
\begin{lemma}\label{l0.1}
Let $h_x(t)$ be a non-decreasing function on $\{t>0\}$ for each $x\neq 0$. If $\lim_{x\to 0} h_x(q)=h_0(q)$
for all rational $q>0$, where $h_0(t)$ is continuous on $\{t>0\}$, then $\lim_{x\to 0} h_x(t)=h_0(t)$ holds for all $t>0$.
\end{lemma}
\begin{proof}
This follows by the elementary density argument in Helly's selection theorem.
\end{proof}
\noindent Now we are ready to turn to the
\begin{proof}[Proof of Theorem \ref{t2}]
Proposition \ref{p2}, Lemma \ref{l2.7} and Lemma \ref{l2.8} imply that for any $t>0$, we have $P_{\delta_0}$-a.s. that $L_t^x-(1/\pi) \log (1/|x|) \to M_t(g_0)-X_t(g_0)$ as $x\to 0$.  Let $q_n$ be all the rationals in $\{t>0\}$ and then choose $\omega$ outside a null set $N$ such that for all $n$, we have $L_{q_n}^x-(1/\pi) \log (1/|x|) \to M_{q_n}(g_0)-X_{q_n}(g_0)$ as $x\to 0$. One can check that for any $T>0$, as $\varepsilon \downarrow 0$,
\[\sup_{ t \leq T} |M_t(P_\varepsilon g_0) -M_t(g_0)|\xrightarrow[]{L^2}  0, \text{ and } \sup_{0<t \leq T} |X_t(P_\varepsilon g_0) -X_t(g_0)|\to 0.\] The proof will be given in (ii) and (iii) of Section \ref{s3.3.1}. Therefore $M_t(g_0)$ and $X_t(g_0)$ are continuous on $\{t>0\}$. Note that for each $x\neq 0$, $t\mapsto L_{t}^x-(1/\pi) \log (1/|x|)$ is a non-decreasing function on $\{t>0\}$. So use Lemma \ref{l0.1} to conclude that for all $t>0$, we have $L_t^x-(1/\pi) \log (1/|x|) \to M_t(g_0)-X_t(g_0)$ as $x\to 0$ . The $t=\infty$ case follows since the extinction time $\zeta<\infty$, $P_{\delta_0}$-a.s..  
\end{proof}

\section{Remaining Proof of Renormalization in $d=3$ (Theorem \ref{t1})}

\subsection{Convergence in distribution}\label{s3.1}
In Section \ref{s2.2} we have reduced the proof of the convergence in distribution of the renormalized local time in Theorem \ref{t1} to the proof of Lemma \ref{l2.2}. Now we will finish the
\begin{proof}[Proof of Lemma \ref{l2.2}]
Proposition \ref{p1} and \eqref{e2.6.1} imply
\begin{equation*}
[M(\phi_x)]_t=2   c_{\ref{t1}}^2 \Big(X_t(g_x)-\delta_0(g_x)- M_t(g_x)\Big).
\end{equation*}
Note that $\delta_0(g_x)=-\log (1/|x|)$. Then
\begin{eqnarray*}
 & & \frac{[M(\phi_x)]_t-2c_{\ref{t1}}^2 \log (1/|x|) }{2c_{\ref{t1}}^2 \log (1/|x|)}= \frac{2c_{\ref{t1}}^2 \Big(X_t(g_x)- M_t(g_x)\Big)}{2c_{\ref{t1}}^2 \log (1/|x|)} \xrightarrow[]{a.s.} 0 \text{ as } x \to 0.
\end{eqnarray*} 
The a.s. convergence follows from Lemma \ref{l2.7} and Lemma \ref{l2.8}. Hence we have shown that 
\begin{equation}
\tau_x(t):=\frac{[M(\phi_x)]_t}{2c_{\ref{t1}}^2 \log (1/|x|)} \xrightarrow[]{a.s.} 1, \text{ as } x\to 0.
\end{equation}
Recall from \eqref{e2.5.1} that we can find some Brownian motion $B^{x} (t)$ such that
\[
\frac{M_t(\phi_x)}{(2c_{\ref{t1}}^2 \log (1/|x|))^{1/2}}=B^{x} \Big(\frac{[M(\phi_x)]_t}{2c_{\ref{t1}}^2 \log (1/|x|)}\Big):=B_{\tau_x(t)}^{x}.
\]
Let $h$ be a bounded and uniformly continuous function on $\R$. Let $\varepsilon>0$ and choose $\delta>0$ such that $|h(x)-h(y)|<\varepsilon$ holds for any $x,y \in \R$ with $|x-y|<\delta$. Then
\begin{equation*}
E_{\delta_0} |h(B_{\tau_x(t)}^{x})-h(B_1^{x})| \leq \varepsilon+ 2 \|h\|_{\infty} \cdot P_{\delta_0}(|B_{\tau_x(t)}^{x}-B_1^{x}|>\delta).\\
\end{equation*}
If $\gamma>0$, then
\begin{align*}
   P_{\delta_0}(|B_{\tau_x(t)}^{x}-B_1^{x}|>\delta) \leq& P_{\delta_0}(|B_{\tau_x(t)}^{x}-B_1^{x}|>\delta,|\tau_x(t)-1|<\gamma)+P_{\delta_0}(|\tau_x(t)-1|>\gamma)\\
 \leq& P_{\delta_0}(\sup_{|s-1| \leq \gamma} |B_s^{x}-B_1^{x}|> \delta)+P_{\delta_0}(|\tau_x(t)-1|>\gamma)\\
  =& P(\sup_{|s-1| \leq \gamma} |B_s-B_1|> \delta)+P_{\delta_0}(|\tau_x(t)-1|>\gamma)\\
  <& \ \varepsilon+P_{\delta_0}(|\tau_x(t)-1|>\gamma), \text{ if we pick } \gamma \text{ small enough.}
\end{align*}
Since $\tau_x(t)$ converges a.s. to 1 by (3.1), for $|x|$ small enough, we have $P_{\delta_0}(|\tau_x(t)-1|>\gamma)<\varepsilon$ and hence
 \[E_{\delta_0}|h(B_{\tau_x(t)}^{x})-h(B_1^{x})|\leq \varepsilon+ 2 \|h\|_{\infty} 2 \varepsilon.\] 
 Therefore
\begin{equation*}
\frac{M_t(\phi_{x})}{(2c_{\ref{t1}}^2 \log (1/|x|))^{1/2}}=B_{\tau_x(t)}^{x} \xrightarrow[]{d} Z \text{\ as \ } x\to 0,
\end{equation*}
and the proof is complete.
\end{proof}

\subsection{Independence of $X$ and $Z$}\label{s3.2}
Throughout this section we write $E$ for $E_{\delta_0}$ for simplicity (suppressing the dependence on initial condition $\delta_0$).  Fix $0<t< \infty$ and a sequence $x_n \to 0$. Let $Z_t^{x_n}=(L_t^{x_n}-c_{\ref{t1}}/|x_n|)/(2c_{\ref{t1}}^2\log 1/|x_n|)^{1/2}$. By tightness of each component in $(X,Z^{x_n}_t)$, we clearly have tightness of $(X,Z^{x_n}_t)$ as $x_n\to 0$, so it suffices to show all weak limit points coincide. By taking a subsequence we may assume that $(X,Z^{x_n}_t)$ converges weakly to $(X,Z)$. Let $(X,Z)$ be defined on $(\widetilde{\Omega},\widetilde{\cF}_t,\widetilde{P})$ where $X$ is super-Brownian motion and $Z$ is standard normal under $\widetilde{P}$. For any $0<t_1<t_2<\cdots<t_m$, let $\phi_0: \R\to \R$ and $\psi_i: M_F \to \R$, $1\leq i\leq m$ be bounded continuous.  We have \[\lim_{n\to\infty} E\Big[\psi_1(X_{t_1})\cdots \psi_m(X_{t_m}) \phi_0(Z^{x_n}_t)\Big]=\widetilde{E}\Big[\psi_1(X_{t_1})\cdots \psi_m(X_{t_m}) \phi_0(Z)\Big]\] since we assume that $(X,Z^{x_n}_t)$ converges weakly to $(X,Z)$.\\

 Pick $\varepsilon>0$ such that $\varepsilon<t_1$ and $\varepsilon<t$. Let $n\to \infty$, by \eqref{e3.1} we get 
 \begin{equation}
  Z_t^{x_n}-Z_\varepsilon^{x_n}=\frac{L_t^{x_n}-L_\varepsilon^{x_n}}{(2c_{\ref{t1}}^2 \log (1/|x_n|))^{1/2}} \to 0\text{ a.s.},
 \end{equation}
 and hence $(0,Z_t^{x_n}-Z_\varepsilon^{x_n}) \to (0,0)$ a.s.. By Theorem 25.4 in Billingsley (1999)  \[(X,Z_\varepsilon^{x_n})=(X,Z_t^{x_n})-(0,Z_t^{x_n}-Z_\varepsilon^{x_n})\xrightarrow[]{d} (X,Z). \]
 Therefore  since $Z_\varepsilon^{x_n} \in \cF_\varepsilon^{X}$,
\begin{eqnarray*}
 I &:=& \widetilde{E} \Big[\psi_1(X_{t_1})\cdots \psi_m(X_{t_m})  \cdot \phi_0(Z)\Big]= \lim_{n\to \infty} E \Big[\psi_1(X_{t_1})\cdots \psi_m(X_{t_m})  \cdot \phi_0(Z_\varepsilon^{x_n})\Big]\\
  &=& \lim_{n\to \infty} E \Big[E \Big( \psi_1(X_{t_1})\cdots \psi_m(X_{t_m}) \big| \cF_\varepsilon^X \Big) \cdot \phi_0(Z_\varepsilon^{x_n})\Big]\\
  &=&\lim_{n\to \infty} E \Big[E_{X_\varepsilon} \Big( \psi_1(X_{t_1-\varepsilon})\cdots \psi_m(X_{t_m-\varepsilon}) \Big) \cdot \phi_0(Z_\varepsilon^{x_n})\Big].
\end{eqnarray*}
Define \[F_\varepsilon(\nu):=E_\nu\Big(\psi_1(X_{t_1-\varepsilon})\cdots \psi_m(X_{t_m-\varepsilon})\Big)\] for $\nu\in M_F$. We claim  $F_\varepsilon \in C_b(M_F)$. For $m=1$ we have \[F_\varepsilon(\nu)=E_\nu \Big( \psi_1(X_{t_1-\varepsilon})\Big).\]
By Theorem II.5.1 in Perkins (2002), if $T_t \psi(\nu) \equiv E_\nu \psi(X_t)$, then $T_t: C_b(M_F) \to C_b(M_F)$ so $F_\varepsilon = T_{t_1-\varepsilon} \psi_1 \in C_b(M_F)$ since $\psi_1 \in C_b(M_F)$. For $m=2$, 
\begin{eqnarray*}
F_\varepsilon(\nu)= E_\nu \Big[  \psi_{1}(X_{t_{1}-\varepsilon}) E_\nu \Big(\psi_2(X_{t_2-\varepsilon}) \big| \cF_{t_{1}-\varepsilon}^X \Big)\Big]= E_\nu \Big[  \psi_{1}(X_{t_{1}-\varepsilon}) \bar{P}_{t_2-t_{1}} \psi_2(X_{t_{1}-\varepsilon})  \Big].
\end{eqnarray*}
It is then reduced to the case $m=1$ with $\widetilde \psi_1=\psi_1 \bar{P}_{t_2-t_{1}} \psi_2$. The general case follows by a simple induction in $m$.\\

Therefore by the weak convergence of $(X,Z_\varepsilon^{x_n})$ to $(X,Z)$, we have \[ I=\lim_{n\to \infty} E \Big[F_\varepsilon(X_\varepsilon) \cdot \phi_0(Z_\varepsilon^{x_n})\Big]=\widetilde{E} \Big[F_\varepsilon(X_\varepsilon) \cdot \phi_0(Z)\Big].\]
Blumenthal 0-1 law implies that $\widetilde{\cF}_{0+}^X:=\bigcap_{s>0} \widetilde{\cF}_{s}^X$ is trivial and so the martingale convergence theorem gives us that as $\varepsilon \to 0$, \[F_\varepsilon(X_\varepsilon)=\widetilde{E}  \Big( \psi_1(X_{t_1})\cdots \psi_m(X_{t_m})  \big| \widetilde{\cF}_\varepsilon^X \Big) \xrightarrow[]{L^1 } \widetilde{E}  \Big( \psi_1(X_{t_1})\cdots \psi_m(X_{t_m}) \Big).\] Therefore
\begin{align*}
 & \widetilde{E} \Big[\psi_1(X_{t_1})\cdots \psi_m(X_{t_m})  \cdot \phi_0(Z)\Big]= I= \lim_{\varepsilon \to 0 }\widetilde{E} \Big[ F_\varepsilon(X_\varepsilon) \cdot \phi_0(Z)\Big]\\
 =& \widetilde{E} \Big[\widetilde{E}  \Big( \psi_1(X_{t_1})\cdots \psi_m(X_{t_m}) \Big) \cdot  \phi_0(Z)\Big]= \widetilde{E} \Big( \psi_1(X_{t_1})\cdots \psi_m(X_{t_m}) \Big) \cdot \widetilde{E} \phi_0(Z).
\end{align*}

The above functionals are a determining class on $C([0,\infty),M_F) \times \R$ and so we get weak
convergence of $(X,Z^x_t)$ to $(X,Z)$ where the latter are independent.

\subsection{Proofs of Proposition \ref{p1} and Proposition \ref{p2}}\label{s3.3}

We first consider $d=2$ or $d=3$. The standard mollifier $\eta\in C^\infty(\R^d)$ is defined by 
\begin{equation}
\eta(x):=C(d) \exp\Big(\frac{1}{|x|^2-1}\Big) 1_{|x|<1 },\label{e4.8}
\end{equation}
the constant $C(d)$ selected such that $\int_{\R^d} \eta dx=1$. For any $N\geq 1$, if $\chi_{N}$ is the convolution of $\eta$ and the indicator function of the ball $B(0,N)$, then
\begin{equation}\label{e4.8.1}
\chi_{N} (x)=\int_{\R^d} 1_{\{|x-y|<N\}} \eta (y) dy=\int_{|y|<1} 1_{\{|x-y|<N\}} \eta(y) dy.
\end{equation}
One can check that $\chi_{N}$ is a $C^{\infty}$ function with support in $B(0,N+1)$, and 
\[\chi_{N} (x)=\int_{|y|<1} 1_{\{|x-y|<N\}} \eta (y) dy=\int_{|y|<1} \eta(y) dy=1 \text{\ for \ }|x|<N-1.\]
 Let $x\neq 0$. Recall that $g_x(y)=\log|y-x|$. Let $(P_t)$ be the Markov semigroup of $d$-dimensional Brownian motion, then for any $\varepsilon>0$, $P_\varepsilon g_x$ is a $C^{\infty}$ function and in particular $P_{\varepsilon} g_{x} \cdot \chi_N \in C_b^2(\R^d)$, so the martingale problem \eqref{e1.0} implies that $P_{\delta_0}$-a.s. we have
 \begin{equation}\label{e1.7}
 X_t(P_{\varepsilon} g_{x} \cdot \chi_N)=\delta_0(P_{\varepsilon} g_{x} \cdot \chi_N)+ M_t(P_{\varepsilon} g_{x} \cdot \chi_N)+\int_0^t X_s(\frac{\Delta}{2} (P_{\varepsilon} g_{x} \cdot \chi_N))  ds,  \ \forall t\geq 0,
 \end{equation}
 where $M_t(P_{\varepsilon} g_{x} \cdot \chi_N)$ is a martingale with quadratic variation \[[M(P_{\varepsilon} g_{x} \cdot \chi_N)]_t=\int_0^t X_s\Big((P_{\varepsilon} g_{x} \cdot \chi_N)^2\Big) ds.\] One can check that $P_{\varepsilon} g_{x} \cdot \chi_N \uparrow P_{\varepsilon} g_{x}$ as $N \uparrow \infty$. Then use monotone convergence theorem to see that $\delta_0(P_{\varepsilon} g_{x} \cdot \chi_N) \to \delta_0(P_{\varepsilon} g_{x})$. By the compactness of the support of super-Brownian motion (see Corollary III.1.7 of Perkins (2002)), we have with $P_{\delta_0}$-probability one, for $N(\omega)$ large enough, $\int_0^\infty X_s(B(0,N)^c) ds=0$ and so obtain \[\sup_{t\leq T} \Big|X_t(P_{\varepsilon} g_{x} \cdot \chi_N)- X_t(P_{\varepsilon} g_{x})\Big|\to 0 \text{ as } N \to \infty,\]  and \[\sup_{t\leq T}\Big|\int_0^t X_s(\frac{\Delta}{2} P_{\epsilon} g_{x} \cdot \chi_N)  ds - \int_0^t X_s(\frac{\Delta}{2} P_{\epsilon} g_{x})  ds\Big| \to 0 \text{ as } N \to \infty.\]
 On the other hand, we can use Dominated Convergence Theorem to get
 \[E_{\delta_0}\Big[\Big(\sup_{t\leq T}|M_t(P_{\varepsilon} g_{x} \cdot \chi_N) - M_t(P_{\varepsilon} g_{x})|\Big)^2\Big] \to 0 \text{ as } N \to \infty.\] Take appropriate subsequence $N_k \to \infty$ to conclude with $P_{\delta_0}$-probability one,
  \begin{equation}\label{e3.5}
 X_t(P_{\varepsilon} g_x)=\delta_0(P_{\varepsilon} g_x)+ M_t(P_{\varepsilon} g_x)+\int_0^t X_s(\frac{\Delta}{2} P_{\varepsilon} g_x) ds,  \ \forall t\geq 0,
 \end{equation}
 where $M_t(P_{\varepsilon} g_x)$ is a martingale with quadratic variation \[[M(P_{\varepsilon} g_x)]_t=\int_0^t X_s\Big((P_{\varepsilon} g_x)^2\Big) ds.\] Now we want to let $\varepsilon \downarrow 0$ in  \eqref{e3.5} to establish a.s. convergence.

\subsubsection{Prelimineries}\label{s3.3.1}
\begin{lemma}\label{l3.3}
Let $d\geq 1$, for any $0<\alpha<d$, there exists a constant $C=C(\alpha,d)>0$ such that for any $x\neq 0$ in $\R^d$ and $t>0$, 
\[
\int_{\R^d} p_t(y) \frac{1}{|y-x|^\alpha} dy\leq C \frac{1}{|x|^\alpha}.
\]
\end{lemma}

\begin{proof}
Let $\delta=|x|/2$. Then
\begin{align*}
\int_{\R^d} p_t(y) \frac{1}{|y-x|^\alpha} dy \leq \frac{1}{\delta^\alpha}+  \int_{|y-x|<\delta} p_t(y) \frac{1}{|y-x|^\alpha} dy\leq \frac{1}{\delta^\alpha}+(\frac{1}{2\pi t})^{d/2} e^{-\frac{\delta^2}{2t}} \int_0^\delta \frac{1}{r^\alpha}\ C(d) r^{d-1} dr.
\end{align*}
Note that
\[(\frac{1}{2\pi t})^{d/2} e^{-\frac{\delta^2}{2t}} \leq \sup_{t> 0} (\frac{1}{2\pi t})^{d/2} e^{-\frac{\delta^2}{2t}} =C(d) \frac{1}{\delta^d}.\]
Therefore
\begin{equation*}
\int_{\R^3} p_t(y) \frac{1}{|y-x|^\alpha} dy \leq  \frac{1}{\delta^\alpha}+ C(d) \frac{1}{\delta^d} \cdot \frac{C(d)}{d-\alpha} \delta^{d-\alpha} =C(\alpha,d) \frac{1}{|x|^\alpha}.
\end{equation*}
\end{proof}
\begin{lemma} \label{l4.2}
Let $d>1$. Then 
\begin{equation*}
\int_0^t \int \frac{1}{|y-x|} p_s(y) dy ds  \leq  \frac{2\sqrt{d}}{d-1} \sqrt{t},  \ \ \  \forall x\in \R^d.
\end{equation*}
\end{lemma}
\begin{proof}
Let $B$ be a $d$-dimensional Brownian motion starting at 0 and $\rho_t = |x + B_t|$ be a $d$-dimensional Bessel process starting at $|x|$. Then for a standard Brownian motion $\beta$, we have the sde
\[\rho_t = |x| + \beta_t + ((d-1)/2) \int_0^t 1/\rho_s ds.\]
Take means in the above and use $E(|x + B_t|-|x|)\leq E(|B_t|)\leq \sqrt{d}\sqrt{t}$. The result follows since
$E(1/\rho_s) = E(1/|B_s-x|)$.
\end{proof}
Now we proceed to the convergence of each term in \eqref{e3.5}, except for the last, as $\varepsilon \downarrow 0$. Note that we are in the case $d=2$ or $d=3$. All the constants showing below in (i)-(iii) depend only on $d$.
\begin{enumerate}[(i)]
\item Let $B$ be a $d$-dimensional Brownian motion starting at $0$. Let $\varepsilon \downarrow 0$ to see that
 \begin{align*}
 & \Big|\delta_0(P_\varepsilon g_x)-\delta_0(g_x)\Big|\leq E\Big(\Big|\log |B_\varepsilon-x|-\log |x|\Big|\Big)\leq  E \Big[\frac{|B_\varepsilon|^{1/2}}{|x|^{1/2}}\Big]+E\Big[\frac{|B_\varepsilon|^{1/2}}{|B_\varepsilon-x|^{1/2}}\Big] \\
\leq & |x|^{-1/2} \Big(E|B_\varepsilon|\Big)^{1/2} + \Big(E|B_\varepsilon|\Big)^{1/2}\cdot \Big(C |x|^{-1}\Big)^{1/2} \leq C |x|^{-1/2} \varepsilon^{1/4} \to 0. 
\end{align*}
The second inequality is by Lemma \ref{l2.6} and the third inequality is by Lemma \ref{l3.3}.
\item We know from the proof of (i) that for $y-x\neq 0$, 
\begin{equation}
\int p_\varepsilon(z)  \Big|\log|z-(y-x)| -\log|y-x|\Big| dz \leq C|y-x|^{-1/2} \varepsilon^{1/4} . \label{e3.7}
\end{equation}
By Doob's inequality, for any $T>0$ we have
\begin{align*}
 &E_{\delta_0} \Big[\Big(\sup_{t\leq T} |M_t(P_\varepsilon g_x)-M_t(g_x)|\Big)^2\Big] \leq 4 E_{\delta_0} \Big[ \int_0^T X_s\Big((P_\varepsilon g_x-g_x)^2\Big) ds \Big]\\
\leq& C^2 \varepsilon^{1/2} \cdot \int_0^T  \int p_s(y) \frac{1}{|y-x|} dy  ds \leq CT^{1/2} \varepsilon^{1/2} \to 0,
\end{align*}
the second inequality by \eqref{e3.7} and that $E_{\delta_0} X_t(dy)=p_t(y) dy$ by Lemma 2.2 of Konno and Shiga (1988) and the last inequality by Lemma \ref{l4.2}. Take a subsequence $\varepsilon_n \downarrow 0$ to obtain \[\sup_{t\leq T} |M_t(P_{\varepsilon_n} g_x)-M_t(g_x)| \to 0, \ P_{\delta_0}\text{-a.s..} \]

\item
By \eqref{e3.7}, for any $T>0$ we have
\begin{align}\label{e3.00}
\sup_{t\leq T} \Big|X_t(P_\varepsilon g_x)-X_t(g_x)\Big|\leq  C \varepsilon^{1/4} \sup_{t\leq T} \int |y-x|^{-1/2} X_t(dy).
\end{align}
By Corollary III.1.5 in Perkins (2001), with $P_{\delta_0}$-probability one, there is some $\delta'(\omega) \in (0,1]$ such that for all $0<t<\delta'$, the closed support of $X_t$ is within the region $\{y: |y|<3(t \log (1/t))^{1/2}\}$. Then pick $\delta<\delta'$ small enough such that $3(\delta \log (1/\delta))^{1/2} <|x|/2$ and hence
\begin{equation}\label{e3.01}
\sup_{t\leq \delta} \int |y-x|^{-1/2} X_t(dy) \leq 2^{1/2} |x|^{-1/2} \sup_{t\leq \delta} X_t(1).
\end{equation} 
On the other hand, similar to the derivation of \eqref{e2.2}, we use Lemma \ref{l2.1} to see that with $P_{\delta_0}$-probability one, there is some $r_0(\delta,\omega)\in (0,1]$ and some constant $C(d)>0$ such that
\begin{equation}\label{e3.02}
\sup_{\delta \leq t\leq T} \int |y-x|^{-1/2} X_t(dy) \leq r_0^{-1/2} \sup_{\delta \leq t\leq T} X_t(1)+ C.
\end{equation} 
Therefore by \eqref{e3.01} and \eqref{e3.02}, with $P_{\delta_0}$-probability one we have
\[\sup_{t\leq T} \int |y-x|^{-1/2} X_t(dy) \leq 2^{1/2} |x|^{-1/2} \sup_{t\leq \delta} X_t(1)+r_0^{-1/2} \sup_{\delta \leq t\leq T} X_t(1)+ C,\] and use \eqref{e3.00} to conclude as $\varepsilon \downarrow 0$,  \[\sup_{t\leq T} \Big|X_t(P_\varepsilon g_x)-X_t(g_x)\Big| \to 0,\ P_{\delta_0}\text{-a.s.} .\]
\end{enumerate}

\subsubsection{Proof of Proposition \ref{p1}}
Now we are in the case $d=3$. The tricky part about the last term $\int_0^t X_s\big(\frac{\Delta}{2} P_{\varepsilon} g_x) ds$ in \eqref{e3.5} is that we are taking the Laplacian of convolution. By using the following lemma, we can interchange the Laplacian with the convolution.
\begin{lemma}\label{l3.2}
In $\R^3$, for any $\varepsilon>0$ and $y\neq x$, we have \[\Delta_y P_\varepsilon g_x(y)= \int p_\varepsilon(y-z) \frac{1}{|z-x|^2}dz.\]
\end{lemma}
\begin{proof}
See Appendix A for the proof.
\end{proof}
\noindent Now we will turn to the
\begin{proof}[Proof of Proposition \ref{p1}]
 For any $T>0$ we have
\begin{align*}
& E_{\delta_0} \Big( \sup_{t\leq T} \Big|\int_0^t X_s (\frac{\Delta}{2} P_{\varepsilon} g_x) ds-\frac{1}{2}\int_0^t  \int \frac{1}{|y-x|^2} X_s(dy) ds \Big| \Big)\\
\leq& \frac{1}{2} \ E_{\delta_0} \Big(\int_0^T  \int  \Big| \Delta_y P_{\varepsilon} g_x(y)-\frac{1}{|y-x|^2}\Big| \ X_s(dy) ds\Big) \\
=& \frac{1}{2} \int_0^T \int  \Big|\int p_\varepsilon(y-z) \frac{1}{|z-x|^2}dz-\frac{1}{|y-x|^2}\Big| \ p_s(y) dy ds.
\end{align*}
We use Lemma \ref{l3.2} for the equality above. 
\begin{claim}
\begin{equation}\label{e3.9}
\Big|\int p_\varepsilon(y-z) \frac{1}{|z-x|^2}dz-\frac{1}{|y-x|^2}\Big| \to 0 \text{ as } \varepsilon\to 0\text{ for } y\neq x . 
\end{equation}
\end{claim}
\begin{proof}
For $w=y-x \neq 0$, 
\begin{align*}
 & \Big|\int p_\varepsilon(y-z) \frac{1}{|z-x|^2}dz-\frac{1}{|y-x|^2}\Big| \leq  E\bigg(\Big|\frac{1}{|B_\varepsilon-w|^2}-\frac{1}{|w|^2}\Big|\bigg)\nonumber\\
 =E \Big(&\Big||B_\varepsilon-w|-|w|\Big| \cdot \frac{|B_\varepsilon-w|+|w|}{|B_\varepsilon-w|^2 |w|^2}\Big) \leq E\Big( \frac{|B_\varepsilon|}{|B_\varepsilon-w|^2 |w|}\Big)+E\Big( \frac{|B_\varepsilon|}{|B_\varepsilon-w|\ |w|^2}\Big).
\end{align*}
For the first term above, we use Holder's inequality with $1/p=1/5$ and $1/q=4/5$ to get 
\begin{align*}
 E\Big(|B_\varepsilon|\cdot \frac{1}{|B_\varepsilon-w|^2 |w|}\Big)&\leq \frac{1}{|w|} \cdot \Big(E(|B_\varepsilon|^5)\Big)^{1/5} \cdot \Big(E\big(\frac{1}{|B_\varepsilon-w|^{5/2} }\big)\Big)^{4/5}\\
&\leq \frac{1}{|w|} \Big(E(|B_\varepsilon|^5)\Big)^{1/5} \Big(C |w|^{-5/2}\Big)^{4/5}\to 0 \text{ as } \varepsilon\to 0.
\end{align*}
The last inequality is by Lemma \ref{l3.3}. Similarly the second term converges to $0$.
\end{proof}
By Lemma \ref{l3.3}, for all $\varepsilon>0$ we have \[\Big|\int p_\varepsilon(y-z) \frac{1}{|z-x|^2}dz-\frac{1}{|y-x|^2}\Big| \leq (C+1) \frac{1}{|y-x|^2},\] which is integrable w.r.t $\int_0^T ds  p_s(y) dy$ by Lemma \ref{l3.3}. Then use Dominated Convergence Theorem and \eqref{e3.9} to conclude as $\varepsilon \downarrow 0$, \[\int_0^T ds \int_{\R^3} p_s(y) dy \Big|\int p_\varepsilon(y-z) \frac{1}{|z-x|^2}dz-\frac{1}{|y-x|^2}\Big|  \to 0,\] and hence \[\sup_{t\leq T} \Big(\int_0^t X_s\big(\frac{\Delta}{2} P_{\varepsilon} g_x) ds-\frac{1}{2}\int_0^t  \int \frac{1}{|y-x|^2} X_s(dy) ds\  \Big) \xrightarrow[]{L^1}  0.\]  Take a subsequence $\varepsilon_n \downarrow 0$ to obtain \[\sup_{t\leq T} \Big(\int_0^t X_s\big(\frac{\Delta}{2} P_{\varepsilon_n} g_x) ds-\frac{1}{2}\int_0^t  \int \frac{1}{|y-x|^2} X_s(dy) ds\  \Big) \to 0, \ P_{\delta_0}\text{-a.s..} \]
The proof of Proposition \ref{p1} follows by \eqref{e3.5} in $d=3$ and the a.s. convergence (i)-(iii) already established in Section \ref{s3.3.1} if we take appropriate subsequence $\varepsilon_{n_k} \downarrow 0$.
\end{proof}

\subsubsection{Proof of Proposition \ref{p2}}
Now we are in the case $d=2$. For the last term $\int_0^t X_s\big(\frac{\Delta}{2} P_{\varepsilon} g_x) ds$, compared to the $d=3$ case in Lemma \ref{l3.2}, we have the following result from Theorem 1 in Chp. 2.2 of Evans (2010).
 \begin{lemma}\label{l4.1}
In $\R^2$, for any fixed $\varepsilon>0$, we have \[\frac{\Delta}{2} P_{\varepsilon} g_x(y)=\pi p_\varepsilon^x(y).\]
\end{lemma}
\begin{proof}[Proof of Proposition \ref{p2}]
By taking a subsequence $\varepsilon_n$ goes to $0$, we know from Theorem 6.1 in [1] that for any $T>0$, $\sup_{t\leq T} |\int_0^t X_s(p_{\varepsilon_n}^x) ds-L_t^x|\to 0$, $P_{\delta_0}$-a.s.. The proof of Proposition \ref{p2} follows by \eqref{e3.5} in $d=2$ and the a.s. convergence (i)-(iii) already established in Section \ref{s3.3.1} if we take appropriate subsequences $\varepsilon_{n_k} \downarrow 0$.
\end{proof}

 %for any $\varepsilon>0$, we can find $t_n,1\leq n\leq K$ such that for any $\delta \leq t \leq \delta^{-1}$, there is some $1\leq n\leq K$ such that $t_n\leq t<t_{n+1}$ and $|(M_t(g_0)-X_t(g_0))-(M_{t_n}(g_0)-X_{t_n}(g_0))|<\varepsilon$. On the other hand, since , we have $L_{t_n}^x \leq L_{t}^x \leq L_{t_{n+1}}^x$. Pick $\delta_0>0$ such that for any $0<|x|<\delta$, we have $|(L_{t_n}^x-(1/\pi) \log (1/|x|)) -(M_{t_n}(g_0)-X_{t_n}(g_0))|<\varepsilon$ for all $1\leq n\leq K$. Therefore \[L_t^x-(1/\pi) \log (1/|x|) \leq L_{t_{n+1}}^x-(1/\pi) \log (1/|x|)\leq \]

\section{Cumulants of super-Brownian motion}\label{s4.1}
In Section \ref{s2.3} we have reduced the proof of Theorem \ref{t2} to the proof of Lemma \ref{l2.9}, which will be given at the end of this section. Note that we are in the case $d=2$ or $3$ with initial condition $X_0=\mu$. We know from (3.30) and (3.31) in Sugitani (1989) that for $\phi \geq 0 $ continuous with compact support,
\begin{equation}
E_{\mu} \Big[\exp{ \Big(\theta \int_0^t X_s(\phi) ds- \theta \int_0^t \mu(P_s \phi) ds \Big)} \Big]=\exp{\Big(2\sum_{n=2}^{\infty} (\frac{\theta}{2})^n \mu\big(v_n(t)\big)\Big)} \label{e4.2}
\end{equation}
where $v_n(t), n \geq 2$ are given by 
\begin{equation}
\begin{cases}
 v_1(t,z)=\int_0^t P_s \phi (z) ds,\\
 v_n(t,z)=\sum_{k=1}^{n-1} \int_0^t P_{t-s}\Big(v_k(s) v_{n-k}(s)\Big)(z) ds. \label{e4.3}
\end{cases}
\end{equation}
For $v_n, n\geq 1$ we have the following estimates.
\begin{lemma}\label{l4.3}
Let $d=2$ or $3$. For any nonnegative measurable function $\phi$, let $v_n$ be defined as in \eqref{e4.3}. If there exist some constants $r, \alpha, \beta \geq 0$ such that for all $t\geq 0$ and $z\in \R^{d}$, we have $|v_1(t,z)| \leq r ((t+\alpha)^{1/2}+\beta)$.  Then there exist some positive constants $c_{n}$ such that
\begin{equation}
|v_{n}(t, z)|\leq c_{n} r^n (t+\alpha)^{(n-1)/2} ((t+\alpha)^{1/2}+\beta)^{2n-1} \label{e4.7}
\end{equation}
holds for every $t\geq 0$, $z\in \R^{d}$ and $n\geq 1$ .
\end{lemma}

\begin{proof}
The case $n=1$ follows by letting $c_1=1$, so it suffices to show \eqref{e4.7} in the case $n\geq 2$. Assuming it holds for every $1\leq k\leq n-1$, then 
\begin{align*}
|v_{n}(t,z)| &\leq \sum_{k=1}^{n-1} c_k c_{n-k} r^n \int_0^t \int p_{t-s}(z-y) (s+\alpha)^{(n-2)/2}   ((s+\alpha)^{1/2}+\beta)^{2n-2} dy ds \\
&\leq \sum_{k=1}^{n-1} c_k c_{n-k} r^n (t+\alpha)^{(n-2)/2}  \int_0^t  ((s+\alpha)^{1/2}+\beta)^{2n-2} ds\\
&\leq \sum_{k=1}^{n-1} c_k c_{n-k} r^n (t+\alpha)^{(n-1)/2} ((t+\alpha)^{1/2}+\beta)^{2n-1} .
\end{align*}
Let $c_n=\sum_{k=1}^{n-1} c_kc_{n-k}$ to get \eqref{e4.7}.
\end{proof}

If $\phi(y)=f_x(y):=1/|y-x|$, then $f_x$ is not continuous everywhere and not compactly supported. However, \eqref{e4.2} holds for $\int_0^t X_s(f_x) ds- \int_0^t \mu(P_s f_x)ds$ in a weak sense. If $X$ is a random variable, following Sugitani (1989) we say that 
\begin{equation} \label{e4.4}
E\Big[\exp\big(\theta X\big)\Big]=\exp\Big(\sum_{n=1}^{\infty} a_n \theta^n\Big)
\end{equation}
$\mathbf{holds\ formally}$ if $E|X|^n <\infty$ and \[E(X^n)=\frac{d^n}{d\theta^n} \Bigg(\exp\Big(\sum_{k=1}^n\ a_k \theta^k \Big)\Bigg)\Bigg|_{\theta=0}\] for every $n\geq 1$. Note that if \eqref{e4.4} actually holds, then it holds formally. We will then prove the following lemma:

\begin{lemma}\label{l4.4}
Let $d=2$ or $3$. For all $x\in \R^d$, let $f_x(y)=1/|y-x|$. Then for all $t\geq 0$, we have the following holds formally:
\begin{equation} \label{e4.62}
E_{\mu} \Big[\exp{ \Big(\theta \int_0^t X_s(f_x) ds- \theta \int_0^t \mu(P_s f_x) ds \Big)} \Big]=\exp{\Big(2\sum_{n=2}^{\infty} (\frac{\theta}{2})^n \mu\big(v_n^x(t)\big)\Big)}, 
\end{equation}
 where \[v_1^x(t,z)=\int_0^{t} ds \int p_s(z-y) \frac{1}{|y-x|} dy,\] and for $n\geq 2$ 
\begin{equation}
v_n^x(t,z)=\sum_{k=1}^{n-1} \int_0^t ds \int p_{t-s}(z-y) (v_k^x(s,y) v_{n-k}^x (s,y)) dy.\label{e4.61}
\end{equation}
\end{lemma}

\begin{proof}
For any $0<\varepsilon<1$, let \[f_x^\varepsilon(z):= P_\varepsilon f_x(z)= \int p_\varepsilon (z-y) \frac{1}{|y-x|} dy.\] Then $f_x^\varepsilon \in C_b(\R^d)$. For any $N\geq 1$, recall from \eqref{e4.8.1} that $\chi_N$ is a $C^{\infty}$ function and $\chi_N \uparrow 1$. Then $f_x^\varepsilon \cdot \chi_N$ is continuous with compact support and hence \eqref{e4.2} holds for $\int_0^t X_s(f_x^\varepsilon \cdot \chi_N) ds -\int_0^t \mu(P_s (f_x^\varepsilon \cdot \chi_N)) ds $ and in particular it holds formally. Let $N \to \infty$, then $f_x^\varepsilon \cdot \chi_N \uparrow f_x^\varepsilon$. By monotone convergence theorem, we have formally \eqref{e4.2} for $\int_0^t X_s(f_x^\varepsilon) ds -\int_0^t \mu(P_s f_x^\varepsilon) ds $, that is to say,
\begin{align} \label{e4.81}
E_{\mu} \Big[ \Big( \int_0^t X_s(f_x^\varepsilon) ds -\int_0^t \mu(P_s f_x^\varepsilon) ds  \Big)^n\Big]=  \frac{d^n}{d\theta^n} \bigg(\exp\Big(2 \sum_{k=2}^n (\frac{\theta}{2})^n \mu(v_n^{\varepsilon,x}(t))\Big) \bigg) \bigg|_{\theta=0}
\end{align}
where 
\[v_1^{\varepsilon,x}(t, z) =\int_0^t ds \int p_s(z-y) f_x^\varepsilon(y) dy=\int_\varepsilon^{t+\varepsilon} ds \int p_s(z-y) \frac{1}{|y-x|} dy  ,\] and 
\[v_n^{\varepsilon,x}(t,z) =\sum_{k=1}^{n-1} \int_0^t ds \int p_{t-s}(z-y) (v_k^{\varepsilon,x}(s,y) v_{n-k}^{\varepsilon,x} (s,y)) dy. \] 
By Lemma \ref{l4.2}, for $0<\varepsilon<1$, 
\begin{equation} \label{e4.91}
v_1^{\varepsilon,x}(t, z)\leq \int_0^{t+1} \int p_s(z-y) \frac{1}{|y-x|} dy ds \leq \frac{2d^{1/2} }{d-1} (t+1)^{1/2}:=r (t+1)^{1/2},
\end{equation} 
and
\begin{equation}\label{e4.6}
v_1^x(t,z)=\int_0^t \int \frac{1}{|y-x|} p_s(z-y) dy ds \leq  r t^{1/2}. 
\end{equation}
Then Lemma \ref{l4.3} applies and we have
\begin{equation}\label{e4.9}
v_n^{\varepsilon,x}(t,z) \leq c_n r^n (t+1)^{(3n-2)/2},\ \forall n\geq 1,
\end{equation}
and
\begin{equation}\label{e4.92}
v_n^{x}(t,z) \leq c_n r^n t^{(3n-2)/2},\ \forall n\geq 1. 
\end{equation}
 Therefore by Dominated Convergence Theorem, we have $v_1^{\varepsilon,x}(t, z) \to v_1^x(t,z)$ as $\varepsilon \to 0$. For $n=2$, since for each $0<s<t$, we have $(v_1^{\varepsilon,x}(s,y) )^2 \to (v_1^x(s,y) )^2$ as $\varepsilon \to 0$ and $(v_1^{\varepsilon,x}(s,y))^2 \leq r^2 (t+1)$ by \eqref{e4.91}, which is integrable with respect to $\int_0^t ds \int p_{t-s} (z-y) dy$. Hence by Dominated Convergence Theorem \[v_2^{\varepsilon,x}(t, z) =\int_0^t ds \int p_{t-s}(z-y) (v_1^{\varepsilon,x}(s,y) )^2 dy \to \int_0^t \int p_{t-s}(z-y) (v_1^x(s,y) )^2 dy= v_2^x(t,z).\] 
By a simple induction on $n$ and by using \eqref{e4.9}, for every $n\geq 1$ we have
 \[v_n^{\varepsilon,x}(t, z) \to v_n^x(t,z) \text{\ as \ } \varepsilon \to 0.\] 
Therefore for each $t\geq 0$ and $n\geq 1$, by \eqref{e4.81} and Dominated Convergence Theorem we have
\begin{align} 
& \lim_{\varepsilon \to 0} E_{\mu} \Big[ \Big( \int_0^t X_s(f_x^\varepsilon) ds -\int_0^t \mu(P_s f_x^\varepsilon) ds \Big)^n\Big]\nonumber= \lim_{\varepsilon \to 0} \frac{d^n}{d\theta^n} \bigg(\exp\Big(2 \sum_{k=2}^n (\frac{\theta}{2})^n \mu(v_n^{\varepsilon,x}(t))\Big) \bigg) \bigg|_{\theta=0}\nonumber \\
= & \frac{d^n}{d\theta^n} \bigg(\exp\Big(2 \sum_{k=1}^n (\frac{\theta}{2})^n \mu(v_n^x(t))\Big) \bigg) \bigg|_{\theta=0} <\infty. \label{e4.10}
\end{align}
We know by \eqref{e4.92} that the above term is finite. By \eqref{e4.91}, for all $0<\varepsilon<1$ 
\begin{align}\label{e4.101}
\int_0^t \mu(P_s f_x^\varepsilon) ds=\mu(v_1^{\varepsilon,x}(t)) \leq r(t+1)^{1/2} \mu(1).
\end{align}
Then \eqref{e4.10} implies \[\lim_{\varepsilon \to 0} E_{\mu} \Big[ \Big(\int_0^t X_s(f_x^\varepsilon) ds \Big)^n\Big] <\infty.\] By Fatou's lemma
\begin{align} \label{e4.102}
E_{\mu} \Big[ \Big( \int_0^t X_s(f_x) ds \Big)^n\Big]\leq \liminf_{\varepsilon \to 0} E_{\mu} \Big[\Big( \int_0^t X_s(f_x^\varepsilon) ds \Big)^n\Big]<\infty. 
\end{align}
For all $0<\varepsilon<1$, by Lemma \ref{l3.3} we have \[f_x^\varepsilon(y)=\int p_\varepsilon(y-z) \frac{1}{|z-x|} dz \leq C \frac{1}{|y-x|}=Cf_x(y) \text{\ for \ } y\neq x ,\] so by \eqref{e4.101} 
\begin{align*}
\Big| \Big( \int_0^t X_s(f_x^\varepsilon) ds -\int_0^t \mu(P_s f_x^\varepsilon) ds\Big)^n\Big| \leq 2^n \Big(C \int_0^t X_s( f_x) ds\Big)^n +2^n (r (t+1)^{1/2} \mu(1))^n, 
\end{align*} where the right-hand side is integrable w.r.t. $E_\mu$ by \eqref{e4.102}. Therefore Dominated Convergence Theorem implies
\begin{align*}
E_{\mu} \Big[ \Big( \int_0^t X_s(f_x) ds -\int_0^t \mu(P_s f_x) ds \Big)^n\Big]=& \lim_{\varepsilon \to 0} E_{\mu} \Big[ \Big( \int_0^t X_s(f_x^\varepsilon) ds -\int_0^t \mu(P_s f_x^\varepsilon) ds \Big)^n\Big]\\
= & \frac{d^n}{d\theta^n} \bigg(\exp\Big(2 \sum_{k=1}^n (\frac{\theta}{2})^n \mu(v_n^x(t))\Big) \bigg) \bigg|_{\theta=0} \text{\ (by \eqref{e4.10}). }
\end{align*}
Then \eqref{e4.62} holds formally and the proof is complete. 
\end{proof}

The following lemma is from Lemma 3.1 in Sugitani (1989).
\begin{lemma}\label{l4.5}
Let $X$ be a random variable such that \eqref{e4.4} holds formally.
\begin{enumerate}[(i)]
 \item If for some integer $N$ there exists $r,b>0$ such that \[|a_n|\leq br^n, \text{ for } 1\leq n \leq 2N,\] then there exists $C=C(b,N)>0$ such that \[E(X^{2N})\leq C r^{2N}.\]
 \item If $\sum_{n=1}^{\infty} a_n \theta_0^n$ converges for some $\theta_0>0$, then \[E\Big[\exp(|\theta X|)\Big]<\infty \text { \ for }  |\theta|<\theta_0.\]
 \end{enumerate}
\end{lemma}

\noindent Now we will finish the 
\begin{proof}[Proof of Lemma \ref{l2.9}]
We will show that the assumptions in Lemma \ref{l4.5}(i) hold for the case $X=\int_0^t X_s(f_x) ds -\int_0^t\mu(P_s f_x) ds$. 
By \eqref{e4.92} we have  \[2 (\frac{1}{2})^n \mu (v_n^x(t)) \leq c_n \frac{1}{2^{n-1}} t^{(3n-2)/2} r^n\mu(1):=b_n r^n \mu(1).\]
Pick $N\geq 1$. Let $b=\max_{1\leq n\leq 2N} b_n$. Then \[\Big|2 (\frac{1}{2})^n \mu (v_n^x(t))\Big|\leq b \mu(1) r^n, \text{ for } 1\leq n \leq 2N. \] 
So by Lemma \ref{l4.5}(i), there exists some $C=C(b\mu(1),N)=C(t,N,\mu(1))>0$ such that 
 \begin{align*}
 E_{\mu}\Big[ \Big(\int_0^t X_s(f_x) ds -\int_0^t \mu(P_s f_x) ds \Big)^{2N} \Big] \leq   C r^{2N}.
 \end{align*}
By \eqref{e4.6}, we have
\[\int_0^t \mu(P_s f_x) ds=\mu(v_1^{x}(t)) \leq rt^{1/2} \mu(1)\]
and hence
\begin{align*}
E_{\mu} \Big[ \Big(\int_0^t X_s(f_x) ds\Big)^{2N} \Big] \leq   2^{2N} C r^{2N}+2^{2N}(rt^{1/2}\mu(1))^{2N}=C(t,N,d,\mu(1)). 
\end{align*}
So the proof is complete.
\end{proof}

\section{Rate of Convergence in $d=3$ (Theorem \ref{p2.3})}

This section completes the proof of Theorem \ref{p2.3}. Before proceeding to the proof , we state the following lemma:
\begin{lemma}\label{l5.4}
Let $d\geq 3$ .Then for any $x\neq 0$ in $\R^d$ and $t\geq 0$ , \[\int_0^t \int p_s(y) \frac{1}{|y-x|^2} dy ds  \leq \frac{2}{d-2} \Big(\log^+ \frac{1}{|x|}+1+\sqrt{d}\sqrt{t} \Big).\]
\end{lemma}

\begin{proof}
For $\varepsilon>0$ and $x\neq 0$ in $\R^d$, let $h_{\varepsilon,x} (y)=\log(|y-x|^2+\varepsilon)$. Then \[\nabla h_{\varepsilon,x} (y)=\frac{2(y-x)}{|y-x|^2+\varepsilon}\] and \[\Delta h_{\varepsilon,x} (y)=\frac{(2d-4)|y-x|^2+2d \varepsilon}{(|y-x|^2+\varepsilon)^{2}}.\]
Let $B_t$ be a $d$-dimensional Brownian motion starting at 0. By Ito's Lemma,
\[\log(|B_t-x|^2+\varepsilon)=\log(|x|^2+\varepsilon)+\int_0^t \frac{2(B_s-x)}{|B_s-x|^2+\varepsilon} \cdot dB_s+  \int_0^t \frac{(d-2)|B_s-x|^2+d \varepsilon}{(|B_s-x|^2+\varepsilon)^{2}} ds.\]
Let $H_s=\frac{2(B_s-x)}{|B_s-x|^2+\varepsilon}$, then 
\[M_t^\varepsilon:=\int_0^t H_s \cdot dB_s  \text{\ is a continuous local martingale.}\] 
Note that
\[E[(M^\varepsilon_t)^2]\leq E \int_0^t \frac{4|B_s-x|^2}{(|B_s-x|^2+\varepsilon)^2}ds \leq 4\varepsilon^{-2} E \int_0^t |B_s-x|^2ds<\infty.\] 
Then $M^\varepsilon$ is an $L^2$ martingale. Now take means to see that
\begin{align}\label{e5.01}
E \log(|B_t-x|^2+\varepsilon)=\log(|x|^2+\varepsilon)+  \int_0^t E \ \frac{(d-2)|B_s-x|^2+d \varepsilon}{(|B_s-x|^2+\varepsilon)^{2}} ds.
\end{align}
By Fatou's Lemma,
\begin{align*}
 & \int_0^t E \ \frac{d-2 }{|B_s-x|^2} ds\leq  \liminf_{\varepsilon \to 0} \int_0^t  \  E \ \frac{(d-2)|B_s-x|^2+d \varepsilon}{(|B_s-x|^2+\varepsilon)^{2}} ds\\
=& \liminf_{\varepsilon \to 0} \Big[E \log(|B_t-x|^2+\varepsilon)-\log(|x|^2+\varepsilon)\Big]= E \log(|B_t-x|^2)-\log(|x|^2).
\end{align*}
The first equality is by \eqref{e5.01} and the last equality follows from $0\leq \log(|x|^2+\varepsilon) -\log(|x|^2)\leq 2 \sqrt{\varepsilon}/|x|$. If $|x|>1$, then \[ E \log(|B_t-x|)-\log |x|  \leq E \log (|B_t|+|x|)-\log |x| \leq E \frac{|B_t|}{|x|} \leq E|B_t|.\]
If $|x|<1$, then \[E \log(|B_t-x|)-\log |x|  \leq E|B_t-x|+ \log^+ \frac{1}{|x|}\leq E|B_t|+1+ \log^+ \frac{1}{|x|}.\] 
Therefore
\begin{align*}
\int_0^t E \ \frac{d-2}{|B_s-x|^2} ds \leq E \log(|B_t-x|^2)-\log(|x|^2) \leq 2 \Big( E|B_t|+1+ \log^+ \frac{1}{|x|} \Big),
\end{align*}
and the result follows since $E|B_t| \leq \sqrt{d}\sqrt{t}$.
\end{proof}
\noindent Now we will turn to the
\begin{proof}[Proof of Theorem \ref{p2.3}]
We claim it suffices to show that for any $0<\alpha<1$ and any fixed $t>0$, we have $|x|^\alpha (L_t^x-c_{\ref{t1}}/|x|) \to 0$ as $x\to 0$, $P_{\delta_0}$-a.s.. To see this, note that \eqref{e3.1} implies for any $t\geq \delta$, we have $L_t^x-L_\delta^x \to L_t^0-L_\delta^0$ as $x\to 0$. Therefore by choosing $\omega$ outside a null set $N(\delta)$, we have 
\[\lim_{x\to 0} |x|^\alpha (L_t^x-c_{\ref{t1}}/|x|) =\lim_{x\to 0} |x|^\alpha (L_\delta^x-c_{\ref{t1}}/|x|)+\lim_{x\to 0} |x|^\alpha (L_t^x-L_\delta^x)=0, \ \forall t\geq \delta,\]
and this implies \[\lim_{x\to 0} |x|^\alpha (L_t^x-c_{\ref{t1}}/|x|)=0, \ \forall t>0, P_{\delta_0}\text{-a.s..}\]
The $t=\infty$ case follows since the extinction time $\zeta<\infty$, $P_{\delta_0}$-a.s.. Fix any $t>0$. Recall the Tanaka formula \eqref{e1.5} that
\begin{equation}
L_t^x-c_{\ref{t1}}/|x|=M_t(\phi_x)-X_t(\phi_x),
\end{equation}
where $\phi_x(y)=c_{\ref{t1}}/|y-x|$. By using \eqref{e2.2}, we have $P_{\delta_0}$-a.s. that \[|x|^{\alpha} X_t(\phi_x) \leq c_{\ref{t1}} |x|^{\alpha} (r_0^{-1} X_t(1)+C) \to 0 \text{\ as  } x\to 0.\] Therefore the proof of Theorem \ref{p2.3} follows if we show that for any $t>0$,  \[|x|^\alpha M_t(\phi_{x}) \to 0 \text{\ as  } x\to 0, \ P_{\delta_0}\text{-a.s..} \] 
By taking a subsequence, e.g. $\{x_n=(1/2^n,0,0)\}$, that goes to $0$, we have 
\[P_{\delta_0}\Big(\big||x_n|^\alpha M_t(\phi_{x_n})\big|>\frac{1}{2^{n\alpha/2}}\Big) \leq \frac{C(t)+n\alpha \log 2}{2^{n\alpha}}\]
by using Lemma \ref{l5.4}. Hence 
\[|x_n|^\alpha M_t(\phi_{x_n}) \to 0 \text{\  as  } n\to \infty, \ P_{\delta_0}\text{-a.s..}  \]  by Borel-Cantelli Lemma. So it suffices to show that there is a jointly continuous version of $|x|^\alpha M_t(\phi_{x})$ on $B(0,1)=\{x\in \R^3 :|x|<1\}$. \\

Fix any $x,x' \in B(0,1)$. Without loss of generality we may assume $|x|\leq |x'|$ and $|x'|>0$. Then 
\begin{align}\label{e9.4}
&E_{\delta_0}\Big((|x|^\alpha M_t(\phi_x)-|x'|^\alpha M_t(\phi_{x'}))^2\Big)=c_{\ref{t1}}^2 E_{\delta_0}\Big( \int_0^t \Big(\frac{|x|^\alpha}{|y-x|}-\frac{|x'|^\alpha}{|y-x'|}\Big)^2X_s(dy) ds\Big)\nonumber\\
&\leq 2c_{\ref{t1}}^2 \int_0^t \int \Big(\frac{|x|^\alpha}{|y-x|}-\frac{|x|^\alpha}{|y-x'|}\Big)^2 p_s(y) dy ds+2c_{\ref{t1}}^2 \int_0^t \int \Big(\frac{|x|^\alpha}{|y-x'|}-\frac{|x'|^\alpha}{|y-x'|}\Big)^2 p_s(y) dy ds\nonumber\\
&=:  2c_{\ref{t1}}^2 (I+J).
\end{align}
By Lemma \ref{l5.4} we have
\begin{align*}
J&=(|x|^\alpha-|x'|^\alpha)^2 \int_0^t \int \frac{1}{|y-x'|^2} p_s(y) dy ds \leq |x-x'|^{2\alpha} 2(\log^+ \frac{1}{|x'|}+1+\sqrt{3t})\\
&\leq |x-x'|^{\alpha} (2|x'|)^\alpha 2(\log^+ \frac{1}{|x'|}+1+\sqrt{3t}) \leq 4|x-x'|^{\alpha} (1/\alpha+1+\sqrt{3t}),
\end{align*}
the last by $|x'|<1$ and $|x'|^\alpha \log^{+} (1/|x'|) \leq 1/\alpha$.\\

Now we deal with $I$. Since $I=0$ if $|x|=0$, we may assume $x\neq 0$. 
Note that for any $0<\gamma<1$,
\begin{align}\label{e5.20.1}
\Big|\frac{1}{|y-x_n|}-\frac{1}{|y-x_0|}\Big| \leq& |x_n-x_0|^{\gamma} \frac{||y-x_n|-|y-x_0||^{1-\gamma}}{|y-x_n||y-x_0|} \nonumber\\
\leq &  |x_n-x_0|^{\gamma} \Big(\frac{1}{|y-x_0|^{1+\gamma}}+\frac{1}{|y-x_n|^{1+\gamma}}\Big).
\end{align}
Let $\gamma=\alpha/2$ in \eqref{e5.20.1} to see that
\begin{align*}
I&\leq |x|^{2\alpha} |x-x'|^\alpha \int_0^t \int \Big(\frac{1}{|y-x|^{2+\alpha}}+\frac{1}{|y-x'|^{2+\alpha}}\Big) p_s(y) dy ds.
\end{align*}
Use similar Ito's Lemma arguments proving Lemma \ref{l5.4} above to conclude for any $0<\alpha<1$ and $x\neq 0$ in $\R^3$,
\[ \int_0^t \int \frac{1}{|y-x|^{2+\alpha}}p_s(y) dy ds \leq \frac{2}{\alpha(1-\alpha)}|x|^{-\alpha}. \]
Therefore
\begin{align*}
I&\leq |x|^{2\alpha} |x-x'|^\alpha C(\alpha) (|x|^{-\alpha}+|x'|^{-\alpha}) \leq C(\alpha) |x-x'|^\alpha,
\end{align*}
the last follows since we assumed $|x|\leq |x'|\leq 1$. Therefore \eqref{e9.4} becomes
\begin{align}\label{e9.1}
E_{\delta_0} \Big( \int_0^t \Big(\frac{|x|^\alpha}{|y-x|}-\frac{|x'|^\alpha}{|y-x'|}\Big)^2 X_s(dy) ds\Big) \leq r|x-x'|^\alpha (t^{1/2}+\beta),
\end{align}
where $r=8\sqrt{3} c_{\ref{t1}}^2$ and $\beta=\beta(\alpha)$ with $\alpha \in (0,1)$.  Let $\phi(y)=(|x|^\alpha/|y-x|-|x'|^\alpha/|y-x'|)^2$ in \eqref{e4.2}, then the above gives $|v_1(t,z)| \leq r|x-x'|^\alpha (t^{1/2}+\beta)$ holds for all $z$ and $t\geq 0$. Apply Lemma \ref{l4.3} to get 
\begin{align}\label{e9.0}
|v_n(t,z)|\leq c_n r^n |x-x'|^{n\alpha} t^{(n-1)/2} (t^{1/2}+\beta)^{2n-1}.
\end{align}
By using the same arguments in proving Lemma \ref{l4.4}, one can show that the following holds formally:
 \begin{align*}
&E_{\delta_0} \Big[\exp{ \Big(\theta  \int_0^t \Big(\frac{|x|^\alpha}{|y-x|}-\frac{|x'|^\alpha}{|y-x'|}\Big)^2 X_s(dy) ds- \theta  \int_0^t \Big(\frac{|x|^\alpha}{|y-x|}-\frac{|x'|^\alpha}{|y-x'|}\Big)^2 p_s(y)dy ds \Big) } \Big]\\
&=\exp{\Big(2\sum_{n=2}^{\infty} (\frac{\theta}{2})^n \delta_0 \big(v_n(t)\big)\Big)}. 
 \end{align*}
  By \eqref{e9.0}, Lemma \ref{l4.5}(i) implies 
 \begin{align*}
 &E_{\delta_0} \Big[\Big(  \int_0^t \Big(\frac{|x|^\alpha}{|y-x|}-\frac{|x'|^\alpha}{|y-x'|}\Big)^2 X_s(dy) ds-   \int_0^t \Big(\frac{|x|^\alpha}{|y-x|}-\frac{|x'|^\alpha}{|y-x'|}\Big)^2 p_s(y)dy ds \Big)^{2N}\Big]\\
 & \leq C(t,N,\alpha) |x-x'|^{2N\alpha}
  \end{align*}
 and by \eqref{e9.1} 
 \begin{align*}
 E_{\delta_0} \Big[\Big( \int_0^t \Big(\frac{|x|^\alpha}{|y-x|}-\frac{|x'|^\alpha}{|y-x'|}\Big)^2 X_s(dy) ds\Big)^{2N}\Big] &\leq 2^{2N} \Big(C(t,N,\alpha) |x-x'|^{2N\alpha}+(r|x-x'|^\alpha (t^{1/2}+\beta))^{2N}\Big)\\
 &=C(t,N,\alpha) |x-x'|^{2N\alpha}.
  \end{align*}
 By the Burkholder-Davis-Gundy Inequality, there exists some $C_N>0$ such that for all $x,x' \in B(0,1)$,
 \begin{align*}
 E_{\delta_0}\Big((|x|^\alpha M_t(\phi_x)-|x'|^\alpha M_t(\phi_{x'}))^{4N}\Big)&\leq C_N (c_{\ref{t1}})^{4N} E_{\delta_0}\Big( \int_0^t \Big(\frac{|x|^\alpha}{|y-x|}-\frac{|x'|^\alpha}{|y-x'|}\Big)^2 X_s(dy) ds\Big)^{2N}\\
 &\leq C(t,N,\alpha) |x-x'|^{2N\alpha}.
 \end{align*}
Take $N$ large enough to apply Kolmogorov's continuity criterion and so obtain a continuous version of $|x|^\alpha M_t(\phi_x)$ on $x \in B(0,1)$.
\end{proof}

\section{General Initial Condition in $d=2$}
Now we are in the case $d=2$. Recall the Tanaka formula \eqref{e2.9} that if $\int\log^+ (1/|y-x|) \mu(dy) <\infty$, then 
\begin{equation}
L_t^x -\mu(g_{\alpha,x})=M_t(g_{\alpha,x})+\alpha \int_0^t X_s(g_{\alpha,x}) ds-X_t(g_{\alpha,x}),\label{e4.11} 
\end{equation}
where $\alpha>0$ and
\begin{equation*}
g_{\alpha,x} (y)=\int_0^\infty e^{-\alpha t} p_t(y-x) dt. 
\end{equation*}
We prove in Appendix C(i) that $g_{\alpha,x}(y)-(1/\pi) \log^+ (1/|y-x|)=f_{\alpha}(y-x)$ where $f_\alpha$ defined in \eqref{e9.2} can be extended to be a bounded continuous function on $\R^2$. Hence
\begin{equation}\label{e4.111} 
x\mapsto \int (g_{\alpha,x}(y)-\frac{1}{\pi} \log^+ \frac{1}{|y-x|})  \mu(dy) \text{\ is continuous on } \R^2.
\end{equation}
So the joint continuity of $L_t^x-\mu(g_{\alpha,x})$ would prove that there is a jointly continuous version of \[L_t^x - \int \frac{1}{\pi} \log^+ \frac{1}{|y-x|} \mu(dy)=\Big(L_t^x -\mu(g_{\alpha,x})\Big)+\int (g_{\alpha,x}(y)-\frac{1}{\pi} \log^+ \frac{1}{|y-x|})  \mu(dy)\] on $\{(t,x): t>0, x\in \R^2\} \bigcup \{(0,x): x$ is a continuity point of $\int \log^+ (1/|y-x|) \mu(dy)\}$.\\

By (3.44) from Sugitani (1989), for any $0<\gamma\leq 1$, there exists some $c=c(\gamma)>0$ such that 
\begin{equation}
|p_t(x)-p_t(y)| \leq c t^{-\gamma/2} |x-y|^\gamma (p_{2t}(x)+p_{2t}(y)), \ t>0, \ x,y \in \R^d. \label{e4.12}
\end{equation}
Then a simple calculation will give us
\begin{align} \label{e4.13}
&|g_{\alpha,x}(y)-g_{\alpha,x'}(y)| \leq c(\gamma)\  |x-x'|^\gamma \int_0^\infty e^{-\alpha t}  t^{-\gamma/2}  (p_{2t}(y-x)+p_{2t}(y-x')) dt\nonumber\\
\leq& c(\gamma)\  |x-x'|^\gamma \Big(\int_0^\infty t^{-\gamma/2}  \frac{1}{4\pi t} e^{-\frac{|y-x|^2}{4t}} dt+\int_0^\infty t^{-\gamma/2} \frac{1}{4\pi t} e^{-\frac{|y-x'|^2}{4t}}  dt\Big)\nonumber\\
=& c(\gamma)\  |x-x'|^\gamma \Big(\int_0^\infty \frac{1}{|y-x|^\gamma} \frac{4^{\gamma/2}}{4\pi} \frac{1}{s^{1-(\gamma/2)}}e^{-s}ds+\int_0^\infty \frac{1}{|y-x'|^\gamma}  \frac{4^{\gamma/2}}{4\pi} \frac{1}{s^{1-(\gamma/2)}} e^{-s}ds\Big)\nonumber\\
=& C(\gamma)\ |x-x'|^\gamma (\frac{1}{|y-x|^\gamma}+\frac{1}{|y-x'|^\gamma}).
\end{align}
Now we proceed to the
\begin{proof}[Proof of Theorem \ref{t5}]
Fix any $t>0$.
\begin{enumerate}[(i)]
\item $M_t(g_{\alpha,x})$: Let $\gamma=1/2$ and $C(1/2)$ be as in \eqref{e4.13}. Then an argument similar to the derivation of \eqref{e2.13} shows that
\begin{align*}
 E_{\mu} \Big[\big|  M_t&(g_{\alpha,x})-M_t(g_{\alpha,x'})\big|^{4n}\Big]   \leq C_n (C(1/2)|x-x'|)^{2n} 2^{2n} \nonumber\\
&\times E_{\mu}\Big[\big( \int_0^t ds \int X_s(dy) \frac{1}{|y-x|}\big)^{2n}+\big( \int_0^t ds \int X_s(dy) \frac{1}{|y-x'|}\big)^{2n}\Big]\nonumber\\
\leq   C_n &(C|x-x'|)^{2n} 2^{2n} C(t,n,\mu(1)) \hfill \text{\ \  (Lemma \ref{l2.9}). }
\end{align*}
Therefore there exists a continuous version of $M_t(g_{\alpha,x})$ in $x$.

\item $\int_0^t X_s(g_{\alpha,x}) ds$: Let $\gamma=1$ in \eqref{e4.13}. Then for each $n\geq 1$, 
\begin{align*}
  E_{\mu}\Big[\big| \int_0^t &X_s(g_{\alpha,x}) ds-\int_0^t X_s(g_{\alpha,x'}) ds\big|^{2n}\Big] \leq  (C|x-x'|)^{2n} 2^{2n}\\
  &\times E_{\mu}\Big[(\int_0^t ds \int X_s(dy) \frac{1}{|y-x|})^{2n} + (\int_0^t ds \int X_s(dy) \frac{1}{|y-x'|})^{2n} \Big]\\
\leq  (C&|x-x'|)^{2n} 2^{2n} \cdot C(t,n,\mu(1)) \hfill \text{\ \  (Lemma \ref{l2.9}). }
\end{align*}
Therefore there exists a continuous version of $\int_0^t X_s(g_{\alpha,x}) ds$ in $x$.

\item $X_t(g_{\alpha,x})$: By using \eqref{e4.13} with $\gamma=1$, we have 
\begin{align*}
|X_t(g_{\alpha,x})-X_t(g_{\alpha,x'})| &\leq C|x-x'| \int X_t(dy) (\frac{1}{|y-x|}+\frac{1}{|y-x'|}).
\end{align*}
Note that with $P_\mu$-probability one there exist some $r_0(t,\omega)\in (0,1]$ and some constant $C>0$ such that for all $x$,  \[\int X_t(dy) \frac{1}{|y-x|} \leq \frac{1}{r_0} X_t(1)+\int_{|y-x|<r_0} X_t(dy) \frac{1}{|y-x|} \leq \frac{1}{r_0} X_t(1)+C.\]
Then 
\[|X_t(g_{\alpha,x})-X_t(g_{\alpha,x'})|  \leq |x-x'| \cdot C(r_0)(X_t(1)+1),\] and the continuity of $x \mapsto  X_t(g_{\alpha,x})$ follows.
\end{enumerate}
Combining (i), (ii) and (iii) above, for any fixed $\varepsilon>0$, we have established that there is a $P_\mu$-a.s. continuous version of $L_\varepsilon^x-\mu(g_{\alpha,x})$ in $x$. Then use \eqref{e3.1} to conclude that $(L_t^x-\mu(g_{\alpha,x}))-(L_\varepsilon^x-\mu(g_{\alpha,x}))=L_t^x-L_\varepsilon^x$ is jointly continuous on $\{(t,x): t\geq \varepsilon, x\in \R^2\}$. Therefore by choosing $\omega$ outside a null set $N(\varepsilon)$, we can see that there is a jointly continuous version of $L_t^x-\mu(g_{\alpha,x})$ on $\{(t,x): t\geq \varepsilon, x\in \R^2\}$. Now take $\varepsilon=1/n$ and $N=\cup_{n=1}^\infty N(1/n)$ to see that for $\omega \in N^c$, there is a jointly continuous version of $L_t^x-\mu(g_{\alpha,x})$ on $\{(t,x): t>0, x\in \R^2\}$.
\end{proof}

\begin{proof}[Proof of Corollary \ref{c1.0}]
For points $(0,x)$ such that $x$ is a continuity point of $\int \log^+ (1/|y-x|) \mu(dy)$, it follows from Appendix C(ii) that $(0,x)$ is a joint continuity point of $\mu q_t(x)$. By Theorem \ref{t3}, we have $(0,x)$ is a joint continuity point of $L_t^x$. Therefore such a point $(0,x)$ is a joint continuity point of $L_t^x-\int (1/\pi) \log^+ (1/|y-x|) \mu(dy)$. 
\end{proof}

\section{General Initial Condition in $d=3$}

\subsection{Smooth Cutoff of Logarithm }
Now we are in the case $d=3$. Recall $\eta(x)$ defined as in \eqref{e4.8}. For each $\varepsilon>0$, set 
$\eta_\varepsilon(x):=\frac{1}{\varepsilon^3} \eta\Big(\frac{x}{\varepsilon}\Big) $ such that
the function $\eta_\varepsilon$ are $C^{\infty}$ and satisfy $\int_{\R^3} \eta_\varepsilon dx=1$ with support in $B(0,\varepsilon)$. 
If $\chi_{1/2}$ is the convolution of $\eta_{1/4}$ and the indicator function of the ball $B(0,3/4)$, then
\[\chi_{1/2} (x)=\int_{\R^3} 1_{\{|x-y|<3/4\}} \eta_{1/4}(y) dy=\int_{|y|<1/4} 1_{\{|x-y|<3/4\}} \eta_{1/4}(y) dy.\]
One can check that $0\leq \chi_{1/2}\leq 1$ and $\chi_{1/2}$ is a $C^{\infty}$ function with support in $B(0,1)$, and $\chi_{1/2} (x) \equiv1$ if $|x|<1/2$. Now define $\bar{g}_x(y)\equiv \log |y-x| \cdot \chi_{1/2}(y-x) \text{\ for\ } y\neq x.$ By definition we have 
\begin{equation}\label{e5.2}
0\leq -\bar{g}_x(y) \leq \log^+(1/|y-x|), \  y\neq x, 
\end{equation} 
and 
\begin{align}\label{e5.1}
-\bar{g}_x(y) = \log^+ (1/|y-x|), \ 0<|y-x|<1/2.
\end{align}
By \eqref{e2.6} and \eqref{e5.1}, one can check that there is some constant $C\geq 1$ such that
\begin{align}\label{e5.3}
|\Delta\bar{g}_x(y)| \leq C\frac{1}{|y-x|^2}, \ y\neq x.
\end{align}
 \noindent Define
\begin{equation}
\bar{f} (y):=
\begin{cases} 
-\bar{g}_0(y)-\log^+ (1/|y|),  & \mbox{if } y\neq 0, \\
0, & \mbox{if } y=0,\\
\end{cases} \label{e5.31}
\end{equation}
and
\begin{equation}
\bar{h} (y):=
\begin{cases} 
\Delta \bar{g}_0(y)-1/|y|^2,  & \mbox{if } y\neq 0, \\
0, & \mbox{if } y=0.\\
\end{cases} \label{e5.4}
\end{equation}
By using \eqref{e5.1}, one can check both $\bar{f}$ and $\bar{h}$ are in $C_b(\R^3)$.
\begin{proposition}\label{p5.1}
Let $X$ be a super-Brownian motion in $d=3$ with initial condition $\mu \in M_F(\R^3)$. For any $x\in \R^3$ with $\int \mu(dy) \log^{+} (1/|y-x|) <\infty$, we have $P_{\mu}$-a.s. 
\begin{align} \label{e5.5}
 X_t(\bar{g}_x)=\mu(\bar{g}_x)+M_t(\bar{g}_x)+ \int_0^t X_s(\frac{\Delta}{2} \bar{g}_x ) ds, \ \forall t\geq 0,
\end{align}
where $ X_t(\bar{g}_x)$ is continuous in $t$ and $M_t(\bar{g}_x)$ is a continuous $L^2$ martingale.
\end{proposition}

Note that the proof of Proposition \ref{p5.1} is very similar to that of Proposition \ref{p2}. Since we have $P_{\varepsilon} \bar{g}_x\in C_b^2(\R^3)$, it follows from the martingale problem \eqref{e1.0} that with $P_\mu$-probability one, for all $t\geq 0$ we have
\begin{equation} \label{e5.02}
X_t(P_{\varepsilon} \bar{g}_x)=\mu(P_{\varepsilon} \bar{g}_x)+ M_t(P_{\varepsilon} \bar{g}_x)+\int_0^t X_s\Big(\frac{\Delta}{2} P_{\varepsilon} \bar{g}_x\Big) ds,
\end{equation}
where $M_t(P_{\varepsilon} \bar{g}_x)$ is a martingale with quadratic variation \[[M(P_{\varepsilon} \bar{g}_x)]_t=\int_0^t X_s\Big((P_{\varepsilon} \bar{g}_x)^2\Big) ds.\]
Before proceeding to the proof of Proposition \ref{p5.1}, we state some preliminary results.
\subsubsection{Preliminaries}

\begin{lemma}\label{l5.2}
Let $d=3$. Then for any fixed $\varepsilon>0$ and $y\neq x$, we have 
\[\Delta_y P_\varepsilon \bar{g}_x(y)= \int p_\varepsilon(y-z) \Delta_z \bar{g}_x(z) dz. \]
\end{lemma}
\begin{proof}
The proof is similar to that of Lemma \ref{l3.2}.
\end{proof}

\begin{lemma}\label{l5.3}
For $d \geq 1$, there exists a constant $C=C(d)>0$ such that for any $x\neq 0$ in $\R^d$ and $t>0$ , \[\int p_t(y) \log^+ \frac{1}{|y-x|} dy \leq C(1+\log^+ \frac{1}{|x|}).\]
\end{lemma}

\begin{proof}
The proof is similar to that of Lemma \ref{l3.3}.
\end{proof}

\subsubsection{Proof of Proposition \ref{p5.1}}

\begin{enumerate}[(i)]
\item  Compared to \eqref{e3.7}, we prove in Appendix B(i) that there is some constant $C>0$ such that for any $y\neq x$ and any $0<\varepsilon<1$,
\begin{equation}\label{e5.6}
\int p_\varepsilon(y-z) |\bar{g}_x(z) -\bar{g}_x(y)| dz \leq C|y-x|^{-1/2} \varepsilon^{1/4}.
\end{equation}
By \eqref{e5.2} and Lemma \ref{l5.3}, for all $\varepsilon>0$ we have 
\begin{equation}
\int p_\varepsilon(y-z) |\bar{g}_x(z) -\bar{g}_x(y)| dz \leq C(1+\log^+ (1/|y-x|)),\label{e5.7}
\end{equation}
which is integrable w.r.t $\mu(dy)$ by assumption. Dominated Convergence Theorem implies \[\int \mu(dy)   \int p_\varepsilon(y-z) |\bar{g}_x(z) -\bar{g}_x(y)| dz \to 0 \text{\ as \ } \varepsilon \to 0,\] and it follows that $\mu(P_\varepsilon \bar{g}_x) \to \mu(\bar{g}_x)$.

\item By using \eqref{e5.6},  it follows by the same argument in proving (ii) in Section \ref{s3.3.1} that for any $T>0$, there is some subsequence $\varepsilon_n \downarrow 0$ such that \[\sup_{t\leq T} |M_t(P_{\varepsilon_n} \bar{g}_x)-M_t(\bar{g}_x)| \to 0, \ P_{\mu}\text{-a.s..} \]

\item For any $T>0$ we have
\begin{align*}
& E_{\mu} \Big(\sup_{t\leq T}\Big|\int_0^t X_s\big(\frac{\Delta}{2} P_{\varepsilon} \bar{g}_x) ds-\int_0^t X_s(\frac{\Delta}{2} \bar{g}_x)ds \  \Big|\Big)\leq E_{\mu} \Big(\int_0^T X_s\big(|\frac{\Delta}{2} P_{\varepsilon}  \bar{g}_x-\frac{\Delta}{2} \bar{g}_x|\big)ds \  \Big) \\
=& \frac{1}{2} \int \mu(dw) \int_0^T ds \int p_s(w-y)  \Big|\int p_\varepsilon(y-z) \Delta  \bar{g}_x(z) dz-\Delta \bar{g}_x(y)\Big| dy.
\end{align*}
The last inequality is by Lemma \ref{l5.2}. Recall $\bar{h}$ defined as in \eqref{e5.4}. For $y\neq x$, we have
\begin{equation} \label{e5.81}
\Delta \bar{g}_x(y)-\frac{1}{|y-x|^2}=\bar{h}(y-x) \text{\ where\ } \bar{h} \in C_b(\R^3).
\end{equation} Then Dominated Convergence Theorem implies that as $\varepsilon \to 0$, \[\Big|\int p_\varepsilon(y-z) \bar{h}(z-x) dz- \bar{h}(y-x)\Big|\leq E |\bar{h}(B_\varepsilon-(y-x))-\bar{h}(y-x)| \to 0.\]
Together with \eqref{e3.9} we have for $y\neq x$,
\begin{align*}
\Big|\int p_\varepsilon(y-z) \Delta \bar{g}_x(z) dz-\Delta \bar{g}_x(y)\Big| &\leq \Big|\int p_\varepsilon(y-z) \frac{1}{|z-x|^2} dz- \frac{1}{|y-x|^2}\Big|\\
&+\Big|\int p_\varepsilon(y-z) \bar{h}(z-x) dz- \bar{h}(y-x)\Big| \to 0 \text{ as } \varepsilon \to 0. 
\end{align*}

By \eqref{e5.3} and Lemma \ref{l3.3}, for $y\neq x$ we have
\begin{align*}
\Big|\int p_\varepsilon(y-z) \Delta  \bar{g}_x(z) dz-\Delta \bar{g}_x(y)\Big|& \leq C\int p_\varepsilon(y-z) \frac{1}{|z-x|^{2}} dz+C\frac{1}{|y-x|^{2}}\leq C\frac{1}{|y-x|^{2}},
\end{align*}
 which is integrable w.r.t. $\int \mu(dw) \int_0^T ds  p_s(w-y) dy$ by Lemma \ref{l5.4} and the assumption on $\mu$. Therefore Dominated Convergence Theorem implies
 \[\int \mu(dw) \int_0^T ds \int p_s(w-y) dy \Big|\int p_\varepsilon(y-z) \Delta  \bar{g}_x(z) dz-\Delta \bar{g}_x(y)\Big| \to 0\text{\ as \ } \varepsilon \to 0,\]
 and hence \[ \sup_{t\leq T} \int_0^t X_s\big(|\frac{\Delta}{2} P_{\varepsilon} \bar{g}_x-\frac{\Delta}{2} \bar{g}_x|) ds  \xrightarrow[]{L^1} 0.\] Take a subsequence $\varepsilon_n \downarrow 0$ to obtain
\[ \sup_{t\leq T}\Big|\int_0^t X_s\big(\frac{\Delta}{2} P_{\varepsilon} \bar{g}_x) ds-\int_0^t X_s(\frac{\Delta}{2} \bar{g}_x)ds \  \Big|\to 0, \ P_{\mu}\text{-a.s..} \]
\item Fix any $T>0$. Set $0<\delta<T$, which will be chosen small enough below. Then we have \[\sup_{t\leq T} |X_t(P_{\varepsilon} \bar{g}_x)-X_t(\bar{g}_x)| \leq \sup_{t\leq \delta} |X_t(P_{\varepsilon} \bar{g}_x)-X_t(\bar{g}_x)|+\sup_{\delta\leq t\leq T} |X_t(P_{\varepsilon} \bar{g}_x)-X_t(\bar{g}_x)|. \] For the second term on the right-hand side, we recall from \eqref{e3.02} that with $P_{\mu}$-probability one, for any $\delta>0$, there is some $r_0(\delta,\omega)\in (0,1]$ and some constant $C>0$ such that for any $T>0$,
\begin{equation*}
\sup_{\delta \leq t\leq T} \int |y-x|^{-1/2} X_t(dy) \leq r_0^{-1/2} \sup_{\delta \leq t\leq T} X_t(1)+ C.
\end{equation*} 
Therefore by \eqref{e5.6} we have
\begin{align*}
\sup_{\delta \leq t\leq T} |X_t(P_{\varepsilon} \bar{g}_x)-X_t(\bar{g}_x)| &\leq C\varepsilon^{1/4} \sup_{\delta \leq t\leq T} \int |y-x|^{-1/2} X_t(dy) \\
&\leq C\varepsilon^{1/4} (r_0^{-1/2} \sup_{\delta \leq t\leq T} X_t(1)+ C) \to 0  \text{ as } \varepsilon \to 0.
\end{align*} 
Now we deal with $t\leq \delta$. Let $\beta_k \downarrow 0$ satisfy $\mu \big(\{y: |y-x|=\beta_k\}\big)=0$. Then use \eqref{e5.6} to get
\begin{align*}
\sup_{t\leq \delta} |X_t(P_{\varepsilon} \bar{g}_x)-X_t(\bar{g}_x)| \leq& C\varepsilon^{1/4} \beta_k^{-1/2} \sup_{t\leq \delta}  X_t(1)\\
&+\sup_{t\leq \delta} \int |P_{\varepsilon} \bar{g}_x(y)-\bar{g}_x(y)| 1_{|y-x|\leq \beta_k} X_t(dy).
\end{align*}
By Lemma \ref{l5.3}, for $|y-x|\leq \beta_k<1/4$ we have $\int p_{\varepsilon}(y-z) \log^{+} (1/|z-x|) dz \leq C(1+\log^+ (1/|y-x|))\leq 2C \log^+ (1/|y-x|)$. Hence by \eqref{e5.2} and \eqref{e5.1}, for $|y-x|\leq \beta_k<1/4$ we have
\[|P_{\varepsilon} \bar{g}_x(y)-\bar{g}_x(y)|\leq C\log^+ (1/|y-x|) = C  |\bar{g}_x(y)|.\] So
\begin{align}
&\sup_{t\leq \delta} \int |P_{\varepsilon} \bar{g}_x(y)-\bar{g}_x(y)| 1_{|y-x|\leq \beta_k} X_t(dy)\leq C \sup_{t\leq \delta} \int |\bar{g}_x(y)| 1_{|y-x|\leq \beta_k} X_t(dy)\nonumber\\
\leq &C \sup_{t\leq \delta} \Big(\int |\bar{g}_x(y)| 1_{|y-x|\leq\beta_k} X_t(dy)-\int |\bar{g}_x(y)|1_{|y-x|\leq\beta_k} \ \mu(dy)\Big)+C \int |\bar{g}_x(y)|1_{|y-x|\leq\beta_k} \mu(dy). \label{e9.7}
\end{align} 
Use Dominated Convergence Theorem to get  \[\int |\bar{g}_x(y)| 1_{|y-x|\leq \beta_k} \ \mu(dy) \leq \int  \log^+(1/|y-x|) 1_{|y-x|\leq \beta_k} \ \mu(dy)\to 0 \text{ as } \beta_k \to 0. \]
To handle the first term on the right hand side of \eqref{e9.7} we may apply Fatou's lemma in \eqref{e5.02} and use the convergence established in (i) to see that for all $t\leq \delta$,
\begin{align} \label{e5.04}
&X_t(|\bar{g}_x|)-\mu(|\bar{g}_x|)=X_t(-\bar{g}_x)-\mu(-\bar{g}_x) \leq \liminf_{\varepsilon_n \to 0} \Big[M_t(-P_{\varepsilon_n} \bar{g}_x )+ \int_0^t X_s(-\frac{\Delta}{2} P_{\varepsilon_n} \bar{g}_x ) ds \Big]\nonumber \\
=& M_t(-\bar{g}_x)+\int_0^t  X_s(-\frac{\Delta}{2} \bar{g}_x) ds\leq \sup_{t\leq \delta} M_t(|\bar{g}_x|)+\int_0^\delta X_s(|\frac{\Delta}{2} \bar{g}_x|) ds.
\end{align} 
For the last equality, we use the a.s. convergence established in (ii) and (iii) above. Then
\begin{align*}
&\sup_{t\leq \delta} (\int |\bar{g}_x(y)| 1_{|y-x|\leq\beta_k} X_t(dy)-\int |\bar{g}_x(y)|1_{|y-x|\leq\beta_k} \ \mu(dy))\\
\leq& \sup_{t\leq \delta} (X_t(|\bar{g}_x|)-\mu(|\bar{g}_x|)) +\sup_{t\leq \delta} (\int |\bar{g}_x(y)| 1_{|y-x|>\beta_k} X_t(dy)-\int |\bar{g}_x(y)| 1_{|y-x|>\beta_k} \mu(dy)).
\end{align*} 
Let $\delta \downarrow 0$. Then the first term converges to $0$ by \eqref{e5.04}. The second term follows by the weak continuity of $X_t$ with the choice of $\{\beta_k\}$, in a way such that the set of discontinuity points of the bounded function $|\bar{g}_x(y)| 1_{|y-x|>\beta_k}$ is $\mu$-null. Let $\varepsilon \to 0$ to conclude that
\[ \sup_{t\leq T} |X_t(P_{\varepsilon} \bar{g}_x)-X_t(\bar{g}_x)| \to 0, \ P_{\mu}\text{-a.s..} \]
\end{enumerate}

The proof of Proposition \ref{p5.1} follows by the a.s. convergence (i)-(iv) and \eqref{e5.02}. $\hfill\blacksquare$

\subsection{Proof of Theorem \ref{t4.1}}
Recall from \eqref{e1.4} that \[D=\{x_0\in \R^3: \int \frac{1}{|y-x_0|} \mu(dy)=\infty  \}.\]

\begin{lemma}\label{l5.5}
Let $D$ be defined as in \eqref{e1.4}. Then $D$ is a Lebesgue null set in $\R^3$ and in particular $D^c$ is dense in $\R^3$.
\end{lemma}
\begin{proof}
By Fubini's Theorem, we have \[\int dx \int \frac{1}{|y-x|} 1_{\{|y-x|<1\}} \mu(dy)= 2\pi \mu(1)<\infty,\] and it follows that  $\int \frac{1}{|y-x|} 1_{\{|y-x|<1\}} \mu(dy)<\infty \text{\ for Lebesgue a.a. } x\in \R^3.$ Therefore \[\int \frac{1}{|y-x|}\mu(dy)\leq \mu(1)+\int \frac{1}{|y-x|} 1_{\{|y-x|<1\}} \mu(dy)<\infty \text{\ for Lebesgue a.a. } x\in \R^3,\]  that is to say, $D^c$ has full measure. Hence $D^c$ is dense. 
\end{proof}
\noindent Now we turn to the
\begin{proof}[Proof of Theorem \ref{t4.1}]
Fix any $x_0\in D$. Since $D^c$ is dense, there exists some sequence $x_n\in D^c$ such that $x_n\to x_0$ as $n\to \infty$. Fix any such sequence. By Fatou's lemma
\begin{equation}
\liminf_{n\to\infty} \int \frac{1}{|y-x_n|} \mu(dy) =\infty. \label{e5.9}
\end{equation}
 Recall the Tanaka formula \eqref{e1.5} and recall that $\phi_x(y)=c_{\ref{t1}}/|y-x|$. For $x_n \in D^c$, since $\mu(\phi_{x_n})<\infty,$ we have 
 \begin{align}\label{e5.161}
  L_t^{x_n}-\mu(\phi_{x_n})=M_t(\phi_{x_n})-X_t(\phi_{x_n}).
 \end{align}
 Then
\begin{align*}
 P_{\mu}\Big(L_t^{x_n}<\frac{1}{2} \mu(\phi_{x_n})\Big)&\leq P_{\mu}\Big(|L_t^{x_n}-\mu(\phi_{x_n})|>\frac{1}{2} \mu(\phi_{x_n})\Big) \\
&\leq  \ P_{\mu}\Big(|M_t(\phi_{x_n})|>\frac{1}{4} \mu(\phi_{x_n})\Big)+P_{\mu}\Big(|X_t(\phi_{x_n})|>\frac{1}{4} \mu(\phi_{x_n})\Big).
\end{align*}
Note that 
\begin{align*}
E_{\mu}\Big(|X_t(\phi_{x_n})|\Big)=\int \mu(dy) \int p_t(y-z) \frac{c_{\ref{t1}}}{|z-x|} dz.
\end{align*}
Fix any $t>0$. For $x, y\in \R^3$, use the fact $1/|y-x|=\int_0^\infty 2\pi p_s(y-x) ds$ to see that 
\begin{align}
 &\int p_t(y) \frac{1}{|y-x|} dy= \int p_t(y) dy  \int_0^\infty 2\pi p_s(y-x) ds = \int_0^\infty  2\pi p_{t+s}(x)ds\nonumber\\
  \leq& \int_0^\infty  2\pi p_{t+s}(0)ds=\int p_t(y) dy  \int_0^\infty 2\pi p_s(y) ds = \int p_t(y)\frac{1}{|y|} dy = C(t) <\infty. \label{e5.10}
\end{align}
Therefore we have
\begin{equation}
E_{\mu}\Big(|X_t(\phi_{x_n})|\Big) \leq c_{\ref{t1}} C(t) \mu(1),\label{e5.11}
\end{equation}
and \[P_{\mu}\Big(|X_t(\phi_{x_n})|>\frac{1}{4} \mu(\phi_{x_n})\Big)\leq \frac{4c_{\ref{t1}} C(t) \mu(1)}{ \mu(\phi_{x_n})}.\]
Next use Lemma \ref{l5.4} to get
\begin{align}\label{e9.21}
&E_\mu \Big[M_t^2(\phi_{x_n})\Big]= E_\mu \Big[\int_0^t X_s(\phi_{x_n}^2) ds\Big]=\int \mu(dy) \int_0^t ds \int p_s(y-z) \frac{c_{\ref{t1}}^2}{|z-x_n|^2} dz\nonumber \\
\leq& 2 c_{\ref{t1}}^2 \int \mu(dy) \Big((3t)^{1/2}+1+\log^+ (1/|y-x_n|)\Big) \leq C(t) \mu(1)+\mu(\phi_{x_n}),
\end{align}
and it follows that
\[P_{\mu}\Big(|M_t(\phi_{x_n})|>\frac{1}{4} \mu(\phi_{x_n})\Big) \leq \frac{16(C(t) \mu(1)+\mu(\phi_{x_n}))}{(\mu(\phi_{x_n}))^2}. \]
Together we have
\[ P_{\mu}\Big(L_t^{x_n}<\frac{1}{2} \mu(\phi_{x_n})\Big)\leq  \frac{4c_{\ref{t1}} C(t) \mu(1)}{ \mu(\phi_{x_n})}+ \frac{16(C(t) \mu(1)+\mu(\phi_{x_n}))}{(\mu(\phi_{x_n}))^2} \to 0 \text{\ by\ } \eqref{e5.9}.\]
Hence $L_t^{x_n} \to \infty$ in probability since $\mu(\phi_{x_n})\to \infty$. Take a subsequence $x_{n_k} \to x_0$ to get $L_t^{x_{n_k}} \to \infty, P_\mu$-a.s. Therefore by choosing $\omega$ outside a null set $N(t)$, we have $L_t^{x_{n_k}} \to \infty$ as $x_{n_k} \to x_0$. Now take $t=1/n$ and $N=\cup_{n=1}^\infty N(1/n)$ to see that for $\omega \in N^c$, we have $L_t^{x_{n_k}} \to \infty$ as $x_{n_k} \to x_0$ for all $t>0$ since $L_t^x$ is monotone in $t$. So we conclude that with $P_\mu$-probability one,  we have for all $t>0$, $x\mapsto L_t^x$ is discontinuous at $x_0$.
\end{proof}

\subsection{Proof of Theorem \ref{t4}}
\subsubsection{Proof of Theorem \ref{t4}(a)}
The proof is similar to that of Theorem \ref{t1}. Use \eqref{e5.11} to see that \[\frac{X_t(\phi_{x_n})}{(2c_{\ref{t1}}^2 \int \mu(dy) \log^+(1/|y-x_n|))^{1/2}} \xrightarrow[]{L^1} 0,  \] and hence convergence in probability follows. Then use the arguments in Section \ref{s2.2} to see that the proof of Theorem \ref{t4}(a) can be reduced to the proof of \[\frac{[M(\phi_{x_n})]_t}{2c_{\ref{t1}}^2 \int \log^+(1/|y-x_n|) \mu(dy) } \xrightarrow[]{P_\mu} 1.\] Note that \[[M(\phi_{x_n})]_t=c_{\ref{t1}}^2 \int_0^t ds \int X_s(dy) \frac{1}{|y-x_n|^2}. \] Recall from \eqref{e5.81} that 
\begin{equation}\label{e5.16}
\Big|\Delta \bar{g}_{x_n} (y)-\frac{1}{|y-x_n|^2}\Big|=|\bar{h}(y-x_n)| \leq \|\bar{h}\|_\infty<\infty. 
\end{equation} Therefore it suffices to show
\begin{equation}
\frac{\int_0^t X_s(\frac{\Delta}{2} \bar{g}_{x_n}) ds}{ \int \mu(dy) \log^+(1/|y-x_n|)} \xrightarrow[]{P_\mu} 1.
\end{equation}
Since \[\int \mu(dy) \log^+ (1/|y-x_n|) \leq \int \mu(dy) \frac{1}{|y-x_n|}<\infty,\] we may use Proposition \ref{p5.1} to get 
\begin{equation}\label{e5.16.1}
\int_0^t X_s(\frac{\Delta}{2} \bar{g}_{x_n}) ds=X_t(\bar{g}_{x_n})-\mu(\bar{g}_{x_n})-M_t(\bar{g}_{x_n}).
\end{equation}
 Note that by \eqref{e5.11}, we have 
 \begin{align*}
  E_{\mu}\Big( |X_t(\bar{g}_{x_n})| \Big) \leq E_{\mu}\Big[ \int X_t(dy) \frac{1}{|y-x_n|} \Big] \leq C(t) \mu(1).
 \end{align*} 
 Hence
  \[\frac{X_t(\bar{g}_{x_n})}{ \int \mu(dy) \log^+(1/|y-x_n|)} \xrightarrow[]{L^1} 0. \]
 Next we observe that 
\begin{align*}
E_{\mu}\Big[M_t^2(\bar{g}_{x_n})\Big] =& E_{\mu}\Big[\int_0^t X_s((\bar{g}_{x_n})^2) ds\Big] =\int \mu(dy) \int_0^t ds \int p_s(y-z) (\bar{g}_{x_n}(z))^2 dz\nonumber\\
\leq & \int \mu(dy) \int_0^t ds \int p_s(y-z) \frac{1}{|z-x_n|} dz\leq  (3t)^{1/2} \mu(1),
\end{align*}
the last by Lemma \ref{l4.2}. Therefore  \[\frac{M_t(\bar{g}_{x_n})}{ \int \mu(dy) \log^+(1/|y-x_n|)} \xrightarrow[]{L^2} 0. \]
Recall the bounded continuous function $\bar{f}$ defined as in \eqref{e5.31}. Then we have
 \begin{equation} \label{e5.13}
\Big|-\bar{g}_{x_n}(y) -\log^+ (1/|y-x_n|) \Big|= |\bar{f}(y-x_n)|\leq \|\bar{f}\|_\infty<\infty,
 \end{equation}
 and it follows that \[\lim_{n\to \infty} \frac{-\mu(\bar{g}_{x_n})}{\int \mu(dy) \log^+ (1/|y-x_n|) }=1+\lim_{n\to \infty} \frac{\int \bar{f}(y-x_n)\mu(dy) }{\int \mu(dy) \log^+ (1/|y-x_n|) }= 1.\] Now use \eqref{e5.16.1} to conclude
  \begin{align*}
 \frac{\int_0^t X_s(\frac{\Delta}{2} \bar{g}_{x_n}) ds}{ \int \mu(dy) \log^+(1/|y-x_n|)}=\frac{X_t(\bar{g}_{x_n})-\mu(\bar{g}_{x_n})-M_t(\bar{g}_{x_n})}{\int \mu(dy) \log^+(1/|y-x_n|)} \xrightarrow[]{P_\mu} 1,
 \end{align*}
 and so the proof is complete.

\subsubsection{Proof of Theorem \ref{t4}(b)}
Use \eqref{e5.20.1} with $\gamma=1/2$ to see that
\[
\Big|X_t(\phi_{x_n})-X_t(\phi_{x_0})\Big| \leq c_{\ref{t1}}  |x_n-x_0|^{1/2} \int \Big(\frac{1}{|y-x_0|^{3/2}}+\frac{1}{|y-x_n|^{3/2}}\Big) X_t(dy).
\]
 Similar to the derivation of \eqref{e2.2}, we use Lemma \ref{l2.1} to see that with $P_{\mu}$-probability one, there exist some $r_0(t,\omega)\in (0,1]$ and some constant $C>0$ such that for all $x\in \R^3$,\[\int X_t(dy) \frac{1}{|y-x|^{3/2}} \leq \frac{1}{(r_0)^{3/2}} X_t(1)+\int_{|y-x|<r_0} X_t(dy) \frac{1}{|y-x|^{3/2}} \leq \frac{1}{(r_0)^{3/2}} X_t(1)+C.\] 
 Then it follows that
 \[\Big|X_t(\phi_{x_n})-X_t(\phi_{x_0})\Big| \leq c_{\ref{t1}} |x_n-x_0|^{1/2} \cdot 2\Big(\frac{1}{(r_0)^{3/2}} X_t(1)+C \Big) \to 0 \text{\ as } x_n \to x_0, \] that is to say \[X_t(\phi_{x_n}) \xrightarrow[]{a.s.}  X_t(\phi_{x_0}).  \]
By \eqref{e5.161} we can see that the proof of Theorem \ref{t4}(b) is now reduced to the proof of convergence of $M_t(\phi_{x_n})$ to $M_t(\phi_{x_0})$ in probability . In fact, we will show the following $L^2$ convergence:
\begin{align*}
 &E_{\mu} \Big[\Big(M_t(\phi_{x_n})-M_t(\phi_{x_0})\Big)^2\Big] = E_{\mu} \Big[\int_0^t X_s((\phi_{x_n}-\phi_{x_0})^2) ds\Big]\\
 =&  \int \mu(dy) \int_0^t ds \int p_s(y-z)  (\phi_{x_n}(z)-\phi_{x_0}(z))^2 dz \to 0 \text{\ as } x_n \to x_0.
\end{align*}
For this I claim it suffices to show that when $x_n\to x_0$,
\begin{equation}
\int \mu(dy) \int_0^t ds \int p_s(y-z) \phi_{x_n}^2(z) dz \to \int \mu(dy) \int_0^t ds \int p_s(y-z) \phi_{x_0}^2(z) dz.\label{e5.14}
 \end{equation}
To see this first note that by Lemma \ref{l5.4}, \[\int \mu(dy) \int_0^t ds \int p_s(y-z) \phi_{x_0}^2(z) dz \leq 2 c_{\ref{t1}}^2 \int \mu(dy) \Big( (3t)^{1/2}+1+\log^{+} (1/|y-x_0|)\Big)<\infty. \] Hence the right hand side of \eqref{e5.14} is finite.
With respect to the measure  $\int \mu(dy) \int_0^t ds p_s(y-z) dz$,  by assuming \eqref{e5.14}, we use $\phi_{x_n}^2(z)  \to \phi_{x_0}^2(z)$ to get the uniform integrability of $\{\phi_{x_n}^2\}$.   Then \[(\phi_{x_n}-\phi_{x_0})^2 \leq 2 (\phi_{x_n})^2+2 (\phi_{x_0})^2 \text{\ is uniformly integrable} . \] Since $(\phi_{x_n}-\phi_{x_0})^2 \to 0$ as $n\to \infty$, by its uniform integrability we have  \[\int \mu(dy) \int_0^t ds \int p_s(y-z)  (\phi_{x_n}(z)-\phi_{x_0}(z))^2 dz \to 0.\]

In order to prove \eqref{e5.14}, we recall from \eqref{e5.81} that for $z\neq x$,
 \[\Delta \bar{g}_{x} (z)-\frac{1}{|z-x|^2}=\bar{h}(z-x), \text{\ where }  \bar{h} \in C_b(\R^3).\]  Then use Dominated Convergence Theorem to get
\[\int \mu(dy) \int_0^t ds \int p_s(y-z) \bar{h}(z-x_n) dz \to \int \mu(dy) \int_0^t ds \int p_s(y-z) \bar{h}(z-x_0) dz.\]
Therefore it suffices to show
\begin{equation}
\int \mu(dy) \int_0^t ds \int p_s(y-z) \Delta \bar{g}_{x_n} (z) dz \to \int \mu(dy) \int_0^t ds \int p_s(y-z) \Delta \bar{g}_{x_0} (z) dz. \label{e5.15}
 \end{equation}

In order to prove \eqref{e5.15}, take means in Proposition \ref{p5.1} to get
\begin{align*}
& \Big| \int \mu(dy) \int_0^t ds \int p_s(y-z) \frac{\Delta}{2} \bar{g}_{x_n}(z) dz - \int \mu(dy) \int_0^t ds \int p_s(y-z) \frac{\Delta}{2}  \bar{g}_{x_0}(z) dz\Big|\\
\leq& \Big|\mu(\bar{g}_{x_n})-\mu(\bar{g}_{x_0})\Big|+E_{\mu} \Big|X_t(\bar{g}_{x_n})-X_t(\bar{g}_{x_0})\Big|.
\end{align*}
Recall from \eqref{e5.13} that for $y\neq x$,  \[-\bar{g}_{x}(y)-\log^+ (1/|y-x|)=\bar{f}(y-x),   \text{\ where }  \bar{f} \in  C_b(\R^3).\] 
Apply Dominated Convergence Theorem to get \[ \int \bar{f}(y-x_n) \mu(dy) \to \int \bar{f}(y-x_0) \mu(dy),\] and then $\mu(\bar{g}_{x_n}) \to \mu(\bar{g}_{x_0})$ follows since we have assumed in Theorem \ref{t4}(b) that  \[\int \mu(dy) \log^+ (1/|y-x_n|) \to \int \mu(dy) \log^+ (1/|y-x_0|) .\]
For the second term, we have \[E_{\mu} \Big|X_t(\bar{g}_{x_n})-X_t(\bar{g}_{x_0})\Big| \leq \int \mu(dy) \int p_t(y-z) | \bar{g}_{x_n}(z)-\bar{g}_{x_0}(z)|  dz.\] It suffices to show \[\int \mu(dy) \int p_t(y-z) | \bar{g}_{x_n}(z)|  dz \to \int \mu(dy) \int p_t(y-z) | \bar{g}_{x_0}(z)|  dz\] by the same uniform integrability arguments above. We have
\begin{align*}
& \Big| \int p_t(y-z) | \bar{g}_{x_n}(z)|  dz - \int p_t(y-z) |\bar{g}_{x_0}(z)|  dz\Big| \\
 =& \Big| \int p_t(x_n-z) | \bar{g}_{0}(z-y)|  dz - \int p_t(x_0-z) |\bar{g}_{0}(z-y)|  dz\Big| \\
 \leq & C t^{-1/2} |x_n-x_0|  \int (p_{2t}(x_n-z)+ p_{2t}(x_0-z))\ | \bar{g}_{0}(z-y)| dz \ \text{\ (by \eqref{e4.12})} \\
  \leq & C t^{-1/2} |x_n-x_0|  \int (p_{2t}(x_n-z)+ p_{2t}(x_0-z)) \frac{1}{|z-y|} dz \ \text{\ (by \eqref{e5.2})} \\
    \leq & C t^{-1/2} |x_n-x_0| \cdot C(t) \text{\ \ (by \eqref{e5.10}).}
\end{align*}
Therefore 
\begin{align*}
& \Big| \int \mu(dy) \int p_t(y-z) | \bar{g}_{x_n}(z)|  dz - \int \mu(dy) \int p_t(y-z) |\bar{g}_{x_0}(z)|  dz\Big| \\
\leq & C t^{-1/2} |x_n-x_0| \cdot C(t) \mu(1) \to 0 \text{\ as } x_n \to x_0.
\end{align*}
Combining the above we have proved \eqref{e5.15} and the proof is now complete. $\hfill\blacksquare$

\section{Application to PDE}\label{s6.2}

Let $V^\lambda(x)$ be the solution to \eqref{e1.6}. In order to prove Theorem \ref{t6}, we need to show that the upper bound coincides with the lower bound as $x\to 0$ and we will prove the following two lemmas:

\begin{lemma}(Lower bound)\label{l6.1}
\[\liminf_{x\to 0} \frac{V^\lambda(x)- \lambda c_{\ref{t1}}/|x|}{c_{\ref{t1}}^2 \lambda^2 \log (1/|x|)} \geq -1 . \]
\end{lemma}

\begin{lemma}(Upper bound)\label{l6.2}
\[\limsup_{x\to 0} \frac{V^\lambda(x)- \lambda c_{\ref{t1}}/|x|}{c_{\ref{t1}}^2 \lambda^2 \log (1/|x|)} \leq -1 . \]
\end{lemma}
\noindent Recall from Section \ref{s1.4} that we have
 \begin{equation*}
E_{\delta_0} \Big(\exp ({-\lambda L_\infty^x}) \Big)=\exp (- V^\lambda (x)  ),
\end{equation*}
suggesting that we study the exponential moments of super-Brownian motion $X$ starting from $\delta_0$ in $d=3$. The proofs of Lemma \ref{l6.1} and \ref{l6.2} will then follow in Section \ref{s6.4} and \ref{s6.5}. 

\subsection{Exponential Moments}\label{s6.4.1}
In this section we give some exponential estimates of super-Brownian motion $X$ starting from $\delta_0$ in $\R^3$. The following lemma is from Lemma III.3.6 in Perkins (2002).
\begin{lemma}\label{l6.3}
If $f\geq 0$ is Borel measurable such that $G(f,t):=\int_0^t \sup_x P_s f(x) ds <2$, then 
\begin{equation}
E_{\delta_0}\Big[\exp\Big(X_t(f)\Big)\Big] \leq \exp \Big\{\delta_0(P_t f)(1-\frac{G(f,t)}{2})^{-1}\Big\}<\infty.
\end{equation}
\end{lemma}

\begin{corollary}\label{c6.4}
For any $\theta>0$, there exists some $t_0>0$ such that for all $0<t<t_0$
\begin{equation}
E_{\delta_0}\Big[\exp\Big(\theta \int \frac{1}{|y-x|} X_t(dy)\Big)\Big]\leq C<\infty, \ \forall x\in \R^3 \label{e6.2}
\end{equation}
 for some constant $C=C(t,\theta)>0$.
\end{corollary}
\begin{proof}
Let $f(y)=\theta/|y-x|$. Then by \eqref{e5.10}
\begin{align*}
 \delta_0(P_t f)=\int p_t(y) \frac{\theta}{|y-x|} dy\leq C(t) \theta <\infty.
\end{align*}
Next
\begin{align*}
 G(f,t)=\int_0^t  \sup_z \int p_s(z-y) \frac{\theta}{|y-x|} dy  ds= \theta  \int_0^t \int p_s(y)\frac{1}{|y|} dy ds   \leq  \theta (3t)^{1/2}.
\end{align*}
 The second equality follows from \eqref{e5.10} and the last inequality follows from Lemma \ref{l4.2}. Pick $t$ small enough such that $\theta (3t)^{1/2}<2$. Then \eqref{e6.2} holds by Lemma \ref{l6.3}.
\end{proof}

Now let's consider the exponential moment of the weighted occupation measure $\int_0^t X_s(\cdot) ds$. The following lemma will be useful.
\begin{lemma}\label{l6.5}
Let $X$ be a random variable such that \eqref{e4.4} holds formally. If $\sum_{n=1}^{\infty} a_n \theta_0^n$ converges for some $\theta_0>0$, then the following holds for $|\theta|<\theta_0$:
\begin{equation*} 
E\Big[\exp\big(\theta X\big)\Big]=\exp\Big(\sum_{n=1}^{\infty} a_n \theta^n\Big).
\end{equation*} 
\end{lemma}
\begin{proof}
We assume $|\theta|<\theta_0$ throughout the proof. Since \eqref{e4.4} holds formally and $\sum_{n=1}^{\infty} a_n \theta^n$ converges, for all $n\geq 1$ we have
\begin{align*}
E(X^n)=\frac{d^n}{d\theta^n} \Bigg(\exp\Big(\sum_{k=1}^n\ a_k \theta^k \Big)\Bigg)\Bigg|_{\theta=0}=\frac{d^n}{d\theta^n} \Bigg(\exp\Big(\sum_{k=1}^\infty\ a_k \theta^k \Big)\Bigg)\Bigg|_{\theta=0}.
\end{align*}
By Lemma \ref{l4.5}(ii), for $|\theta|<\theta_0$, \[E\Big[\exp\big(|\theta X|\big)\Big]<\infty.\] Then by Dominated Convergence Theorem,
\begin{align*}
E\Big[\exp\big(\theta X\big)\Big]=\sum_{n=1}^\infty \frac{\theta^n}{n!} E (X^n)&=\sum_{n=1}^\infty \frac{\theta^n}{n!}\frac{d^n}{d\theta^n} \Bigg(\exp\Big(\sum_{k=1}^\infty\ a_k \theta^k \Big)\Bigg)\Bigg|_{\theta=0}=\exp\Big(\sum_{k=1}^{\infty} a_k \theta^k\Big).
\end{align*}
The last equality is by the Taylor expansion of the analytic function $\exp\Big(\sum_{k=1}^{\infty} a_k \theta^k\Big)$.
\end{proof}

\begin{proposition}\label{p6.6}
For any $\theta>0$ there exists some $t_0>0$ such that for any $t<t_0$, we have \[E_{\delta_0} \Big[\exp\Big (\theta \int_0^t \int \frac{1}{|y-x|}X_s(dy) ds\Big)\Big] \leq C<\infty, \  \forall x\in \R^3, \] and \[E_{\delta_0} \Big[\exp\Big (\theta \int_0^t X_s(1) ds\Big)\Big] \leq C<\infty,\]  for some  constant $C=C(t,\theta)$. 
\end{proposition}

\begin{proof}
Recall that $f_x(y)=1/|y-x|$. Lemma \ref{l4.4} implies the following holds formally:
\begin{equation} \label{e6.3}
E_{\delta_0} \Big[\exp{ \Big(\theta \int_0^t X_s(f_x) ds- \theta \delta_0(v_1^x(t)) \Big)} \Big]=\exp{\Big(2\sum_{n=2}^{\infty} (\frac{\theta}{2})^n \delta_0\big(v_n^x(t)\big)\Big)}.
\end{equation}
We will show that the assumption in Lemma \ref{l6.5} holds for the case of \eqref{e6.3}. By using \eqref{e4.6} and \eqref{e4.92}, for all $n\geq 1$ we have
\[2 (\frac{1}{2})^n \delta_0(v_n^{x}(t)) \leq c_n \frac{1}{2^{n-1}} t^{(3n-2)/2} r^n=2 t^{-1} c_n(\frac{r}{2} t^{3/2})^n,\]
where $r=\sqrt[]{3}$, $c_1=1$ and $c_n=\sum_{k=1}^{n-1} c_k c_{n-k} , \ n\geq 2.$ Let \[F(\theta):=\sum_{n=1}^{\infty} c_n \theta^n.\]  By the definition of $(c_n)$, we have $F(\theta)-\theta=(F(\theta))^2 $ and it gives $F(\theta)=1/2-(1/4-\theta)^{1/2}$ and that $F(\theta)$ has a positive radius of convergence. So
\[2\sum_{n=2}^{\infty} (\frac{\theta}{2})^n \delta_0\big(v_n^x(t)\big) \leq  2 t^{-1} \sum_{n=1}^{\infty} c_n(\frac{r\theta}{2} t^{3/2})^n=2 t^{-1} F(\frac{r\theta}{2} t^{3/2})<\infty \] if we pick $t$ small enough. Therefore for $t$ small, \eqref{e6.3} holds by Lemma \ref{l6.5}. Hence
\begin{align*}
E_{\delta_0} \Big[\exp{ \Big(\theta \int_0^t X_s(f_x) ds \Big)} \Big]=\exp{\Big(2\sum_{n=1}^{\infty} (\frac{\theta}{2})^n \delta_0\big(v_n^x(t)\big)\Big)}\leq \exp{\Big(2 t^{-1} F(\frac{r\theta}{2} t^{3/2})\Big)}=C(t,\theta)<\infty.
\end{align*}
The proof is even easier for $\int_0^t X_s(1) ds$.
\end{proof}

\begin{corollary}\label{c6.8}
For any $\theta>0$, there exists some $t_0>0$ such that for all $t<t_0$  \[E_{\delta_0} \Big[\exp\Big (\theta \int_0^t X_s( \bar{g}_x^2) ds \Big)\Big] \leq C<\infty, \ \forall x\in \R^3\] for some constant $C=C(t_0,\theta)>0$.
\end{corollary}

\begin{proof}
Recall from \eqref{e5.2} that \[|\bar{g}_x (y)|^2 \leq \Big(\log^+ \frac{1}{|y-x|}\Big)^2 \leq \frac{1}{|y-x|}.\] Then it follows immediately from Proposition \ref{p6.6}.
\end{proof}

\noindent Before proceeding to the proof of Lemma \ref{l6.1}, we state another result:

\begin{proposition}\label{c6.7}
For $x\neq 0$ in $\R^3$, there exists some constant $C>0$ such that
\begin{align}
\frac{1}{2} \int_0^t \int  \frac{1}{|y-x|^2} X_s(dy) ds  \leq X_t(\bar{g}_x)-\delta_0(\bar{g}_x)-M_t(\bar{g}_x)+ C \int_0^t X_s(1) ds \label{e6.4}
\end{align}
holds for all $t\geq 0$ $P_{\delta_0}$-a.s..
\end{proposition}

\begin{proof}
By using Proposition \ref{p5.1} with $\mu=\delta_0$, for $x\neq 0$ we have $P_{\delta_0}$-a.s. that 
\[\int_0^t X_s(\frac{\Delta}{2} \bar{g}_x)ds=X_t(\bar{g}_x)-\delta_0(\bar{g}_x)-M_t(\bar{g}_x), \ \forall t\geq 0.\]
By \eqref{e5.16} we have \[\Big|\Delta \bar{g}_x(y) -\frac{1}{|y-x|^2} \Big|=|\bar{h}(y-x)| \leq \|\bar{h}\|_\infty <\infty,\] and then the above result follows.
\end{proof}  

\noindent $\mathbf{Throughout\ the\ rest\ of\ this\ Section}$, for simplicity we write $E$ for $E_{\delta_0}$ when there is no confusion.

\subsection{Lower Bound}\label{s6.4}
Recall that $\phi_x(y)=c_{\ref{t1}}/|y-x|$. For any $\sigma \in \R$, \[\exp\Big(-\sigma M_t(\phi_x)-\frac{1}{2} \sigma^2 \int_0^t X_s(\phi_x^2) ds\Big)\] is an $\cF_t$-supermartingale and therefore  
\begin{equation}
E \Big[\exp\Big(-\sigma M_t(\phi_x)-\frac{1}{2} \sigma^2 \int_0^t X_s(\phi_x^2) ds\Big)\Big] \leq 1. \label{e6.5}
\end{equation}
As is shown later in the proof of upper bound in Section \ref{s6.5}, the above supermartingale is indeed a martingale, but we only need to address \eqref{e6.5} for our use in this Section.
Similarly we have
\begin{equation}
E \Big[\exp\Big(-\sigma M_t(\bar{g}_x)-\frac{1}{2} \sigma^2 \int_0^t X_s(\bar{g}_x^2) ds\Big)\Big] \leq 1.\label{e6.6}
\end{equation}
Set $t>0$, which will be chosen small enough below. Then for any $p,q>1$ with $1/p+1/q=1$ we have
\begin{align}
& \exp\big(-(V^\lambda(x)-\lambda \frac{c_{\ref{t1}}}{|x|})\big)=E \Big[\exp\big(-\lambda(L_{\infty}^x-\frac{c_{\ref{t1}}}{|x|})\big)\Big]\nonumber\\
\leq & E \Big[\exp\big(-\lambda(L_{t}^x-\frac{c_{\ref{t1}}}{|x|})\big)\Big] \nonumber =  E \Big[\exp\big(-\lambda(M_t(\phi_x)-X_t(\phi_x))\big)\Big] \text{ (by \eqref{e5.161}) }\\
\leq & \Big( E \big[\exp\big(-p\lambda M_t(\phi_x)\big)\big]\Big)^{1/p} \Big(E \big[\exp\big(q \lambda X_t(\phi_x)\big)\big]\Big)^{1/q}. \label{e6.7}
\end{align}
Use Corollary \ref{c6.4} with $t>0$ chosen small enough to get
\begin{equation}
E \big[\exp\big(q \lambda X_t(\phi_x)\big)\big] \leq C(q,t,\lambda)<\infty.\label{e6.8}
\end{equation}
By using \eqref{e6.5} with $\sigma=\lambda p^2$, we have
\begin{align}
 & \Big( E \big[\exp\big(-p\lambda M_t(\phi_x)\big)\big]\Big)^{1/p} \nonumber\\
=& \Big( E \big[\exp\big(-p\lambda M_t(\phi_x)-\frac{1}{2} p^3\lambda^2  \int_0^t X_s(\phi_x^2) ds\big) \cdot \exp \big( \frac{1}{2} p^3 \lambda^2   \int_0^t X_s(\phi_x^2) ds\big)\big]\Big)^{1/p} \nonumber\\
\leq & \Big( E \big[\exp\big(-p^2 \lambda  M_t(\phi_x)-\frac{1}{2} p^4  \lambda^2  \int_0^t X_s(\phi_x^2) ds\big)\big]\Big)^{1/p^2} \cdot  \Big(E \big[\exp\big(\frac{1}{2} p^3  \lambda^2 q  \int_0^t X_s(\phi_x^2) ds\big)\big]\Big)^{1/pq} \nonumber\\
\leq&  \Big(E \big[\exp\big(\frac{1}{2} p^3  \lambda^2 q \int_0^t X_s(\phi_x^2) ds\big)\big]\Big)^{1/pq}.\label{e6.9}
\end{align}
Let $k=\frac{1}{2} p^3 \lambda^2  q$. Then
\begin{align}
 & E \big[\exp\big(\frac{1}{2} p^3  \lambda^2 q \int_0^t X_s(\phi_x^2) ds\big)\big]=E \big[\exp\big(k c_{\ref{t1}}^2\int_0^t \int \frac{1}{|y-x|^2} X_s(dy) ds\big)\big] \nonumber\\
 \leq&  E \big[\exp\big(2k c_{\ref{t1}}^2 (X_t(\bar{g}_x)-\delta_0(\bar{g}_x)-M_t(\bar{g}_x) +C\int_0^t X_s(1) ds)\big)\big] \text{\  (Proposition \ref{c6.7}) } \nonumber\\
\leq & \exp \big(2k c_{\ref{t1}}^2 \log \frac{1}{|x|}\big) \cdot  E \big[\exp\big(-2k c_{\ref{t1}}^2 M_t(\bar{g}_x)+2k c_{\ref{t1}}^2  \ C  \int_0^t X_s(1) ds\big)\big]    \nonumber\\
\leq &\exp \big(2k c_{\ref{t1}}^2 \log \frac{1}{|x|}\big) \cdot  \Big(E \big[\exp\big(- 4k c_{\ref{t1}}^2 M_t(\bar{g}_x)\big)\big]\Big)^{1/2} \Big(E \big[\exp\big( 4k c_{\ref{t1}}^2 C \int_0^t X_s(1) ds\big)\big] \Big)^{1/2}.\label{e6.10}
\end{align}
The second inequality follows from $X_t(\bar{g}_x) \leq 0$ and $-\delta_0(\bar{g}_x)= \log (1/|x|)$ for $|x|$ small. Proposition \ref{p6.6} implies that with $t>0$ chosen small, we have
\begin{equation}
\Big(E \big[\exp\big( 4k c_{\ref{t1}}^2 C \int_0^t X_s(1) ds\big)\big] \Big)^{1/2}\leq C(k,t)<\infty.\label{e6.11}
\end{equation}
Let $\theta=4k c_{\ref{t1}}^2$ and use \eqref{e6.6} with $\sigma=2\theta$ to get
\begin{align}
& \Big( E \big[\exp\big(-4k c_{\ref{t1}}^2 M_t(\bar{g}_x)\big)\big]\Big)^{1/2}\nonumber\\
=& \Big( E \big[\exp\big(-\theta M_t(\bar{g}_x)- \theta^2 \int_0^t X_s(\bar{g}_x^2) ds\big)\big] \cdot \exp \big(\theta^2 \int_0^t X_s(\bar{g}_x^2) ds\big) \big] \Big)^{1/2} \nonumber\\
\leq &  \Big(E \big[\exp\big(-2\theta M_t(\bar{g}_x)- 2\theta^2 \int_0^t X_s(\bar{g}_x^2) ds\big)\big]\Big)^{1/4} \cdot \Big(E \big[\exp\big(2\theta^2 \int_0^t X_s(\bar{g}_x^2) ds\big)\big]\Big)^{1/4} \nonumber\\
\leq & \Big(E \big[\exp\big(2\theta^2 \int_0^t X_s(\bar{g}_x^2) ds\big)\big]\Big)^{1/4}\leq C(\theta,t)<\infty.\label{e6.12}
\end{align}
The last inequality follows from Corollary \ref{c6.8} with $t>0$ chosen small. Therefore \eqref{e6.9}, \eqref{e6.10}, \eqref{e6.11} and \eqref{e6.12} imply that with $t>0$ chosen sufficiently small, we have
\begin{align}
\Big( E \big[\exp\big(-p\lambda M_t(\phi_x)\big)\big]\Big)^{1/p}\leq  C  \exp \big( p^2 \lambda^2  c_{\ref{t1}}^2  \log \frac{1}{|x|}\big).\label{e6.13}
\end{align}
In conclusion, \eqref{e6.7}, \eqref{e6.8} and \eqref{e6.13} imply that for $|x|>0$ small \[\exp\big(-(V^\lambda(x)- \lambda \frac{c_{\ref{t1}}}{|x|})\big) \leq C \exp
\big(p^2 \lambda^2  c_{\ref{t1}}^2 \log \frac{1}{|x|}\big) \] for some constant $C=C(p,t, \lambda )>0$, i.e. \[V^\lambda(x)- \lambda \frac{c_{\ref{t1}}}{|x|} \geq - p^2  \lambda^2 c_{\ref{t1}}^2 \log \frac{1}{|x|} +C.\] 
Let $x\to 0$ to conclude \[\liminf_{x\to 0} \frac{V^\lambda(x)- \lambda c_{\ref{t1}}/|x|}{c_{\ref{t1}}^2 \lambda^2  \log (1/|x|)} \geq - p^2 .\] Since $p>1$ can be chosen arbitrarily close to $1$, we get 
\begin{equation}
\liminf_{x\to 0} \frac{V^\lambda(x)- \lambda c_{\ref{t1}}/|x|}{c_{\ref{t1}}^2 \lambda^2 \log (1/|x|)} \geq -1 .\label{e6.14}
\end{equation}

\subsection{Upper Bound}\label{s6.5}
Set $t>0$, which will be chosen small enough below. Then we have
\begin{align}
& \exp\big(-(V^\lambda(x)- \lambda \frac{c_{\ref{t1}}}{|x|})\big)=E \Big[\exp\big(- \lambda (L_{\infty}^x-\frac{c_{\ref{t1}}}{|x|})\big)\Big]\nonumber\\
= & E \Big[\exp\big(- \lambda (L_{\infty}^x-L_t^x)\big)\cdot \exp\big(- \lambda (M_t(\phi_x)-X_t(\phi_x))\big)\Big] \text{\ (by \eqref{e5.161})}   \nonumber\\ 
\geq & E \Big[\exp\big(- \lambda (L_{\infty}^x-L_t^x)\big)\cdot \exp\big(- \lambda M_t(\phi_x)\big)\Big]\nonumber\\
= & E \Big[E \Big(\exp\big(- \lambda (L_{\infty}^x-L_t^x)\big)\big|\cF_t\Big) \cdot \exp\big(- \lambda M_t(\phi_x)\big)\Big]\nonumber \\
= & E \Big[\exp\big(-X_t(V^\lambda_x)\big)\cdot \exp\big(- \lambda M_t(\phi_x)\big)\Big],\label{e6.15}
\end{align}
where $V^\lambda_x (y)=V^\lambda(y-x)$. We use the Markov property and \eqref{e1.6.1} with $\mu=X_t$ for the last equality. Next,
\begin{align}
& E \Big[\exp\big(-X_t(V^\lambda_x)- \lambda M_t(\phi_x)\big)\Big]\nonumber\\
=& E \Big[\exp\big(-X_t(V^\lambda_x)+\frac{1}{2} \lambda^2  \int_0^t X_s(\phi_x^2) ds \big)\cdot \exp\big(- \lambda M_t(\phi_x)-\frac{1}{2} \lambda^2  \int_0^t X_s(\phi_x^2) ds\big)\Big]\nonumber\\
= &E \Big[\exp\big(-X_t(V^\lambda_x)+c_{\ref{t1}}^2  \lambda^2 (X_t(g_x)+\log \frac{1}{|x|} -M_t(g_x)) \big)\nonumber\\
&\ \ \ \ \ \ \    \ \ \  \times \exp\big(- \lambda M_t(\phi_x)-\frac{1}{2}  \lambda^2 \int_0^t X_s(\phi_x^2) ds\big)\Big] \text{\ \ \ \ (by Proposition \ref{p1}) }\nonumber\\
=&\widetilde{E} \Big[\exp\big(-X_t(V^\lambda_x)+c_{\ref{t1}}^2  \lambda^2(X_t(g_x)+\log \frac{1}{|x|} -M_t(g_x)) \big) \Big].\label{e6.16}
\end{align}
We use Dawson's Girsanov Theorem (see Chp. IV.1 in Perkins (2002)) to change measure from $P$ to $\widetilde{P}$ with 
\[\frac{d\widetilde{P}}{dP}\Big|_{\cF_t}=R_t:=\exp\big(- \lambda M_t(\phi_x)-\frac{1}{2} \lambda^2  \int_0^t X_s(\phi_x^2) ds\big),\] where $R_t$ is a martingale by Novikov's Theorem since 
\begin{align*}
& E \Big[\exp\big(\frac{1}{2} \lambda^2 \int_0^t X_s(\phi_x^2) ds\big)\Big] <\infty
\end{align*}
by \eqref{e6.10}, \eqref{e6.11} and \eqref{e6.12} with $x\neq 0$ and $t>0$ chosen small enough. \\

Note that $\widetilde{P}\ll P$, so every thing holds $\widetilde{P}$-a.s. as long as it holds $P$-a.s.. Therefore $x\mapsto M_t(g_x)$ and $x\mapsto X_t(g_x)$ is continuous $\widetilde{P}$-a.s.. For the $X_t(V^\lambda_x)$ term, we use \eqref{e1.3} to see that there is some $\delta>0$ such that \[V^\lambda_x(y) \leq \lambda \frac{1}{|y-x|} \text{ for  } |y-x|<\delta. \] Since $V^\lambda$ is continuous on $\R^3\backslash \{0\}$ and vanishes at infinity, there is some $C_\delta>0$ such that $V^\lambda_x(y) \leq C_\delta$ if $|y-x|>\delta$. So
\[X_t(V^\lambda_x) \leq C_\delta X_t(1)+\lambda \int \frac{1}{|y-x|} X_t(dy). \] By \eqref{e2.2}, with $P$-probability one, there exist some $r_0(t,\omega)\in (0,1]$ and some constant $C>0$ such that
\[\int \frac{1}{|y-x|} X_t(dy) \leq \frac{1}{r_0}X_t(1)+C. \] Therefore 
\begin{equation}
\frac{X_t(V^\lambda_x)}{\log (1/|x|)}  \xrightarrow[]{a.s.} 0 \text{ as } x \to 0. \label{e6.17}
\end{equation}
Combining \eqref{e6.15} and \eqref{e6.16} and moving the $\log (1/|x|)$ term to the left, we get 
\begin{align*}
 \exp\big(-(V^\lambda(x)- \lambda \frac{c_{\ref{t1}}}{|x|}+ c_{\ref{t1}}^2  \lambda^2  \log \frac{1}{|x|})\big)\geq \widetilde{E} \Big[\exp\big(-X_t(V^\lambda_x)+c_{\ref{t1}}^2  \lambda^2 X_t(g_x)- c_{\ref{t1}}^2  \lambda^2 M_t(g_x)\big) \Big].
\end{align*}
Then we have
\begin{align}
&  \liminf_{x\to 0} \exp\big(-(V^\lambda(x)- \lambda \frac{c_{\ref{t1}}}{|x|}+ c_{\ref{t1}}^2 \lambda^2  \log \frac{1}{|x|})/\log \frac{1}{|x|} \big)\nonumber \\
\geq & \liminf_{x\to 0} \Bigg(\widetilde{E} \Big[\exp\big(-X_t(V^\lambda_x)+c_{\ref{t1}}^2  \lambda^2 X_t(g_x)- c_{\ref{t1}}^2 \lambda^2  M_t(g_x) \big) \Big] \Bigg)^{1/\log \frac{1}{|x|}}\nonumber \\
\geq & \liminf_{x\to 0} \widetilde{E} \Big[\exp \Big(\big(-X_t(V^\lambda_x)+c_{\ref{t1}}^2 \lambda^2 X_t(g_x)- c_{\ref{t1}}^2  \lambda^2M_t(g_x)\big)/\log \frac{1}{|x|} \Big) \Big] \nonumber \\
\geq & \widetilde{E} \ \liminf_{x\to 0}\Big[\exp \Big(\big(-X_t(V^\lambda_x)+c_{\ref{t1}}^2  \lambda^2X_t(g_x)- c_{\ref{t1}}^2  \lambda^2 M_t(g_x) \big)/\log \frac{1}{|x|} \Big) \Big] \nonumber \\
=& \widetilde{E} \exp(0)=1.\label{e6.18}
\end{align}
The second inequality is by Jensen's inequality applied to the power $\log (1/|x|)>1$ for $|x|$ small. The third inequality is by Fatou's lemma. The last line follows from \eqref{e6.17} and the continuity of $x\mapsto X_t(g_x)$ and $x\mapsto M_t(g_x)$ (see Lemma \ref{l2.7} and Lemma \ref{l2.8}). 
In  conclusion, \eqref{e6.18} implies
\begin{equation}
\limsup_{x\to 0} \frac{V^\lambda(x)- \lambda c_{\ref{t1}}/|x|}{c_{\ref{t1}}^2 \lambda^2 \log (1/|x|)} \leq -1.\label{e6.19}
\end{equation}

%% The Appendices part is started with the command \appendix;
%% appendix sections are then done as normal sections
\appendix

\section{Proof of Lemma \ref{l3.2}}
Idea of this proof is from Theorem 1 in Chapter 2.2 of Evans (2010). 
\begin{proof}[Proof of Lemma  \ref{l3.2}]
For any fixed $s>0$, $p_{s}(y)=(2\pi s)^{-3/2} e^{-|y|^2/2s} \in C^\infty(\R^3)$  vanishes at infinity. Then we have \[\|Dp_s\|_{L^\infty(\R^3)}<\infty \text{ and }\|\Delta p_s\|_{L^\infty(\R^3)}<\infty.\] Here $Du=D_x u=(u_{x_1},u_{x_2}, u_{x_3})$ denotes the gradient of $u$ with respect to $x=(x_1,x_2,x_3).$ For any $\delta \in (0,1)$, 
\begin{align*}
 \Delta_y \int_{\R^3} p_{s}(y-z) g_x(z)dz=&  \int_{B(x,\delta)} \Delta_y p_{s}(y-z) g_x(z)dz+\int_{\R^3-B(x,\delta)} \Delta_y p_{s}(y-z) g_x(z)dz\\
=:& I_\delta+J_\delta.
\end{align*}
Now \[|I_\delta|\leq \|\Delta p_s\|_{L^\infty(\R^3)} \int_{B(x,\delta)} |g_x(z)| dz\leq C \delta^3 |\log \delta| \to 0.\]
Note that $\Delta_y p_{s}(y-z)=\Delta_z p_{s}(y-z)$. Integration by parts yields
\begin{eqnarray*}
J_\delta &=& \int_{\R^3-B(x,\delta)} \Delta_z p_{s}(y-z) g_x(z)dz\\
&=& \int_{\partial B(x,\delta)} g_x(z) \frac{\partial p_{s}}{\partial \nu}(y-z) dz-\int_{\R^3-B(x,\delta)} D_z p_{s}(y-z) D_z g_x(z)dz\\
&=:& K_\delta+L_\delta,
\end{eqnarray*}
$\nu$ denoting the inward pointing unit normal along $\partial B(x,\delta).$ So
\[|K_\delta|\leq \|Dp_s\|_{L^\infty(\R^3)} \int_{\partial B(x,\delta)} |g_x(z)|dz \leq C \delta^2 |\log \delta| \to 0.   \]
We continue by integrating by parts again in the term $L_\delta$ to find
\begin{eqnarray*}
L_\delta &=& \int_{\R^3-B(x,\delta)} p_{s}(y-z) \Delta_z g_x(z)dz-\int_{\partial B(x,\delta)} p_{s}(y-z) \frac{\partial g_x}{\partial \nu}(z) dz\\
&=:& M_\delta+N_\delta.
\end{eqnarray*}
Now $Dg_x(z)=\frac{z-x}{|z-x|^2}(z\neq x)$ and $\nu=\frac{-(z-x)}{|z-x|}=\frac{-(z-x)}{\delta}$ on $\partial B(x,\delta)$. Hence $\frac{\partial g_x}{\partial \nu}(z)=\nu \cdot Dg_x(z)=-\frac{1}{\delta} $ on $\partial B(x,\delta)$. Since $4\pi \delta^2$ is the surface area of the sphere $\partial B(x,\delta)$ in $\R^3$, we have
\[N_\delta=4\pi \delta \cdot \frac{1}{4\pi \delta^2} \int_{\partial B(x,\delta)} p_{s} (y-z)dz \to 0\cdot p_{s}(y-x) = 0 \text{ as } \delta \to 0.\]
By \eqref{e2.6}, we have $\Delta_z g_x(z)=1/|x-z|^2$ when $z \in \R^3-B(x,\delta)$. Therefore
\[M_\delta=\int_{\R^3-B(x,\delta)} p_{s}(y-z) \frac{1}{|x-z|^2}  dz.\] 
Lemma \ref{l3.3} implies  \[\int p_{s}(y-z) \frac{1}{|x-z|^2}  dz\leq C\frac{1}{|y-x|^2}<\infty \text{ for } y\neq x.\] By Dominated Convergence Theorem, we have \[M_\delta=\int p_{s}(y-z) \frac{1}{|x-z|^2} 1_{\{|z-x|\geq \delta\}} dz \to \int p_{s}(y-z) \frac{1}{|x-z|^2}  dz\]
as $\delta \to 0$. 
\end{proof}

\section{Proof of estimates in $d=3$}
(i) In this section we prove \eqref{e5.6}: there is a constant $C>0$ such that for all $y\neq x$ and $0<\varepsilon<1$,
\begin{align}\label{Ce.0}
\int p_\varepsilon(y-z) |\bar{g}_x(z) -\bar{g}_x(y)| dz \leq C|y-x|^{-1/2} \varepsilon^{1/4}.
\end{align}
Note that 
\begin{align}\label{Ce.1}
 \int p_\varepsilon(y-z) & |\bar{g}_x(z) -\bar{g}_x(y)|dz=\int p_\varepsilon(y-z) \Big|\log |z-x| \chi_{1/2} (z-x) -\log |y-x| \chi_{1/2} (y-x)\Big| dz\nonumber\\
=\int &p_\varepsilon(y-z)  \Big|\log^+ (1/ |z-x|) \chi_{1/2} (z-x) -\log^+(1/ |y-x|) \chi_{1/2} (y-x)\Big|dz\nonumber\\
\leq \int &p_\varepsilon(y-z) \log^+ (1/ |z-x|) \Big|\chi_{1/2} (z-x) - \chi_{1/2} (y-x)\Big| dz\nonumber\\
\ \ \ \  \ \ \ \ \ \ +&\int p_\varepsilon(y-z) \Big|\log^+ (1/ |z-x|) -\log^+(1/ |y-x|)\Big| \chi_{1/2} (y-x) dz :=I+J.
\end{align}
Since $\chi_{1/2}$ is a $C^\infty$ function with compact support, then
\begin{align*}
\int p_\varepsilon(y-z) \Big|\chi_{1/2} (z-x)- \chi_{1/2} (y-x)\Big|^2 dz&= E \Big(\Big|\chi_{1/2}(B_\varepsilon-(y-x))-\chi_{1/2}(y-x)\Big|^2\Big)\\
&\leq \|\nabla \chi_{1/2}\|_\infty^2 E(|B_\varepsilon|^2) \leq C\varepsilon.
\end{align*}
By Cauchy-Schwarz, for $0<\varepsilon<1$ 
\begin{align}\label{Ce.2}
I=&\int p_\varepsilon(y-z) \log^+ (1/ |z-x|) \Big|\chi_{1/2} (z-x) - \chi_{1/2} (y-x)\Big| dz\nonumber\\
\leq & \Big(\int p_\varepsilon(y-z) (\log^+ (1/ |z-x|))^2 dz\Big)^{1/2} \Big(\int p_\varepsilon(y-z) \Big|\chi_{1/2} (z-x)- \chi_{1/2} (y-x)\Big|^2 dz\Big)^{1/2}\nonumber\\
\leq &\Big(\int p_\varepsilon(y-z) \frac{1}{|z-x|} dz\Big)^{1/2} \cdot (C\varepsilon)^{1/2} \leq C|y-x|^{-1/2} \varepsilon^{1/2}\leq C|y-x|^{-1/2} \varepsilon^{1/4},
\end{align}
the third inequality by Lemma \ref{l3.3}. For $J$ in \eqref{Ce.1},  we have
\begin{align*}
J=\int p_\varepsilon(y-z)& \Big|\log^+ (1/ |z-x|) -\log^+(1/ |y-x|)\Big| \chi_{1/2} (y-x) dz\\
\leq   \int_{|z-x|<1}  &p_\varepsilon(y-z)  \Big|\log (1/ |z-x|) -\log (1/ |y-x|)\Big| \chi_{1/2} (y-x)dz\\
\ \ \ \ \ \  \ \ \ \ &+ \int_{|z-x|\geq 1} p_\varepsilon(y-z)  \log^+(1/ |y-x|) \chi_{1/2} (y-x)dz:=J_1+J_2.
\end{align*}
By \eqref{e3.7} we have \[J_1 \leq \int p_\varepsilon(y-z)  \Big|\log |z-x| -\log |y-x|\Big| dz\leq C|y-x|^{-1/2} \varepsilon^{1/4}.\]
For $J_2$ we have
\begin{align*}
J_2\leq& \log^+\frac{1}{|y-x|} \int_{|z-x|\geq 1} p_\varepsilon(y-z) dz =\log^+\frac{1}{|y-x|}  P\Big(|B_\varepsilon-(y-x)|>1\Big)\\
\leq &\log^+\frac{1}{|y-x|}  P\Big(|B_\varepsilon|>1-|y-x|\Big)\leq \frac{C\varepsilon^{1/2}}{(1-|y-x|)} \log^+\frac{1}{|y-x|} .
\end{align*}
Note that for $1/2<|y-x|<1$, \[\frac{1}{(1-|y-x|)} \log^+\frac{1}{|y-x|}=\frac{1}{(1-|y-x|)} \log (1+\frac{1-|y-x|}{|y-x|})\leq  |y-x|^{-1}\leq 2 |y-x|^{-1/2},\] and for $0<|y-x|<1/2$,\[\frac{1}{(1-|y-x|)} \log^+\frac{1}{|y-x|} \leq 2  \log^+\frac{1}{|y-x|} \leq 2  |y-x|^{-1/2}.\]
So for $0<\varepsilon<1$, we have $J_2 \leq  2C\varepsilon^{1/2} |y-x|^{-1/2}\leq C\varepsilon^{1/4} |y-x|^{-1/2}$. Therefore 
$J\leq J_1+J_2\leq C\varepsilon^{1/4} |y-x|^{-1/2}$. Combine \eqref{Ce.1} and \eqref{Ce.2} to conclude that \eqref{Ce.0} holds.\\

(ii) For any $\mu \in M_F(\R^3)$, we prove that the following are equivalent: 
\begin{enumerate}[(a)]
\item $x_0$ is a continuity point of $\int 1/|y-x| \mu(dy)$;
\item $(t_0,x_0)$ is a continuity point of $\mu q_t(x)$ for all $t_0 \geq 0$; 
\item $(t_0,x_0)$ is a continuity point of $\mu q_t(x)$ for some $t_0 \geq 0$.
\end{enumerate}

Recall that $q_t(x)=\int_0^t p_s(x) ds$ and $1/|x|=\int_0^\infty 2\pi p_s(x) ds$. Dominated Convergence Theorem implies as $x\to 0$, we have
\begin{align}\label{Ce.6}
\bar{q}_t(x):=q_t(x)-1/(2\pi |x|)=\int_t^\infty p_s(x) dy \to \int_t^\infty p_s(0) dy=C(t), \ \forall t>0.
\end{align}
Therefore $\bar{q}_t$ can be extended to be a bounded continuous function on $\R^3$ by letting $\bar{q}_t(0)=C(t)$.\\

$(a) \Rightarrow (b)$: Let $x_0$ be a continuity point of $\int 1/|y-x| \mu(dy)$. Then for any $\varepsilon>0$, there is some $\delta>0$ such that for all $|x-x_0|<\delta$, \[\Big|\int \frac{1}{2\pi |y-x|} \mu(dy)  -\int \frac{1}{2\pi |y-x_0|}  \mu(dy)\Big|<\varepsilon,\] and in particular 
\begin{align} \label{e9.6}
\int \int_0^\infty p_s(y-x) ds \mu(dy)=\int 1/(2\pi |y-x|) \mu(dy)<\infty.
\end{align}
For all $t_0\geq 0$ and $|x-x_0|<\delta$ we have
\begin{align*}
|\mu q_t(x)-\mu q_{t_0}(x_0)| \leq & |\mu q_t(x)-\mu q_{t_0}(x)|+|\mu q_{t_0}(x)-\mu q_{t_0}(x_0)|.
\end{align*}
The first term converges to $0$ if $t\to t_0$ by \eqref{e9.6} and Dominated Convergence Theorem. The second term vanishes if $t_0=0$, so we may assume $t_0>0$. Since $\bar{q}_{t_0}$ is bounded continuous for $t_0>0$, we can pick $\gamma>0$ such that $|\mu \bar{q}_{t_0}(x)-\mu \bar{q}_{t_0}(x_0)|<\varepsilon$ for all $|x-x_0|<\gamma$. Then for all $|x-x_0|<\gamma \wedge \delta$,
\begin{align*}
|\mu q_{t_0}(x)-\mu q_{t_0}(x_0)|\leq& |\mu \bar{q}_{t_0}(x)-\mu \bar{q}_{t_0}(x_0)|+\Big|\int \frac{1}{2\pi |y-x|} \mu(dy)  -\int \frac{1}{2\pi |y-x_0|}  \mu(dy)\Big|< 2 \varepsilon,
\end{align*}
and we prove that $(t_0,x_0)$ is a continuity point of $\mu q_t(x)$ for all $t_0\geq 0$.\\

$(c) \Rightarrow (a)$: If $(t_0,x_0)$ is a joint continuity point of $\mu q_t(x)$ for some $t_0>0$, since $\bar{q}_{t_0}$ is bounded continuous, it follows immediately that $x_0$ is a continuity point of $\int 1/|y-x|\mu(dy)$. If $(0,x_0)$ is a joint continuity point of $\mu q_t(x)$, then for any $\varepsilon>0$, there exists some $\delta>0$ such that for all $|x-x_0|<\delta$ and $0<t\leq \delta$, we have $\mu q_t(x)<\varepsilon$. Since $\bar{q}_\delta$ is bounded continuous, there is some $\gamma>0$ such that $|\mu \bar{q}_\delta(x)-\mu \bar{q}_\delta(x_0)|<\varepsilon$ for all $|x-x_0|<\gamma$. Then for all $|x-x_0|<\gamma \wedge \delta$, we have
\begin{align*}
\Big|\int \frac{1}{2\pi |y-x|} \mu(dy)  -\int \frac{1}{2\pi |y-x_0|}  \mu(dy)\Big|\leq |\mu \bar{q}_\delta(x)-\mu \bar{q}_\delta(x_0)|+\mu q_\delta (x)+\mu q_\delta (x_0)<3\varepsilon,
\end{align*} and we prove such an $x_0$ is a continuity point of $\int 1/|y-x|\mu(dy)$.\\

\section{Proof of estimates in $d=2$}
(i) Define 
\begin{align}\label{e9.2}
f_{\alpha}(x):=\int_{0}^{\infty} e^{-\alpha s} \frac{1}{2\pi s} e^{-\frac{|x|^2}{2s}} ds-\frac{1}{\pi} \log^+ \frac{1}{|x|} \text{\ for } x\in \R^2 \backslash\{0\}. 
\end{align}
We will prove that \[f_{\alpha}(x) \to C(\alpha) \text{\ as } x\to 0.\]  Then we can define $f_{\alpha}(0)=C(\alpha)$ to make $f_{\alpha}$ a continuous function in $\R^2$. It's clear from the definition that $f_{\alpha}(x)$ is bounded away from $x=0$. Hence $f_{\alpha}$ is a bounded continuous function on $\R^2$. 
To do this, we have
\begin{align*}
f_{\alpha}(x)=&\Big[\int_{0}^{1} \frac{1}{2\pi s} e^{-\frac{|x|^2}{2s}} ds-\frac{1}{\pi} \log^+ \frac{1}{|x|}\Big]+\int_{0}^{1} (e^{-\alpha s}-1) \frac{1}{2\pi s} e^{-\frac{|x|^2}{2s}} ds\\
+&\int_{1}^{\infty} e^{-\alpha s} \frac{1}{2\pi s} e^{-\frac{|x|^2}{2s}} ds:=I_1+I_2+I_3.
\end{align*}
Note that as $x\to 0$, 
\[I_3=\int_{1}^{\infty} e^{-\alpha s} \frac{1}{2\pi s} e^{-\frac{|x|^2}{2s}} ds \to \int_{1}^{\infty} e^{-\alpha s} \frac{1}{2\pi s} ds= C_1(\alpha)\] by Dominated Convergence Theorem.  We know that \[\Big|(e^{-\alpha s}-1) \frac{1}{2\pi s} \Big| \leq \frac{\alpha}{2\pi}, \text{\ for all } 0<s<1.\] Therefore by Dominated Convergence Theorem again, \[I_2=\int_{0}^{1}( e^{-\alpha s}-1) \frac{1}{2\pi s} e^{-\frac{|x|^2}{2s}} ds \to \int_{0}^{1}( e^{-\alpha s}-1) \frac{1}{2\pi s} ds=C_2(\alpha).\] Now we deal with $I_1$. Note that \[\int_{0}^{1} \frac{1}{2\pi s} e^{-\frac{|x|^2}{2s}} ds \overset{t=|x|^2/2s}{=} \int_{|x|^2/2}^\infty e^{-t} \frac{1}{2\pi t} dt.\] 
Then for $0<|x|<1$, 
\begin{align*}
I_1 &=\Big[\int_{|x|^2/2}^1 \frac{1}{2\pi t} dt-\frac{1}{\pi} \log \frac{1}{|x|}\Big] +\int_{|x|^2/2}^1 (e^{-t}-1) \frac{1}{2\pi t} dt+\int_{1}^\infty e^{-t} \frac{1}{2\pi t} dt\\
&:=J_1+J_2+J_3.
\end{align*}
Note that $J_3$ is integrable. For $J_2$, by similar arguments used to $I_2,$ we have \[J_2=\int_{|x|^2/2}^1 (e^{-t}-1) \frac{1}{2\pi t} dt \to \int_{0}^1 (e^{-t}-1) \frac{1}{2\pi t} dt=C \text{\ as } x\to 0.\]  Finally \[J_1=\int_{|x|^2/2}^1 \frac{1}{2\pi t} dt-\frac{1}{\pi} \log \frac{1}{|x|}=\frac{1}{2\pi} \log 2.\] To make a conclusion, \[f_{\alpha}(x) \to C(\alpha) \text{\ as } x\to 0.\]
(ii)
Recall that $q_t(x)=\int_0^t p_s(x) ds$ and we prove in (i) that 
\[I_1=q_1(x)-(1/\pi)\log^+ (1/|x|) \to C \text{\ as  } x\to 0 \] for some constant $C$. Dominated Convergence Theorem implies that for any $t>0$, $q_t(x)-q_1(x) \to q_t(0)-q_1(0)=C(t)$ as $x\to 0$ . Therefore when $x\to 0$,
\begin{align}\label{e9.5}
\tilde{q}_t(x):=q_t(x)-(1/\pi)\log^+ (1/|x|) \to C(t) , \ \forall t>0,
\end{align}
and $\tilde{q}_t$ can be extended to be a bounded continuous function on $\R^2$ by letting $\tilde{q}_t(0)=C(t)$.  Then by the similar arguments in Appendix B(ii), for any $\mu \in M_F(\R^2)$, the following are equivalent: 
\begin{enumerate}[(a)]
\item $x_0$ is a continuity point of $\int \log^+(1/|y-x|) \mu(dy)$;
\item $(t_0,x_0)$ is a continuity point of $\mu q_t(x)$ for all $t_0 \geq 0$; 
\item $(t_0,x_0)$ is a continuity point of $\mu q_t(x)$ for some $t_0 \geq 0$.
\end{enumerate}

%\leq   \int p_\varepsilon(y-z)  &\Big|\log (1/ |z-x|) -\log (1/ |y-x|)\Big| dz+  \log^+(1/ |y-x|)  \int p_\varepsilon(y-z) dz\\
% \leq  \int p_\varepsilon(y-z) & \Big|\log |z-x| -\log |y-x|\Big| dz+ C |y-x|^{-1/2} \varepsilon^{1/2} \\(ii)
%Recall that $q_t(x)=\int_0^t p_s(x) ds$ and we prove in (i) that 
% \leq C|y-x|^{-1/2} &\varepsilon^{1/4} + C |y-x|^{-1/2} \varepsilon^{1/2} \leq C|y-x|^{-1/2} \varepsilon^{1/4}.

%% \section{}
%% \label{}

%
% ---- Bibliography ----
%

{\ }

\noindent\text{Department of Mathematics}\\
\text{University of British Columbia}\\
\text{1984 Mathematics Road}\\
 \text{Vancouver, B.C.} \\
 \text{Canada V6T 1Z2}\\
\text{E-mail: jlhong@math.ubc.ca}\\


\begin{thebibliography}{11}
\expandafter\ifx\csname natexlab\endcsname\relax\def\natexlab#1{#1}\fi
\expandafter\ifx\csname url\endcsname\relax
  \def\url#1{{\tt #1}}\fi



\bibitem [AL92]{bib:al92}
R. Adler and M. Lewin (1992).
\newblock Local time and Tanaka formulae for super-Brownian motion and super stable processes.
\newblock{\em Stochastic Process. Appl.} 41 (1992) 45-67.

\bibitem [BEP91]{bib:bep91}
M. Barlow, S. Evans and E. Perkins (1991). 
\newblock Collision local times and measure-valued diffusions.
\newblock{\em Can. J. Math.} 43, 897-938.

\bibitem [Bil95]{bib:bil95}
P. Billingsley (1995).
\newblock Probability and Measure.
\newblock{\em Wiley Series in Probab. and Mathematical Stat.} (1995)

\bibitem [Bre87]{bib:bre87}
H. Brezis and L. Oswald (1987). 
\newblock Singular solutions for some semilinear elliptic equations,
\newblock{\em Archive Rational Mech. Anal.} 99 (1987), 249-259.

\bibitem[BPT86]{bib:bpt86}
H. Brezis, L. Peletier and D. Terman (1986).
\newblock A very singular solution of the heat equation with absorption,
\newblock{\em Archive Rational Mech. Anal.} 95 (1986) pp. 185-209. 

\bibitem [DIP89]{bib:dip89}
D. Dawson, I. Iscoe and E. Perkins (1989). 
\newblock Super-Brownian Motion: Path Properties and Hitting Probabilities,
\newblock{\em Probab. Th. Rel. Fields} 83, 135-205 (1989). 

\bibitem [Eva10]{bib:eva10}
L. Evans (2010).
\newblock Partial Differential Equations.
\newblock{\em American Mathematical Soc.} 2010.

\bibitem [Isc86]{bib:isc86}
I. Iscoe (1986).
\newblock A weighted occupation time for a class of measured-valued branching processes.
\newblock{\em Probability Theory and Related Fields} (1986), Volume 71, Issue 1, pp 85-116.

\bibitem [KS88]{bib:ks88}
N. Konno and T. Shiga (1988).
\newblock Stochastic partial differential equations for some measure-valued diffusions.
\newblock{\em Probability Theory and Related Fields} (1988), Volume 79, Issue 2, pp 201-225.

\bibitem [Kro93]{bib:kro93}
S. Krone (1993).
\newblock Local times for superdiffusions.
\newblock{\em Ann. Probab.} 21 (b) (1993) 1599-1623.

\bibitem [Mer06]{bib:mer06}
M. Merle (2006).
\newblock Local behavior of local times of super-Brownian motion. 
\newblock{\em Ann. I. H. Poincar\'e--PR} 42 (2006) 491-520.


\bibitem [MP17]{bib:mp17}
L. Mytnik and E. Perkins (2017).
\newblock The dimension of the boundary of super-Brownian motion, 2017. 
\newblock{\em Math ArXiv} 1711.03486.


\bibitem [Per02]{bib: per02}
E.~Perkins (2002).
\newblock {\em Dawson-Watanabe Superprocesses and Measure-valued Diffusions},
\newblock In {\em Ecole d'Et\'e de Probabilit\'es de Saint Flour 1999}, Lect.
Notes. in Math. 1781, Springer-Verlag, 2002. 

\bibitem [RY94]{bib: ry94}
D. Revuz and M. Yor (1994).
\newblock  Continuous Martingales and Brownian Motion,
\newblock{\em Springer}, Berlin, 1994.

\bibitem [Sug89]{bib: sug89}
S. Sugitani (1989). 
\newblock Some properties for the measure-valued diffusion process. 
\newblock{\em J. Math. Soc. Japan} 41, 437-462, 1989.

\bibitem [Ver81]{bib: ver 81}
L. V\'eron (1981).
\newblock Singular solutions of some nonlinear elliptic equations.
\newblock{\em Nonlinear Anal.} 5 (1981), pp. 225-242. 




	
\end{thebibliography}
\end{document}